\documentclass[12pt]{amsart}
\usepackage{tikz}
\usetikzlibrary{calc,arrows}
\usepackage{xcolor}
\usepackage{amssymb,latexsym,amsmath}
\usepackage{graphicx,mathrsfs}
\usepackage{float}
\usepackage[margin=1in, centering]{geometry}
\usepackage[hidelinks]{hyperref}
\usepackage{url}

\usepackage{amstext}

\numberwithin{equation}{section}

\renewcommand{\theequation}{\arabic{section}.\arabic{equation}}

\newtheorem{theorem}{Theorem}[section]
\newtheorem{lemma}[theorem]{Lemma}

\newtheorem{proposition}[theorem]{Proposition}
\theoremstyle{remark}
\newtheorem{remark}[theorem]{Remark}
\theoremstyle{plain}

\newtheorem{definition}[theorem]{Definition}
\newtheorem{corollary}[theorem]{Corollary}

\theoremstyle{definition}
\newtheorem{ktdefinition}[theorem]{Definition}
\theoremstyle{plain}

\newcommand{\de}{\delta}

\newcommand{\wh}{\widehat}

\newcommand{\R}{\mathbb{R}}

\newcommand{\T}{\mathbb{T}  }

\newcommand{\supp}{\textup{supp}}

\begin{document}

\title[Helical maximal estimates]{Sharp helical maximal estimates in all dimensions}

\author{Shengwen Gan} \address{Shengwen Gan\\   Department of Mathematics, Sun Yat-sen University, Guangzhou, 510275, P.R. China
}\email{shengwengan2018@gmail.com}

\author{Changkeun Oh}\address{ Changkeun Oh\\  Department of Mathematical Sciences and RIM, Seoul National University, Republic of Korea } \email{changkeun.math@gmail.com}

\author{Jaehoon Shin} \address{ Jaehoon Shin\\   Department of Mathematical Sciences, Seoul National University, Republic of Korea }\email{jaehoon651@snu.ac.kr}

\begin{abstract}
We establish sharp $L^p$ maximal estimates for averages over smooth nondegenerate
curves in $\mathbb R^n$ for $p>n$.
More precisely, we show that for any nondegenerate curve $\gamma$ in $\mathbb R^n$ and the associated maximal operator
\begin{equation} M f(x):=\sup_{t>0} \Big|\int_{-1}^1f(x-t\gamma(s))\chi(s)\,\mathrm{d}s\Big|, \end{equation}
 the estimate
\begin{equation} \| Mf\|_{L^p(\mathbb R^n)}\lesssim_{\gamma,p,n} \|f\|_{L^p(\mathbb R^n)} \end{equation}
holds for all $p>n$.
A new ingredient is a weighted Kakeya-type estimate.
\end{abstract}

\maketitle

\section{Introduction}\label{sec:introduction}

Let $n\geq2$, let $\gamma\in C^\infty([-1,1];\mathbb R^n)$, and assume that the curve is nondegenerate, that is,
\begin{equation}\label{intro:nondegeneracy}
 \inf_{s\in[-1,1]}
 \big|\det(\gamma'(s),\gamma''(s),\ldots,\gamma^{(n)}(s))\big|>0.
\end{equation}
Fix $\chi\in C_c^\infty((-1,1))$. For $f\in L^p(\mathbb R^n)$,
$1\leq p\leq\infty$, define
\begin{equation}\label{intro:averages}
 A_tf(x):=\int_{-1}^1 f(x-t\gamma(s))\chi(s)\,\mathrm{d}s,
 \qquad Mf(x):=\sup_{t>0}|A_tf(x)|.
\end{equation}

\begin{theorem}\label{thm:helical-maximal}
Let $n\geq2$ and let $\gamma$ be nondegenerate.
For $n<p\leq\infty$,
\begin{equation}
 \|Mf\|_{L^p(\mathbb R^n)}
 \lesssim \|f\|_{L^p(\mathbb R^n)}.
\end{equation}
\end{theorem}

The range $p>n$ is sharp; see
\cite{PS2007,KLO2023} for the relevant necessary conditions.
For $n=2$, the result is the celebrated circular maximal theorem of
Bourgain~\cite{Bourgain1986}.
The corresponding bound for $n=3$ was proved independently by
Ko--Lee--Oh~\cite{KLO2022} and
Beltran--Guo--Hickman--Seeger~\cite{BGHS2025}.
The case $n=4$ was established by Gan--Maldague--Oh~\cite{GMO2026}.
Theorem~\ref{thm:helical-maximal} is new for $n\geq5$ and thus settles
the helical maximal problem in every dimension.

\subsection{Previous results}

The study of maximal averages over curved sets goes back to Stein's spherical maximal
theorem~\cite{Stein1976}, which gives the sharp bounds in dimensions at least three.
Bourgain~\cite{Bourgain1986} proved the planar circular maximal theorem and, more generally,
the sharp maximal bound for smooth nondegenerate planar curves in the range $p>2$.
In an unpublished survey from the late 1980s, Christ posed the $L^p$ boundedness
problem for the helical maximal operator; see
\cite[Introduction]{PS2007}.
Determining the sharp $L^p$ range for this operator gives a
higher-dimensional analogue of Bourgain's theorem.

Pramanik--Seeger~\cite{PS2007} derived maximal estimates for nondegenerate
curves in $\mathbb R^3$ from a cone decoupling inequality.
With the cone decoupling inequality~\cite{Wolff2000}, available at the time, their argument
gave $L^p$ bounds for sufficiently large $p$ for the first time.
Combining the optimal cone decoupling theorem of
Bourgain--Demeter~\cite{BD2015} with the Pramanik--Seeger argument yields
the range $p>4$.
The sharp range $p>3$ was then established independently by Ko--Lee--Oh~\cite{KLO2022}
and Beltran--Guo--Hickman--Seeger~\cite{BGHS2025} in 2021.

The local smoothing approach, initiated by Sogge~\cite{Sogge1991}, gains
regularity by averaging also in the dilation parameter. In the present setting,
one seeks estimates of the form
\begin{equation}\label{intro:local-smoothing}
 \Big(\int_1^2 \|A_tf\|_{L_\sigma^p(\mathbb R^n)}^p\,\mathrm{d}t\Big)^{1/p}\lesssim\|f\|_{L^p(\mathbb R^n)}.
\end{equation}
For $\sigma>1/p$, such estimates yield maximal bounds by Sobolev embedding.
Mockenhaupt--Seeger--Sogge~\cite{MSS1992} used this approach to give another proof
of Bourgain's circular maximal theorem.
Ko--Lee--Oh~\cite{KLO2023} later proved sharp estimates of the form
\eqref{intro:local-smoothing} for nondegenerate curves in higher dimensions
when $p$ is large. As a consequence, they obtained maximal
bounds for $p>2n-2$ when $n\geq4$.

A natural hope was to obtain the full maximal range through sharp local
smoothing estimates.
Beltran--Hickman~\cite{BH2025}, however, constructed a new example imposing
the additional necessary condition $\sigma\leq(1/2+2/p)/n$ in
\eqref{intro:local-smoothing} for $n\geq3$.
Building on the wave-envelope approach of Guth--Wang--Zhang~\cite{GWZ2020},
which yields sharp local smoothing for nondegenerate planar curve averages,
Gan--Maldague--Oh~\cite{GMO2026} established sharp local smoothing
estimates for curve averages in all dimensions. For $2\leq p<\infty$, they
proved \eqref{intro:local-smoothing} whenever
$\sigma<\min\{1/n,(1/2+2/p)/n,2/p\}$, matching the necessary conditions
up to the endpoints. Their theorem gives the sharp maximal range $p>4$
in dimension four and the range $p>2n-4$ when $n\geq5$.

For $n\geq5$, the standard local smoothing and Sobolev-embedding argument
therefore leaves the range $n<p\leq2n-4$ untreated.
Oh--Woo~\cite{OW2026} went beyond this threshold using a weighted formulation
of the maximal problem, introduced in dimension three by
Ko--Lee--Oh~\cite{KLO2022}. For $n\geq5$ and a fixed $C_0$, there is
$\varepsilon_0>0$ such that, for every $2n-4-\varepsilon_0<p<\infty$,
the following estimate holds for some $\eta>0$.
If $E\subset\mathbb R^n\times[1,2]$ is a Katz--Tao $(2^{-k},n,C_0)$-set
(see Definition~\ref{def:katz-tao}), then
\begin{equation}\label{intro:previous-weighted}
 \|A_tf\|_{L^p(E)}\lesssim 2^{-k(1/p+\eta)}\|f\|_{L^p(\mathbb R^n)}.
\end{equation}
Here $k\geq0$ is an integer and
$\operatorname{supp}\widehat f\subset\{\xi:2^{k-1}\leq|\xi|\leq2^{k+1}\}$. 
This yields maximal bounds in the same range of $p$.
We obtain Theorem~\ref{thm:helical-maximal} by proving the weighted estimate
\eqref{intro:previous-weighted} for every $n<p<\infty$; see
Theorem~\ref{thm:weighted-helical}.

\subsection{Ingredients of the proof}\label{intro:new-ingredient}

Fourier analytic and geometric methods offer complementary approaches to
the study of geometric maximal operators; see the recent discussion of
G{\"u}nd{\"u}z--Hickman~\cite{GH2026}.\footnote{Schlag~\cite{Schlag1998} gave a geometric proof of the local circular maximal theorem.}
It was initially unclear to us which approach would be most effective for
the helical maximal problem, particularly since the standard local
smoothing argument does not by itself reach the sharp range when
$n\geq5$~\cite{BH2025}.
Our proof uses Fourier analysis together with geometric information
encoded in a weighted Kakeya-type estimate, which is itself proved by
Fourier analysis.

We use the Fourier-decay decomposition of Gan--Maldague--Oh~\cite{GMO2026}
and their multiscale $\ell^{p/2}$ estimates for sums of functions with
Fourier support in  planks.
We combine these ingredients with a refined decoupling estimate adapted to the
planks and a weighted Kakeya-type estimate, both proved here.
The weighted Kakeya-type estimate is the main new ingredient.
We work with estimates on subsets of spacetime satisfying an
$n$-dimensional Katz--Tao condition.

\subsubsection{Fourier decay}

In the proof, we are naturally led to estimating the Fourier decay of a surface measure.
The Fourier decay varies with the direction of the frequency, so the worst
uniform decay rate does not have enough information. We use the
decomposition of Gan--Maldague--Oh~\cite{GMO2026}, which characterizes the Fourier decay by anisotropic frequency planks adapted to the derivative
directions of the curve. The underlying multiscale  geometry also appears in the restricted
projection work of
Gan--Guo--Guth--Harris--Maldague--Wang~\cite{GGGHMW2024} in dimension three
and Gan--Guo--Wang~\cite{GGW2024} in higher dimensions. 

\subsubsection{Refined decoupling for the planks}

Refined decoupling strengthens Bourgain--Demeter decoupling~\cite{BD2015}.
Its proof uses the underlying decoupling theorem as a black box and
obtains an additional gain from information about wave packet overlap.
It was introduced by Guth--Iosevich--Ou--Wang~\cite{GIOW2020} in their work
on the planar Falconer distance problem.

In Proposition~\ref{c13all:rp}, we prove a version adapted to the anisotropic
planks in the Fourier-decay decomposition of
Gan--Maldague--Oh~\cite{GMO2026}.
In this proof, we use the conic $\ell^2$ decoupling estimate of
Beltran--Guo--Hickman--Seeger~\cite[Section~7.4]{BGHS2021} as a black box.
They derive this estimate from the decoupling theorem for curves of
Bourgain--Demeter--Guth~\cite{BDG2016}. Appendix~\ref{app:refined} contains the proof of Proposition~\ref{c13all:rp}.
We use this refined decoupling estimate in
Proposition~\ref{c13all:prefix-transfer}, which is applied both in the proof of the
weighted Kakeya-type estimate and in the final proof of the main theorem.

\subsubsection{The multiscale $\ell^{p/2}$ estimates}

The wave-envelope method was developed by
Guth--Wang--Zhang~\cite{GWZ2020} to prove the sharp square-function estimate
for the cone in $\mathbb R^3$.
Guth--Maldague--Wang~\cite{GMW2024} developed high--low techniques
for decoupling on the parabola.
Maldague~\cite{Maldague2024} combined and developed these methods to prove
the sharp square-function estimate for the moment curve in $\mathbb R^3$.
Gan--Maldague--Oh~\cite[Theorem~5.3 and Proposition~9.5]{GMO2026}
adapted Maldague's moment-curve methods to obtain the cone $\ell^{p/2}$
estimate and multiscale estimates for planks.
We use the versions stated in Proposition~\ref{prop:gmo-unweighted}
to prove the weighted Kakeya-type estimate described below.

\subsubsection{A weighted Kakeya-type estimate}

Our main new estimate is Proposition~\ref{c13all:positive-moments}.
It controls the contribution from regions where many boxes with different
orientations overlap, for weights satisfying the $n$-dimensional growth condition
\eqref{c13all:weight}. Compared with the bound obtained from the number
of directions alone, it gains the factor $\rho_{r-2}$, one of the
frequency-plank widths.

Let us sketch a proof of the weighted Kakeya-type estimate.
For $n\geq4$, we prove this estimate by induction on the integer exponent
$4\leq r\leq n$.
The base case uses the $L^4$ estimate of
Gan--Maldague--Oh~\cite{GMO2026}.
For $r\geq5$, assume the weighted Kakeya-type estimate at exponent $r-1$.
We apply the high--low method of
Guth--Solomon--Wang~\cite{GSW2019} to the function measuring this overlap.
Local constancy and the growth condition on the weight control the
low-frequency part. The high-frequency part is oscillatory, so the
induction hypothesis for nonnegative functions does not apply directly.
We need an estimate for oscillatory sums at exponent $r$.
We obtain this estimate by combining the induction hypothesis with refined
decoupling, as in Proposition~\ref{c13all:prefix-transfer}.
After rescaling, we divide the pieces further and group nearby directions.
The Fourier supports of the resulting sums fit planks of the same type as
before, with different widths. We can therefore apply the proposition.
The multiscale $\ell^{p/2}$ estimates described above separate the angular contributions.
Together with the low-frequency bound, this proves the estimate at
exponent $r$ and closes the induction.

\subsection{Organization of the paper}

Section~\ref{sec:reduction} derives the maximal theorem from the weighted
estimate for curve averages and reduces that estimate to an estimate for a restriction-type operator in Theorem~\ref{thm:translated-plates}.
Section~\ref{sec:inputs} collects the multiscale $\ell^{q/2}$ estimates,
a lemma for removing convolution kernels, the wave packet decomposition,
and the refined decoupling estimate.
Section~\ref{sec:model-case} illustrates the argument for the outermost
planks in dimension five.
Section~\ref{sec:transfer} states the weighted Kakeya-type estimate and does some preparation.
Section~\ref{sec:positive} proves the weighted Kakeya-type estimate
by induction on the exponent, using the high--low method.
Section~\ref{sec:completion} combines these weighted estimates with the
multiscale estimates and an $L^4$ square-function estimate to prove
Theorem~\ref{thm:translated-plates}, completing the proof of the maximal
theorem. Appendix~\ref{app:refined} proves the refined decoupling estimate
in Proposition~\ref{c13all:rp}.

\subsection{Notations}\label{sec:notations}

Fix a large enough decay exponent
$M$. Let $U:=x_U+A[-1,1]^m$, where $\det A\ne0$. Choose a nonnegative spatial weight $W_U$ such that
\begin{equation}
 W_U(x)\sim(1+|A^{-1}(x-x_U)|)^{-2M},\qquad
 \operatorname{supp}\widehat W_U\subset cA^{-T}B(0,1),
\end{equation}
where $c>0$ is fixed and small and $A^{-T}B(0,1)$ is the dual of $U$. 

Put $w_U:=W_U/\int W_U$, and choose
a nonnegative normalized weight $w'_U$ satisfying
\begin{equation}
 \int w'_U=1,\qquad
 w'_U(x)\sim |U|^{-1}(1+|A^{-1}(x-x_U)|)^{-M}.
\end{equation}
We remark that $w_U'$ is morally the same as $w_U$, but with a larger tail.

A normalized kernel adapted to $A[-1,1]^m$ satisfies
\begin{equation}
 \int K=1,\qquad
 |K(x)|\lesssim  |\det A|^{-1}(1+|A^{-1}x|)^{-L}
\end{equation}
for a sufficiently large decay exponent $L$.

For a box $T$ with center $x_T$ and $c>0$, write
$cT:=x_T+c(T-x_T)$.
For a centered box $Q$, we write $T\parallel Q$ for the translates of $Q$
in a fixed tiling, or a fixed finite union of such tilings; in particular,
$\sum_{T\parallel Q}\mathbf1_T\lesssim1$.

\subsection{Acknowledgements}
Changkeun Oh was
supported by the POSCO Science Fellowship of POSCO TJ Park Foundation, and the National Research
Foundation of Korea (NRF) grant funded by the Korea government (MSIT) RS-2024-00341891.

\subsection{Use of AI}

The authors used ChatGPT-6 for discussions during the development of the
arguments and for assistance in preparing the manuscript. The authors take
full responsibility for the mathematical claims and the final text.

\section{Reduction to weighted estimates}\label{sec:reduction}

In this section, we first derive Theorem~\ref{thm:helical-maximal} from
the weighted estimate for curve averages in Theorem~\ref{thm:weighted-helical}.
We then use the Fourier decomposition of Gan--Maldague--Oh~\cite{GMO2026} to reduce Theorem~\ref{thm:weighted-helical} to Theorem~\ref{thm:translated-plates}.

\subsection{Reduction to weighted estimates for curve averages}

For $\delta>0$ and $F\subset\mathbb R^m$, let $|F|_\delta$ be the minimum number
of balls of radius $\delta$ that cover $F$.
\begin{definition}\label{def:katz-tao}
For $\de,\alpha,C>0$, we call $E\subset\mathbb R^m$ a
Katz--Tao $(\de,\alpha,C)$-set~if
\begin{equation}\label{conj:katz-tao}
 |E\cap B(x,r)|_\de\leq C(r/\de)^\alpha
 \quad\text{for every }x\in\mathbb R^m\text{ and }r\geq \de.
\end{equation}
\end{definition}
The same definition using a cube in place of each ball gives an
equivalent condition, up to a constant depending on $m$. 

\begin{theorem}\label{thm:weighted-helical}
Let $n<p<\infty$ and $C_0\geq1$. There exists
$\eta=\eta(p,n,C_0,\gamma,\chi)>0$ such that, for every integer $k\geq0$
and every Katz--Tao $(2^{-k},n,C_0)$-set
$E\subset\mathbb R^n\times[1,2]$,
\begin{equation}\label{conj:weighted-averaging}
 \|A_tf\|_{L^p(E)}
 \lesssim 
 2^{-k(1/p+\eta)}\|f\|_{L^p(\mathbb R^n)}
\end{equation}
whenever $\operatorname{supp}\widehat f\subset
\{\xi \in \R^n :2^{k-1}\leq|\xi|\leq2^{k+1}\}$.
\end{theorem}

For $0<a<b<\infty$, define the maximal operator
\begin{equation}
 M_{[a,b]}f(x):=\sup_{a\leq t\leq b}|A_tf(x)|.
\end{equation}

\begin{proposition}\label{intro:weighted-to-maximal}
Theorem~\ref{thm:weighted-helical} implies Theorem~\ref{thm:helical-maximal}.
\end{proposition}

This proposition follows immediately from the following two lemmas.

\begin{lemma}\label{intro:local-maximal}
Assume Theorem~\ref{thm:weighted-helical}. For every $n<p<\infty$,
there exists $\sigma>0$ such that
\begin{equation}\label{intro:local-decay}
 \|M_{[1,2]}f\|_p\lesssim R^{-\sigma}\|f\|_p
\end{equation}
for every $R\geq2$ and every $f$ satisfying
$\operatorname{supp}\widehat f\subset\{\xi:R/2\leq|\xi|\leq2R\}$.
\end{lemma}

\begin{lemma}\label{intro:global-maximal}
Let $n\leq p<\infty$. Suppose that \eqref{intro:local-decay} holds
for some $\sigma>0$.
Then we have
\begin{equation}
 \|Mf\|_p\lesssim\|f\|_p
\end{equation}
 for every $f\in L^p(\mathbb R^n)$.
\end{lemma}

\begin{proof}[Proof of Lemma \ref{intro:local-maximal}]
Choose $b\in\mathcal S(\mathbb R)$ such that
\begin{equation}
 \operatorname{supp}\widehat b\subset[-c,c],
 \qquad \inf_{1\leq t\leq2}|b(t)|\geq1,
\end{equation}
for a sufficiently small  $c>0$.  Put
\begin{equation}
 H(x,t):=b(t)A_tf(x).
\end{equation}
Then,
\begin{equation}
 \widehat H(\xi,\tau)
 =\widehat f(\xi)\int_{-1}^1
       \chi(s)\widehat b(\tau+\gamma(s)\cdot\xi)\,\mathrm{d}s.
\end{equation}
Consequently, $\operatorname{supp}\widehat H\subset B(0,C R)$ in
$\mathbb R^{n+1}$.
The local constancy of $H$ at scale $R^{-1}$ implies, for every
 sufficiently large $L$,
\begin{equation}
 |H(z)|^p\lesssim
 \int |H(z')|^p R^{n+1}
 (1+R|x-x'|)^{-L}(1+R|t-t'|)^{-L}\,\mathrm{d}z'.
\end{equation}

Fix the disjoint tiling
$\mathcal Q:=\{R^{-1}(k+[0,1)^n):k\in\mathbb Z^n\}$ of $\mathbb R^n$.
Partition $[1,2]$ into $R$ intervals $I_j$ of length $R^{-1}$. For every $Q\in\mathcal Q$,
choose an index $j(Q)$ for which
\begin{equation}
 \sup_{Q\times I_{j(Q)}}|H|
 =\max_{1\leq j\leq R}\sup_{Q\times I_j}|H|.
\end{equation}
Define
\begin{equation}\label{intro:selected-set}
 E_f:=\bigcup_{Q\in\mathcal Q} Q\times I_{j(Q)}.
\end{equation}
Projection onto $\mathbb R^n$ and the disjoint tiling by $Q$ give
\begin{equation}
 \#\{Q\in\mathcal Q:(Q\times I_{j(Q)})\cap B(z,r)\ne\varnothing\}
 \leq C_n(rR)^n,\qquad z\in\mathbb R^{n+1},\quad r\geq R^{-1}.
\end{equation}
Hence $E_f$ is a Katz--Tao $(R^{-1},n,C_n)$-set.
Put
\begin{equation}
 \Psi_R(x,t):=R^{n+1}(1+R|x|)^{-L}(1+R|t|)^{-L}.
\end{equation}
If $z,z_0$ lie in the same box of side $R^{-1}$, then
\begin{equation}
 \Psi_R(z-z')\sim\Psi_R(z_0-z')
\end{equation}
for every $z'\in\mathbb R^{n+1}$.
Since $|Q|=R^{-n}$ and $|Q\times I_{j(Q)}|=R^{-(n+1)}$, it follows that
\begin{equation}\begin{aligned}
 \int_Q |M_{[1,2]}f(x)|^p\,\mathrm{d}x
 &\leq |Q|\sup_{Q\times I_{j(Q)}}|H|^p \\
 &\lesssim R\int |H(z)|^p
       (\mathbf1_{Q\times I_{j(Q)}}*\Psi_R)(z)\,\mathrm{d}z.
\end{aligned}\end{equation}
Summing over $Q$ gives
\begin{equation}\label{intro:discretization}
 \|M_{[1,2]}f\|_p^p
 \lesssim R\int_{\mathbb R^{n+1}}\Psi_R(u,v)
       \int_{E_f+(u,v)}|b(t)A_tf(x)|^p\,\mathrm{d}x\,\mathrm{d}t\,\mathrm{d}u\,\mathrm{d}v.
\end{equation}

We will apply Theorem~\ref{thm:weighted-helical}. Indeed,
Theorem~\ref{thm:weighted-helical} also gives, for some $\eta_1>0$,
\begin{equation}
 \|A_tf\|_{L^p(F)}
 \lesssim R^{-1/p-\eta_1}\|f\|_p
\end{equation}
for Katz--Tao $(R^{-1},n,C_n)$-sets
$F\subset\mathbb R^n\times[3/4,9/4]$.
Thus, for $|v|\leq1/4$,
\begin{equation}
 \int_{E_f+(u,v)}|b(t)A_tf(x)|^p\,\mathrm{d}x\,\mathrm{d}t
 \lesssim R^{-1-p\eta_1}\|f\|_p^p.
\end{equation}
For the remaining values of $(u,v)$, Minkowski's inequality gives
$\|A_tf\|_p\leq\|\chi\|_1\|f\|_p$ uniformly in $t$, and hence
\begin{equation}
 \int_{E_f+(u,v)}|b(t)A_tf(x)|^p\,\mathrm{d}x\,\mathrm{d}t
 \leq\int_{\mathbb R}|b(t)|^p\|A_tf\|_p^p\,\mathrm{d}t
 \lesssim\|f\|_p^p.
\end{equation}
Since $L>n$,
\begin{equation}
 \int\Psi_R\lesssim1,
 \qquad
 \int_{\mathbb R^n}\int_{|v|>1/4}\Psi_R(u,v)\,\mathrm{d}v\,\mathrm{d}u
 \lesssim R^{1-L}.
\end{equation}
Substituting these bounds into \eqref{intro:discretization} and choosing
$L>2+p\eta_1$ gives
\begin{equation}
 \|M_{[1,2]}f\|_p^p
 \lesssim(R^{-p\eta_1}+R^{2-L})\|f\|_p^p
 \lesssim R^{-p\eta_1}\|f\|_p^p.
\end{equation}
This proves the lemma with $\sigma=\eta_1$.
\end{proof}

We next prove Lemma \ref{intro:global-maximal}.

\begin{proof}[Proof of Lemma \ref{intro:global-maximal}]
Choose $\varphi\in C_c^\infty(\mathbb R^n)$, equal to one on
$B(0,1)$ and supported in $B(0,2)$. Define the Littlewood-Paley type operators
\begin{equation}
 \widehat{P_{\leq l}f}(\xi):=\varphi(2^{-l}\xi)\widehat f(\xi),
 \qquad P_l:=P_{\leq l}-P_{\leq l-1}.
\end{equation}
For each integer $j$,
\begin{equation}
 f=P_{\leq-j}f+\sum_{m=1}^\infty P_{m-j}f,
 \qquad
 \operatorname{supp}\widehat{P_lf}
 \subset\{2^{l-1}\leq|\xi|\leq2^{l+1}\}.
\end{equation}
Put $\rho:=\varphi^\vee$.
For $1\leq u\leq2$,
\begin{equation}
 A_{2^ju}P_{\leq-j}f=f*K_{j,u},\qquad
 K_{j,u}(y):=2^{-jn}\int
       \rho(2^{-j}y-u\gamma(s))\chi(s)\,\mathrm{d}s.
\end{equation}
Boundedness of $\gamma$ on $\operatorname{supp}\chi$ and rapid decay
of $\rho$ give, for every $L$,
\begin{equation}
 |K_{j,u}(y)|\lesssim 
 2^{-jn}(1+2^{-j}|y|)^{-L}.
\end{equation}
Let $M_{\rm HL}$ denote the Hardy--Littlewood maximal operator.
Decompose the last bound into $|y|\leq2^j$ and the annuli
$2^{j+l-1}<|y|\leq2^{j+l}$, $l\geq1$. If $L>n$, then
\begin{equation}
 \sup_{j\in\mathbb Z} M_{[2^j,2^{j+1}]}P_{\leq-j}f
 \lesssim\sum_{l=0}^\infty 2^{-l(L-n)}M_{\rm HL}f
 \lesssim M_{\rm HL}f.
\end{equation}

For $m\geq1$, set
\begin{equation}
 G_m(x):=\sup_{j\in\mathbb Z}
             M_{[2^j,2^{j+1}]}P_{m-j}f(x).
\end{equation}
For $g_{j,m}(y):=(P_{m-j}f)(2^jy)$,
\begin{equation}
 \operatorname{supp}\widehat g_{j,m}
 \subset\{\xi:2^{m-1}\leq|\xi|\leq2^{m+1}\},\qquad
 M_{[2^j,2^{j+1}]}P_{m-j}f(2^jy)=M_{[1,2]}g_{j,m}(y).
\end{equation}
Applying \eqref{intro:local-decay} and changing variables gives
\begin{equation}
 \|M_{[2^j,2^{j+1}]}P_{m-j}f\|_p
 \lesssim 2^{-m\sigma}\|P_{m-j}f\|_p.
\end{equation}
Using $p\geq2$ and the Littlewood--Paley inequality,
\begin{equation}\begin{aligned}
 \|G_m\|_p^p
 &\leq\sum_{j\in\mathbb Z}\|M_{[2^j,2^{j+1}]}P_{m-j}f\|_p^p \\
 &\lesssim 2^{-mp\sigma}\sum_{l\in\mathbb Z}\|P_lf\|_p^p \\
 &\leq2^{-mp\sigma}
       \int\Big(\sum_{l\in\mathbb Z}|P_lf|^2\Big)^{p/2}
 \lesssim 2^{-mp\sigma}\|f\|_p^p.
\end{aligned}\end{equation}
Combining the frequency decomposition and the low-frequency bound with
the triangle inequality gives
\begin{equation}
 Mf=\sup_{j\in\mathbb Z}M_{[2^j,2^{j+1}]}f
 \lesssim M_{\rm HL}f+\sum_{m=1}^\infty G_m.
\end{equation}
The $L^p$ norm is at most
\begin{equation}
 C_p\|f\|_p+C_p\sum_{m=1}^\infty2^{-m\sigma}\|f\|_p
 \lesssim \|f\|_p.
\end{equation}
This finishes the proof.
\end{proof}

\subsection{Reduction to weighted estimates for Fourier restriction-type operators}\label{sec:frame-planks}
We first express the averaging operator as a convolution in space-time.
Choose $\vartheta\in C_c^\infty((1/2,3))$ with $\vartheta=1$ on $[1,2]$, and put
\begin{equation}
 \mu:=\int_{\mathbb R}\int_{-1}^{1}
 \chi(s)\vartheta(u)\,\delta_{(u\gamma(s),u)}\,\mathrm ds\,\mathrm du.
\end{equation}
Then
\begin{equation}
 \vartheta(t)A_tf=(f\otimes\delta_0)*\mu,
 \qquad
 \mathcal F_{x,t}[\vartheta(t)A_tf](\xi,\tau)
 =\widehat f(\xi)\widehat\mu(\xi,\tau),
\end{equation}
where
\begin{equation}
 \widehat\mu(\xi,\tau)
 =\int_{\mathbb R}\int_{-1}^{1}
 e^{-iu(\gamma(s)\cdot\xi+\tau)}
 \chi(s)\vartheta(u)\,\mathrm ds\,\mathrm du.
\end{equation}
Gan--Maldague--Oh~\cite[Section~3]{GMO2026} describe the pointwise decay of
$\widehat\mu$ through an adapted plank decomposition. Following their
framework, we reduce Theorem~\ref{thm:weighted-helical} to
Theorem~\ref{thm:translated-plates}. We now introduce the associated planks
and scales.

Put $\Gamma(s):=(\gamma(s),1)$. Expansion along the last row gives
\begin{equation}
 \det(\Gamma,\Gamma',\ldots,\Gamma^{(n)})
 =(-1)^n\det(\gamma',\ldots,\gamma^{(n)})\ne0.
\end{equation}
Thus $\Gamma,\Gamma',\ldots,\Gamma^{(n)}$ are linearly independent. Apply the Gram--Schmidt process in
this order, denote the resulting orthonormal frame by
$(\mathbf v_1,\ldots,\mathbf v_{n+1})$, and put
\begin{equation}\label{red:frame}
 \mathbf e_i(s):=\mathbf v_{n+2-i}(s),\qquad 1\leq i\leq n+1.
\end{equation}
In particular,
\begin{equation}\label{red:last-vector}
 \mathbf e_{n+1}(s)=\frac{\Gamma(s)}{|\Gamma(s)|},
 \qquad
 \Gamma(s)\cdot\mathbf e_i(s)=0\quad(1\leq i\leq n).
\end{equation}
For each fixed integer $k\geq0$,
\begin{equation}
 \sup_{s\in[-1,1]}|\partial_s^k\mathbf e_i(s)|\leq C_k
 \qquad(1\leq i\leq n+1),
\end{equation}
where $C_k$ depends on the curve bounds already specified.

Let $R\geq1$. Suppose we are given numbers $0<\delta_j\leq1$, $1\leq j\leq n-1$,
such that
\begin{equation}\label{red:scale-equation}
 \prod_{j=1}^{n-1}\delta_j^{n+1-j}=R^{-1}.
\end{equation}
Define
\begin{equation}\label{red:scales}
 \delta:=\prod_{j=1}^{n-1}\delta_j,\qquad
 \rho_0:=1,\qquad
 \rho_j:=\prod_{l=1}^j\delta_l^{j+1-l}\ (1\leq j\leq n-1),
 \qquad\rho_n:=R^{-1}.
\end{equation}
For $0\leq J\leq n-1$, replacing $\delta_i$ by $1$ for $i>J$ gives
\begin{equation}
 R_J:=(\rho_n^{(J)})^{-1},\qquad
 \rho_l^{(J)}:=\prod_{i=1}^{\min(l,J)}\delta_i^{\,l+1-i}
 \quad(0\leq l\leq n).
\end{equation}
Empty products equal one. A truncated plank has widths
$\rho_0^{(J)},\ldots,\rho_n^{(J)}$.
We call $(\de_j)_{j=1}^{n-1}$ the \textbf{$\delta$-sequence} and $(\rho_j)_{j=0}^{n}$ the \textbf{$\rho$-sequence}. \eqref{red:scales} gives a correspondence between them. We call $\de=\prod_{j=1}^{n-1}\delta_j$ the \textbf{angular scale}.
Choose $\nu_j\in\{-1,0,1\}$, $1\leq j\leq n-1$, with
$\nu_j=0$ if and only if $\delta_j=1$, and set
$\nu_n:=0$ and $\nu_{n+1}:=0$.
The pair $(\vec\delta,\vec\nu)$ is called \textbf{admissible} at scale $R$. 

Fix constants $0<c<1<C$, independent of the scales. The
frequency plank at $s$ is
\begin{equation}\label{red:plates}
 p_s:=\left\{\sum_{j=0}^{n}b_j\mathbf e_{j+1}(s):
 \begin{array}{ll}
 |b_j|\leq C\rho_j,&0\leq j\leq n,\ \nu_{j+1}=0,\\
 c\rho_j\leq\nu_{j+1}b_j\leq C\rho_j,&0\leq j\leq n,\ \nu_{j+1}\ne0
 \end{array}\right\}.
\end{equation}
For a parameter $s$ at scale $\delta$, let
\begin{equation}
 \begin{gathered}
 q_s:=\left\{\sum_{j=0}^n b_j\mathbf e_{j+1}(s):
 |b_j|\leq C\rho_j,\ 0\leq j\leq n\right\},\\
 q_s^*:=\{x:|x\cdot\mathbf e_{j+1}(s)|\leq\rho_j^{-1},\ 0\leq j\leq n\}.
 \end{gathered}
\end{equation}
A translate of the dual box is
denoted by $T:=x_T+q_s^*$.
We call $s$ the \textbf{angular parameter} of $p_s$ and $(\rho_j)_j$ its \textbf{widths}.
In \cite{GMO2026}, $p_s$ is called a $(\vec\delta,\vec\nu)$-plank.
An admissible enlargement means that $c$ is replaced by a smaller
constant $c'>0$ and $C$ by a larger constant. Thus the nonzero
signed coordinates are still separated from zero. 

\begin{definition}\label{exp:adapted-bump}
Let $p:=\xi_p+A^{-T}[-1,1]^m$ be a box, where $\det A\ne0$.
A function $\varphi_p\in C_c^\infty(\mathbb R^m)$ is a smooth bump
function adapted to $p$ if
\begin{equation}
 \operatorname{supp}\varphi_p\subset \xi_p+CA^{-T}[-1,1]^m,\qquad
 |\partial_u^\alpha[\varphi_p(\xi_p+A^{-T}u)]|\leq C_\alpha
\end{equation}
for every multi-index $\alpha$, with $C$ and $C_\alpha$ independent
of $\xi_p$ and $A$.
\end{definition}

For every multi-index $\beta=(\beta_0,\ldots,\beta_n)\in\mathbb N_0^{n+1}$, write
\begin{equation}
 D_s^\beta:=\prod_{i=0}^{n}
                  (\mathbf e_{i+1}(s)\cdot\nabla_\zeta)^{\beta_i}.
\end{equation}
We use smooth bump functions $\phi_{p_s}$ uniformly adapted to $p_s$
and supported in a fixed admissible enlargement of $p_s$.
By Definition~\ref{exp:adapted-bump} and the chain rule,
\begin{equation}\label{red:adapted}
 |D_s^\beta\phi_{p_s}(\zeta)|
 \leq C_\beta\prod_{i=0}^{n}\rho_i^{-\beta_i}.
\end{equation}
The constants $C_\beta$ are uniform in $R,s$, and the admissible
parameters.

Let $\mathbb{I}_\delta$ be a $\delta$-net of $[-1,1]$, by which we
mean a maximal $\delta$-separated subset of $[-1,1]$. Thus distinct
$s,s'\in\mathbb{I}_\delta$ satisfy $|s-s'|\geq\delta$, and every
$t\in[-1,1]$ satisfies $|t-s|\leq\delta$ for some
$s\in\mathbb{I}_\delta$. 
For each $s\in \mathbb{I}_\delta$, choose $|a_s|\leq1$ and $u_s\in[1,2]$.
Put
\begin{equation}\label{red:multipliers}
 m_{R,s}(\zeta):=a_s e^{-iRu_s\Gamma(s)\cdot\zeta}\phi_{p_s}(\zeta),
 \qquad M_R:=\sum_{s\in \mathbb{I}_\delta}m_{R,s}.
\end{equation}
We have $\operatorname{supp}m_{R,s}\subset\operatorname{supp}\phi_{p_s}$.
By \eqref{red:last-vector},
\begin{equation}
 (\mathbf e_i(s)\cdot\nabla_\zeta)e^{-iRu_s\Gamma(s)\cdot\zeta}
 =
 \begin{cases}
  0,&1\leq i\leq n,\\
  -iRu_s|\Gamma(s)|e^{-iRu_s\Gamma(s)\cdot\zeta},&i=n+1.
 \end{cases}
\end{equation}
Since $R=\rho_n^{-1}$, the product rule and
\eqref{red:adapted} give
\begin{equation}
 |D_s^\beta m_{R,s}(\zeta)|
 \lesssim \prod_{i=0}^{n}\rho_i^{-\beta_i}.
\end{equation}
Define
\begin{equation}\label{red:operator}
 \mathsf{T}f(x,t):=\mathcal F^{-1}_{\xi,\tau}
       \big[(\mathcal F_x f)(\xi)M_R(\xi,\tau)\big](x,t).
\end{equation}
Here $\mathcal F_x$ acts on the $n$ spatial variables, and
$\mathcal F^{-1}_{\xi,\tau}$ on all $n+1$ frequency variables. Equivalently,
\begin{equation}\label{red:translated-form}
 \mathsf{T}f(z)=\sum_{s\in \mathbb{I}_\delta}a_s
       \mathcal F^{-1}_{\xi,\tau}
       \big[(\mathcal F_x f)(\xi)\phi_{p_s}(\xi,\tau)\big]
       (z-Ru_s\Gamma(s)).
\end{equation}
This is the operator in (9.2) of \cite{GMO2026}.

\begin{theorem}\label{thm:translated-plates}
Let $n<p<\infty$ and $C_0\geq1$. There exists $\eta>0$ such that for every admissible pair at scale $R$, every Katz--Tao
$(1,n,C_0)$-set $E\subset\mathbb R^n\times[R,2R]$,
\begin{equation}\label{red:unit-estimate}
 \|\mathsf{T}f\|_{L^p(E)}
 \lesssim\delta^{-1}R^{-1-\eta}\|f\|_{L^p(\mathbb R^n)}
\end{equation}
for every $f$ with $\operatorname{supp}\widehat f\subset
\{1/2\leq|\xi|\leq2\}$. 
\end{theorem}

Theorem~\ref{thm:weighted-helical} is now reduced to a Fourier restriction-type estimate (Theorem~\ref{thm:translated-plates}).

\begin{proposition}\label{red:plank-to-weighted}
Theorem~\ref{thm:translated-plates} implies Theorem~\ref{thm:weighted-helical}.
\end{proposition}

\begin{proof}[Proof of Proposition~\ref{red:plank-to-weighted}]
Put $R=2^k$. Let $P_s:=Rp_s$ and define
\begin{equation}
 \Phi_{P_s}(\zeta):=\phi_{p_s}(\zeta/R),
 \qquad \phi_{p_s}(\zeta)=\Phi_{P_s}(R\zeta).
\end{equation}
Then $\Phi_{P_s}$ is supported in a fixed admissible enlargement of $P_s$
and satisfies
\begin{equation}
 |D_s^\beta\Phi_{P_s}(\zeta)|
 \leq C_\beta\prod_{i=0}^n(R\rho_i)^{-\beta_i}.
\end{equation}
Put $\mathbf n_s:=u_s\Gamma(s)$.
By the Fourier reduction in
\cite[Theorem~3.22 and Section~3.4]{GMO2026}, in the weighted
form of \cite[Section~3.1]{OW2026}, it suffices to prove, for some
$\eta_0>0$,
\begin{equation}\label{red:original-scale}
 \left\|\sum_{s\in\mathbb{I}_\delta}a_s
 \mathcal F^{-1}_{\xi,\tau}\big[(\mathcal F_x f)(\xi)\Phi_{P_s}(\xi,\tau)\big]
 (\cdot-\mathbf n_s)\right\|_{L^p(E)}
 \lesssim\delta^{-1}R^{-1/p-\eta_0}\|f\|_p,
\end{equation}
uniformly over the admissible parameters and cutoff functions, for
\begin{equation}
 \begin{gathered}
 E\subset\mathbb R^n\times[1,2]
 \quad\text{a Katz--Tao }(R^{-1},n,C_0)\text{-set},\\
 \operatorname{supp}(\mathcal F_x f)\subset\{R/2\leq|\xi|\leq2R\}.
 \end{gathered}
\end{equation}

We apply a linear change of variables. Set
\begin{equation}
 g(y):=f(y/R),\qquad \widetilde E:=RE.
\end{equation}
For $Z\in\mathbb R^{n+1}$ and $r\geq1$,
\begin{equation}
 |\widetilde E\cap B(Z,r)|_1
 =|E\cap B(Z/R,r/R)|_{R^{-1}}\leq C_0r^n.
\end{equation}
Thus $\widetilde E\subset\mathbb R^n\times[R,2R]$ is a
Katz--Tao $(1,n,C_0)$-set. Also,
\begin{equation}
 (\mathcal F_y g)(\xi)=R^n(\mathcal F_x f)(R\xi),\qquad
 \|g\|_p=R^{n/p}\|f\|_p,
 \qquad\operatorname{supp}(\mathcal F_y g)\subset\{1/2\leq|\xi|\leq2\}.
\end{equation}
Changing all $n+1$ frequency variables gives
\begin{equation}\begin{aligned}
 &\mathcal F^{-1}_{\xi,\tau}\big[(\mathcal F_x f)(\xi)\Phi_{P_s}(\xi,\tau)\big]
                 (Z/R-\mathbf n_s) \\
 &\quad=(2\pi)^{-(n+1)}R^{n+1}
  \int e^{i(Z-R\mathbf n_s)\cdot(\zeta,\sigma)}
       (\mathcal F_x f)(R\zeta)\Phi_{P_s}(R\zeta,R\sigma)
       \,\mathrm{d}\zeta\,\mathrm{d}\sigma \\
 &\quad=R\mathcal F^{-1}_{\zeta,\sigma}\big[(\mathcal F_y g)(\zeta)\phi_{p_s}(\zeta,\sigma)\big]
                 (Z-R\mathbf n_s).
\end{aligned}\end{equation}
Summing over $s\in\mathbb{I}_\delta$ with coefficients $a_s$ and
changing variables, we obtain from Theorem~\ref{thm:translated-plates}
\begin{equation}\begin{aligned}
 \left\|\sum_{s\in\mathbb{I}_\delta}a_s
 \mathcal F^{-1}_{\xi,\tau}\big[(\mathcal F_x f)(\xi)\Phi_{P_s}(\xi,\tau)\big]
 (\cdot-\mathbf n_s)\right\|_{L^p(E)}
 &=R^{1-(n+1)/p}\|\mathsf{T}g\|_{L^p(\widetilde E)}\\
 &\lesssim\delta^{-1}R^{-(n+1)/p-\eta}\|g\|_p
 =\delta^{-1}R^{-1/p-\eta}\|f\|_p.
\end{aligned}\end{equation}
This proves \eqref{red:original-scale} with $\eta_0=\eta$.
The Fourier reduction therefore yields Theorem~\ref{thm:weighted-helical}.
\end{proof}

\section{Preliminaries}\label{sec:inputs}

This section collects the auxiliary estimates used in the proof of our main estimate.

\subsection{Multiscale \texorpdfstring{$\ell^{q/2}$}{ell(q/2)} estimates}\label{sec:scales-angular}
Recall the scales in \eqref{red:scale-equation} and \eqref{red:scales}.
For $0\leq j\leq n-1$, let $\mathcal J_j$ be the dyadic partition of
$[-1,1]$ into intervals of length comparable to $c\prod_{i=1}^{j}\delta_i$, where
$c>0$ is sufficiently small. These partitions are nested, and each
$I\in\mathcal J_{n-1}$ contains at most one point of $\mathbb I_\delta$.
For $0\leq j\leq n-1$ and $I\in\mathcal J_j$, the corresponding
truncated signed plank is given by \eqref{red:plates} at the center of $I$,
with parameters $(\delta_1,\ldots,\delta_j,1,\ldots,1)$ and signs
$(\nu_1,\ldots,\nu_j,0,\ldots,0)$.

\begin{proposition}\label{prop:gmo-unweighted}
Let $H_s$, $s\in\mathbb I_\delta$, be smooth functions.
Assume that $\operatorname{supp}\widehat H_s$ is contained in a fixed
admissible enlargement of $p_s$ for every $s\in\mathbb I_\delta$.
Then, for every $\varepsilon>0$,
\begin{equation}\label{c13all:unweighted-four}
 \Big\|\sum_{s\in\mathbb I_\delta}H_s\Big\|_4^4
 \lesssim R^\varepsilon
 \int\Big(\sum_{s\in\mathbb I_\delta}|H_s|^2\Big)^2.
\end{equation}
Moreover, for $1\leq j\leq n-1$ and
$4\leq q\leq(n+1-j)(n+2-j)-2$,
\begin{equation}\label{c13all:unweighted-stage}
 \int\Big(\sum_{I\in\mathcal J_{j-1}}
       \Big|\sum_{s\in\mathbb I_\delta\cap I}H_s\Big|^{q/2}\Big)^2
 \lesssim R^\varepsilon\delta_j^{-(q/2-2)}
 \int\Big(\sum_{I\in\mathcal J_j}
       \Big|\sum_{s\in\mathbb I_\delta\cap I}H_s\Big|^{q/2}\Big)^2.
\end{equation}
\end{proposition}

The support condition for the grouped functions follows from
\cite[Lemmas~3.10 and~3.13]{GMO2026}.
Estimate \eqref{c13all:unweighted-stage} follows from
\cite[Theorem~5.3]{GMO2026} for $j=1$ and
\cite[Proposition~9.5]{GMO2026} for $j\geq2$.
Iterating it at $q=4$ gives \eqref{c13all:unweighted-four}.

\subsection{Removing the convolution kernels}

In this subsection, we prove the following lemma. This is used frequently throughout the paper.

\begin{lemma}\label{c13all:kernel-removal}
Let $K_s$ be kernels adapted to $q_s^*$ in the sense of
Section~\ref{sec:notations}, with
\begin{equation}
 K_s\geq0,\qquad \int K_s=1,\qquad
 \operatorname{supp}\widehat K_s\subset Cq_s.
\end{equation}
For nonnegative functions $b_s$, $s\in\mathbb I_\delta$, one has
\begin{equation}\label{c13all:kr}
 \int\Big(\sum_{s\in\mathbb{I}_\delta} b_s*K_s\Big)^2
 \lesssim  \delta^{-\varepsilon}
       \int\Big(\sum_{s\in\mathbb{I}_\delta} b_s\Big)^2.
\end{equation}
\end{lemma}
\begin{proof}
For $\delta\leq\sigma\leq1$ dyadic, we follow the notation in \cite[Section 4.1]{GMO2026} to define
\begin{equation}
 q[\leq\sigma](s):=
 \left\{\sum_{j=0}^n a_j\mathbf e_{j+1}(s):
 |a_j|\leq\min\{\rho_j,R^{-1}\sigma^{j-n}\},\ 0\leq j\leq n\right\}.
\end{equation}
Let $\mathcal I_\sigma$ be an interval cover of $[-1,1]$ by intervals
of length $c\sigma$ and bounded overlap. For $I\in\mathcal I_\sigma$, write
$s_I$ for its center and $U_{\sigma,I}:=q[\leq\sigma](s_I)^*$.
Definition 4.1 and Lemma 4.4 of \cite{GMO2026} give, uniformly in $s,\sigma$ and the admissible tuple,
\begin{equation}
 q[\leq\sigma](s_I)\subset Cq_s\quad(s\in I),
 \qquad q_s^*\subset CU_{\sigma,I}.
\end{equation}

Put $g_s:=b_s*K_s$. Then
$
 \operatorname{supp}\widehat g_s\subset Cq_s.
$
Proposition 4.5 of \cite{GMO2026}, with box constants adjusted for the fixed support enlargements, gives
\begin{equation}\label{exp:envelope-input}
 \int\Big|\sum_{s\in\mathbb{I}_\delta}g_s\Big|^2
 \lesssim  \delta^{-\varepsilon}
 \sum_{\sigma\in2^{\mathbb Z}\cap[\delta,1]}
 \sum_{I\in\mathcal I_\sigma}\sum_{U\parallel U_{\sigma,I}}
 |U|\left(\int\sum_{s\in\mathbb{I}_\delta\cap I}|g_s|w_U\right)^2.
\end{equation}

For $s\in\mathbb{I}_\delta\cap I$, the kernel decay and
$q_s^*\subset CU_{\sigma,I}$ give
$\widetilde K_s*w_U\lesssim w'_U$, where
$\widetilde K_s(x):=K_s(-x)$. Changing the order of integration gives
\begin{equation}
 \int\sum_{s\in\mathbb{I}_\delta\cap I}g_sw_U
 =\sum_{s\in\mathbb{I}_\delta\cap I}\int b_s(\widetilde K_s*w_U)
 \lesssim\int\Big(\sum_{s\in\mathbb{I}_\delta\cap I}b_s\Big)w'_U.
\end{equation}
Since $\int w'_U=1$,
H\"older yields
\begin{equation}
 |U|\left(\int\sum_{s\in\mathbb{I}_\delta\cap I}g_sw_U\right)^2
 \lesssim\int\Big(\sum_{s\in\mathbb{I}_\delta\cap I}b_s\Big)^2|U|w'_U.
\end{equation}
Summing first over $U$ and using
$\sum_{U\parallel U_{\sigma,I}}|U|w'_U\lesssim1$ gives the
estimate with $\int(\sum_{s\in\mathbb{I}_\delta\cap I}b_s)^2$.
By bounded overlap and $b_s\geq0$, we have
\begin{equation}
 \sum_{I\in\mathcal I_\sigma}\Big(\sum_{s\in\mathbb{I}_\delta\cap I}b_s\Big)^2
 \lesssim\Big(\sum_{s\in\mathbb{I}_\delta}b_s\Big)^2.
\end{equation}
Summing over dyadic $\sigma$ in \eqref{exp:envelope-input} and absorbing the logarithmic loss
into $\delta^{-\varepsilon}$ proves \eqref{c13all:kr}.
\end{proof}

\subsection{Wave packets and refined decoupling}\label{sec:refined}

We use the following standard wave packet decomposition.

\begin{proposition}[Wave packet decomposition]\label{exp:packets}
Let $B$ be a rectangular box, and let $F$ be a function satisfying
$\operatorname{supp}\widehat F\subset B$.
One can write
\begin{equation}
 F=\sum_{T\parallel B^*}c_T\psi_T,
 \qquad |\psi_T|\lesssim W_T,\qquad
 \|\psi_T\|_q\sim|T|^{1/q}\quad(q\geq1),
\end{equation}
where
$\operatorname{supp}\widehat{\psi_T}\subset CB$. Here, the constant $C>1$ may be chosen arbitrarily close to one. Moreover, for every $1\leq r<\infty$,
\begin{equation}\label{exp:packet-coefficients}
 |c_T|^r\lesssim_r\int |F(x)|^r w_T(x)\,\mathrm dx.
\end{equation}
\end{proposition}

For the following proposition, fix $R\geq1$ and $1\leq J\leq n-1$.
Let $(\vec\delta,\vec\nu)$ be an admissible pair at scale $R$
as in Section~\ref{sec:frame-planks}, with $\delta_j=1$ for
$J<j\leq n-1$, and use the scales in \eqref{red:scales} and the
signed planks in \eqref{red:plates}.
For each $s\in\mathbb I_\delta$, fix a tiling $\mathbb T_s$ of
$\mathbb R^{n+1}$ by translates of $q_s^*$.
Let $\mathbb T$ be a subset of these planks
\begin{equation}
 \mathbb T\subset\bigcup_{s\in\mathbb I_\delta}\mathbb T_s.
\end{equation}
To record the angular parameter of the planks, for each $T\in\mathbb T$, we use $s_T$ to denote the number so that $T\in \mathbb T_{s_T}$. We also use $x_T$ to denote the center of $T$. Hence, each $T$ has the form
\[ T=x_T+q_{s_T}^*. \]

The following estimate is a refined decoupling, and we give a proof in the appendix.

\begin{proposition}[Refined decoupling]\label{c13all:rp}
For every $\varepsilon>0$ and $a_0>0$, the following  holds for every family $\mathbb T$ as above.
Let $f_T$, $T\in\mathbb T$, be smooth functions satisfying
\begin{equation}
 \begin{gathered}
 |f_T(x)|\lesssim |T|^{-1/p}\|f_T\|_p\,W_T(x),\\
 \operatorname{supp}\widehat{f_T}
 \subset \text{a fixed admissible enlargement of }p_{s_T}.
 \end{gathered}
\end{equation}
Suppose $Y$ is a finite union of translates of $[0,1]^{n+1}$,
and every unit box $Q$ in this union satisfies
\begin{equation}
 \#\{T\in\mathbb T:CQ\cap10\delta^{-a_0}T\ne\varnothing\}\leq M,
 \qquad M\geq1.
\end{equation}
Then for $2\leq p\leq(n+1-J)(n+2-J)$ we have
\begin{equation}\label{c13all:rp-estimate}
 \left\|\sum_{T\in\mathbb T}f_T\right\|_{L^p(Y)}
 \lesssim  \delta^{-\varepsilon}
 M^{\frac12-\frac1p}
 \left(\sum_{T\in\mathbb T}\|f_T\|_p^p\right)^{1/p}.
\end{equation}
\end{proposition}

Sharp $\ell^2$ decoupling for nondegenerate curves is due to
Bourgain--Demeter--Guth~\cite{BDG2016}.
Refined estimates for the paraboloid
and the cone in $\mathbb R^3$ appear in
\cite[Theorem~4.2]{GIOW2020} and
\cite[Theorem~2.8]{Harris2022}, respectively. In this proof, we use the conic $\ell^2$ decoupling estimate of
Beltran--Guo--Hickman--Seeger~\cite[Section~7.4]{BGHS2021} as a black box.
They derive this estimate from the decoupling theorem for curves of
Bourgain--Demeter--Guth~\cite{BDG2016}.

\section{A model case in dimension five}\label{sec:model-case}

We present a simple case illustrating the key ideas in the proof of
Theorem~\ref{thm:translated-plates}.
Readers interested only in the general proof may skip this section.
We take $n=5$ and consider only the outermost part, where
\begin{equation}
 \delta_1=R^{-1/5},\qquad \delta_2=\delta_3=\delta_4=1,\qquad
 \delta=\delta_1.
\end{equation}

\subsection{The outermost planks and the desired estimate}

In the normalization of Section~\ref{sec:reduction}, these scales give
\begin{equation}
 \rho_j=\delta^j\quad(0\leq j\leq5).
\end{equation}
Let $\mathbb{I}_\delta$ be a $\delta$-net of $[-1,1]$. The angular parameter is
$s\in\mathbb{I}_\delta$.
In the orthonormal frame $(\mathbf e_i(s))_{i=1}^6$ of
\eqref{red:frame}, the centered frequency plank is
\begin{equation}
 q_s=\left\{\sum_{j=0}^5\lambda_j\mathbf e_{j+1}(s):
                  |\lambda_j|\leq\delta^j\quad(0\leq j\leq5)\right\}.
\end{equation}
Its dual physical plank $q_s^*$ has size
\begin{equation}\label{sketch:physical-plank}
 1\times\delta^{-1}\times\delta^{-2}\times\delta^{-3}
   \times\delta^{-4}\times\delta^{-5}
 =1\times R^{1/5}\times R^{2/5}\times R^{3/5}\times R^{4/5}\times R.
\end{equation}
The angular parameters are $\delta$-separated. For a fixed sign
$\nu_1\in\{-1,1\}$ and fixed $0<c<C$, the signed frequency plank is
\begin{equation}
 p_s=\left\{\sum_{j=0}^5\lambda_j\mathbf e_{j+1}(s):
 c\leq\nu_1\lambda_0\leq C,\quad
 |\lambda_j|\leq C\delta^j\quad(1\leq j\leq5)\right\}.
\end{equation}

Let $F_s$ satisfy $\operatorname{supp}\widehat F_s\subset p_s$
for each $s\in\mathbb{I}_\delta$.
The $\ell^{q/2}$-function estimate in
\cite[Theorem~5.3]{GMO2026} gives
\begin{equation}\label{sketch:gmo}
 \int_{\mathbb R^6}\Big|\sum_{s\in\mathbb{I}_\delta}F_s\Big|^q
 \lesssim  R^\varepsilon\delta^{-(q/2-2)}
 \int_{\mathbb R^6}\Big(\sum_{s\in\mathbb{I}_\delta}|F_s|^{q/2}\Big)^2,
 \qquad 4\leq q\leq28.
\end{equation}
At $q=5$, this has the factor $\delta^{-1/2}=R^{1/10}$.
Let $0\leq w\leq1$ be a weight satisfying the five-dimensional growth condition
$\int_{B(x,t)}w\leq C_0t^5$ for every $x\in\mathbb R^6$ and $t\geq1$. We will obtain
\begin{equation}\label{sketch:weighted-five}
 \int_{\mathbb R^6}\Big|\sum_{s\in\mathbb{I}_\delta}F_s\Big|^5w
 \lesssim  R^\varepsilon
 \int_{\mathbb R^6}\Big(\sum_{s\in\mathbb{I}_\delta}|F_s|^{5/2}\Big)^2.
\end{equation}
The removal of $R^{1/10}$ is the gain needed in this model case.
The goal of this section is the operator estimate
\eqref{sketch:operator-five}.
As shown in Section~\ref{sec:reduction}, the original maximal problem
reduces to the estimate in Theorem~\ref{thm:translated-plates}, whose
outermost $n=5$ case is illustrated by \eqref{sketch:operator-five}.

\subsection{High--low and the weighted Kakeya-type estimate}

Apply Proposition~\ref{exp:packets} to each $F_s$ and write
\begin{equation}
 F_s=\sum_{T\parallel q_s^*}f_T= \sum_{T\parallel q_s^*}c_T\psi_T.
\end{equation}
Each wave packet is essentially supported in $T=x_T+q_s^*$ and Fourier supported in $p_s$.
By dyadic pigeonholing and homogeneity, we work with a finite family $\mathbb T$ satisfying
\begin{equation}
 1\leq|c_T|<2,\qquad \|f_T\|_\infty\asymp1
 \quad(T\in\mathbb T).
\end{equation}
Here $F_s$ still denotes the full original function, rescaled by the same factor as the coefficients.
For a fixed large $L$,
choose nonnegative $W_T$ with
\begin{equation}
 W_T(x)\asymp\left(1+\sum_{j=0}^5\delta^{2j}
 |(x-x_T)\cdot\mathbf e_{j+1}(s)|^2\right)^{-L}.
\end{equation}
We may arrange
\begin{equation}
 |f_T|\lesssim W_T,\qquad
 \|f_T\|_{20}^{20}\lesssim |T|,\qquad
 \sum_{T\parallel q_s^*}W_T\lesssim1.
\end{equation}
We choose $W_T$ so that
$\operatorname{supp}\widehat W_T\subset Cq_s$ whenever
$T\parallel q_s^*$, for a fixed $C$.
The sum $\sum_{T\in\mathbb T}W_T$ is a smoothed count of the wave packets
through a point. The \textit{weighted Kakeya-type estimate} needed below is
\begin{equation}\label{sketch:kakeya-four}
 \int \Big|\sum_{T\in\mathbb T}W_T\Big|^4w\lesssim  R^\varepsilon\int \Big|\sum_{T\in\mathbb T}W_T\Big|^2.
\end{equation}

We now outline the proof of \eqref{sketch:kakeya-four}.
Apply a smooth dyadic frequency partition at scale $\rho_2=\delta^2$
to each $W_T$. Summing the resulting pieces gives
\begin{equation}
 \sum_{T\in\mathbb T}W_T
 =\sum_{T\in\mathbb T}(W_T)_{\mathrm{low}}
 +\sum_{\substack{\ell\in2^{\mathbb Z}\\
                  \delta^2\lesssim\ell\lesssim1}}
       \sum_{T\in\mathbb T}(W_T)_\ell.
\end{equation}
Here $\ell$ runs over dyadic frequency scales
$\delta^2\lesssim\ell\lesssim1$, and, for every $T\in\mathbb T$,
\begin{equation}
 \begin{gathered}
 \operatorname{supp}\widehat{(W_T)_\ell}
 \subset\{\xi:|\xi|\sim\ell\},\qquad
 \operatorname{supp}\widehat{(W_T)_{\mathrm{low}}}
 \subset\{\xi:|\xi|\lesssim\delta^2\}.
 \end{gathered}
\end{equation}
Moreover,
\begin{equation}
 \Big\|\sum_{T\in\mathbb T}(W_T)_{\mathrm{low}}\Big\|_\infty
 \lesssim\delta^{-1}.
\end{equation}
First, the $L^\infty$ bound above gives
\begin{equation}
 \int\Big|\sum_{T\in\mathbb T}(W_T)_{\mathrm{low}}\Big|^4w
 \lesssim\delta^{-2}
 \int\Big|\sum_{T\in\mathbb T}(W_T)_{\mathrm{low}}\Big|^2w.
\end{equation}
To estimate the weighted $L^2$ integral, choose an even, nonnegative smooth kernel $\Psi_1$ satisfying
\begin{equation}
 \Psi_1(x)\sim(1+|x|)^{-10},\qquad \int\Psi_1=1.
\end{equation}
Define $\Psi_L(x):=L^{-6}\Psi_1(x/L)$.
For $L\geq1$, the growth condition in Definition~\ref{def:dimensional-weight} gives
\begin{equation}
 (w*\Psi_L)(x)
 \lesssim L^{-6}\sum_{k=0}^\infty2^{-10k}
       \int_{B(x,2^{k+1}L)}w
 \lesssim L^{-1}.
\end{equation}
In particular, taking $L=\delta^{-2}$ gives
$w*\Psi_{\delta^{-2}}\lesssim\delta^2$.
The low-frequency sum has Fourier support in $B(0,C\delta^2)$, so local constancy gives
\begin{equation}
 \Big|\sum_{T\in\mathbb T}(W_T)_{\mathrm{low}}\Big|^2
 \lesssim
 \Big|\sum_{T\in\mathbb T}(W_T)_{\mathrm{low}}\Big|^2*\Psi_{\delta^{-2}}.
\end{equation}
Multiplying by $w$, integrating, and using Fubini's theorem, we obtain
\begin{equation}\begin{aligned}
 \int\Big|\sum_{T\in\mathbb T}(W_T)_{\mathrm{low}}\Big|^2w
 &\lesssim\int\left(
 \Big|\sum_{T\in\mathbb T}(W_T)_{\mathrm{low}}\Big|^2
 *\Psi_{\delta^{-2}}\right)w\\
 &=\int\Big|\sum_{T\in\mathbb T}(W_T)_{\mathrm{low}}\Big|^2
       (w*\Psi_{\delta^{-2}}) \lesssim\delta^2
 \int\Big|\sum_{T\in\mathbb T}(W_T)_{\mathrm{low}}\Big|^2.
\end{aligned}\end{equation}
Finally, the $L^2$ multiplier bound for the low-frequency cutoff gives
\begin{equation}
 \Big\|\sum_{T\in\mathbb T}(W_T)_{\mathrm{low}}\Big\|_2
 \lesssim\Big\|\sum_{T\in\mathbb T}W_T\Big\|_2.
\end{equation}
Combining these estimates cancels the factors $\delta^{-2}$ and $\delta^2$ and yields
\begin{equation}\label{sketch:low}
 \int\Big|\sum_{T\in\mathbb T}(W_T)_{\mathrm{low}}\Big|^4w
 \lesssim\int\Big|\sum_{T\in\mathbb T}W_T\Big|^2.
\end{equation}

For the other part, we consider only the highest-frequency
case $\ell\sim1$ for simplicity. On the Fourier support of each angular contribution
$\sum_{T\parallel q_s^*}(W_T)_\ell$, write
\begin{equation}
 \begin{gathered}
 \xi=\sum_{j=0}^5\lambda_j\mathbf e_{j+1}(s),\qquad |\xi|\sim1, \qquad
 |\lambda_j|\lesssim\delta^j\quad(1\leq j\leq5).
 \end{gathered}
\end{equation}
We may assume $\delta\ll1$, which implies $|\lambda_0|\sim1$.
Splitting smoothly into $\lambda_0>0$ and $\lambda_0<0$ therefore
gives the original signed planks $p_s$. 
Using $w\leq1$ and applying the $q=4$ case of \eqref{sketch:gmo}, we obtain
\begin{equation}\begin{aligned}
 \int\Big|\sum_{T\in\mathbb T}(W_T)_\ell\Big|^4w
 &\leq\int\Big|\sum_{T\in\mathbb T}(W_T)_\ell\Big|^4\\
 &\lesssim R^\varepsilon
 \int\bigg(\sum_{s\in\mathbb{I}_\delta}
      \Big|\sum_{T\parallel q_s^*}W_T\Big|^2\bigg)^2 \lesssim R^\varepsilon
 \int\Big|\sum_{T\in\mathbb T}W_T\Big|^2.
\end{aligned}\end{equation}
These estimates illustrate the low- and highest-frequency contributions.

\subsection{Refined decoupling}

For dyadic $M\geq1$, let $Y_M$ denote the region where $\sum_{T\in\mathbb T}W_T\sim M$. We ignore the negligible tails
of the wave packets. Since every side length of $T$ is at least $1$, we may assume that $Y_M$ is a union of unit cubes, and each unit cube
$Q\subset Y_M$ meets $O(M)$ planks. For every $x\in Q\cap\{\sum_{T\in\mathbb T}W_T\sim M\}$, we have
\begin{equation}
 \#\{T\in\mathbb T:Q\cap 10\delta^{-a_0}   T\ne\varnothing\}
 \lesssim \delta^{-Ca_0} \sum_{T\in\mathbb T}W_T(x)\lesssim \delta^{-Ca_0} M.
\end{equation}
The $L^{20}$ refined decoupling estimate gives
\begin{equation}\label{sketch:refined-twenty}
 \begin{aligned}
 \int_{Y_M}\Big|\sum_{T\in\mathbb T}f_T\Big|^{20}w
 &\leq\int_{Y_M}\Big|\sum_{T\in\mathbb T}f_T\Big|^{20}\\
 &\lesssim  R^\varepsilon
 M^{20/2-1}\sum_{T\in\mathbb T}\|f_T\|_{20}^{20}
 \lesssim  R^\varepsilon M^9\int \Big|\sum_{T\in\mathbb T}W_T\Big|^2.
 \end{aligned}
\end{equation}
This is Proposition~\ref{c13all:rp} at exponent $p=20$. In the present case
$J=1$, so its permitted range extends to
$(n+1-J)(n+2-J)=5\cdot6=30$; the exponent $20$ is sufficient here.
By H\"older and \eqref{sketch:kakeya-four},
\begin{align}\label{sketch:five-numerology}
 \int_{Y_M}\Big|\sum_{T\in\mathbb T}f_T\Big|^5w
 &\lesssim\left(\int_{Y_M}\Big|\sum_{T\in\mathbb T}f_T\Big|^{20}w\right)^{1/4}
       \left(M^{-4}\int \Big|\sum_{T\in\mathbb T}W_T\Big|^4w\right)^{3/4}\notag\\
 &\lesssim  R^\varepsilon
       \left(M^9\int \Big|\sum_{T\in\mathbb T}W_T\Big|^2\right)^{1/4}
       \left(M^{-4}\int \Big|\sum_{T\in\mathbb T}W_T\Big|^2\right)^{3/4}\notag\\
 &=R^\varepsilon M^{-3/4}\int \Big|\sum_{T\in\mathbb T}W_T\Big|^2.
\end{align}
To recover the right-hand side of \eqref{sketch:weighted-five}, use
$1\leq|c_T|<2$ and apply the local coefficient estimate in
Proposition~\ref{exp:packets} with $r=5/2$.
The decay of the weights and the bounded overlap of the tiling then give
\begin{equation}
 \sum_{\substack{T\in\mathbb T\\T\parallel q_s^*}}W_T(x)
 \leq\sum_{\substack{T\in\mathbb T\\T\parallel q_s^*}}
       |c_T|^{5/2}W_T(x)
 \lesssim\bigl(|F_s|^{5/2}*w_{q_s^*}\bigr)(x).
\end{equation}
Applying Lemma~\ref{c13all:kernel-removal} to $|F_s|^{5/2}$ and the
normalized kernels $w_{q_s^*}$ therefore gives
\begin{equation}\begin{aligned}
 \int\Big|\sum_{T\in\mathbb T}W_T\Big|^2
 &\lesssim\int\Big(\sum_{s\in\mathbb I_\delta}
                    |F_s|^{5/2}*w_{q_s^*}\Big)^2 \lesssim R^\varepsilon
 \int\Big(\sum_{s\in\mathbb I_\delta}|F_s|^{5/2}\Big)^2.
\end{aligned}\end{equation}
The loss is absorbed into $R^\varepsilon$ since $\delta=R^{-1/5}$.
Apply this bound to \eqref{sketch:five-numerology} and sum over $M$.
So we obtain \eqref{sketch:weighted-five}.

\subsection{Proof of Theorem~\ref{thm:translated-plates} in the model case}

We prove Theorem~\ref{thm:translated-plates} in the outermost $n=5$ case.
Fix $p>5$. On $Y_M$, we have
$\big|\sum_{T\in\mathbb T}f_T\big|\lesssim\sum_{T\in\mathbb T}W_T\lesssim M$.
Thus \eqref{sketch:five-numerology} implies
\begin{equation}
 \begin{aligned}
 \int_{Y_M}\Big|\sum_{T\in\mathbb T}f_T\Big|^pw
 \lesssim M^{p-5}\int_{Y_M}\Big|\sum_{T\in\mathbb T}f_T\Big|^5w \lesssim R^\varepsilon M^{p-23/4}
 \int\Big|\sum_{T\in\mathbb T}W_T\Big|^2.
 \end{aligned}
\end{equation}
Since $\sum_{T\parallel q_s^*}W_T\lesssim1$, we have
$M\lesssim\delta^{-1}$.
Repeating the  argument
in the preceding subsection with $r=p/2$ gives
\begin{equation}
 \int\Big|\sum_{s\in\mathbb I_\delta}F_s\Big|^pw
 \lesssim R^\varepsilon\delta^{-\max\{p-23/4,0\}}
 \int\Big(\sum_{s\in\mathbb I_\delta}|F_s|^{p/2}\Big)^2.
\end{equation}

Write $\mathsf{T}=\sum_{s\in\mathbb{I}_\delta}\mathsf{T}_s$, where
\begin{equation}
 \mathsf{T}_sg(z)
 :=a_s[\widehat g(\xi)\phi_{p_s}(\xi,\tau)]^\vee
       (z-Ru_s\Gamma(s)),
\end{equation}
as defined in \eqref{red:translated-form}.
We assume $\operatorname{supp}\widehat g
\subset\{\xi:1/2\leq|\xi|\leq2\}$. To prove the theorem, we localize the estimate. On an $R$-ball $B\subset\mathbb R^6$, choose $\chi_B$ with 
\begin{equation}
 \begin{gathered}
 |\chi_B|\geq1\ \text{on }B,\qquad
 |\chi_B|\lesssim W_B, \qquad
 \operatorname{supp}\widehat\chi_B\subset B(0,cR^{-1}).
 \end{gathered}
\end{equation}
Here $c>0$ is sufficiently small, and $W_B,W_{B^x}$ are the weights
from Section~\ref{sec:notations} for $B$ and its spatial
projection $B^x$.
The kernel bound \eqref{c13all:kernel-bound} gives
\begin{equation}
 \sup_{s\in\mathbb{I}_\delta}\|\chi_B\mathsf{T}_sg\|_\infty
 \lesssim R^{-1}\|g\|_\infty.
\end{equation}
The $L^4$ estimate for the square function
$\big(\sum_{s\in\mathbb{I}_\delta}|\mathsf{T}_sg|^2\big)^{1/2}$ in
\cite[Proposition~9.4]{GMO2026}, recorded in
\eqref{c13all:operator-four}, gives
\begin{equation}
 \begin{aligned}
 \int\Big(\sum_{s\in\mathbb{I}_\delta}|\chi_B\mathsf{T}_sg|^2\Big)^2
 \lesssim\int\Big(\sum_{s\in\mathbb{I}_\delta}|\mathsf{T}_sg|^2\Big)^2W_B 
  \lesssim  R^{-3+\varepsilon}
                         \|g\|_{L^4(W_{B^x})}^4.
 \end{aligned}
\end{equation}
Using
\begin{equation}
 \Big(\sum_{s\in\mathbb I_\delta}|\mathsf T_sg|^{p/2}\Big)^2
 \leq\Big(\sup_{s\in\mathbb I_\delta}|\mathsf T_sg|\Big)^{p-4}
       \Big(\sum_{s\in\mathbb I_\delta}|\mathsf T_sg|^2\Big)^2,
\end{equation}
the preceding weighted estimate and the kernel bounds give
\begin{equation}
 \begin{aligned}
 \int_B|\mathsf Tg|^pw
 &\lesssim R^\varepsilon\delta^{-\max\{p-23/4,0\}}
 \int\Big(\sum_{s\in\mathbb I_\delta}|\chi_B\mathsf T_sg|^{p/2}\Big)^2\\
 &\lesssim R^\varepsilon\delta^{-\max\{p-23/4,0\}}
 (R^{-1}\|g\|_\infty)^{p-4}
 R^{-3}\|g\|_{L^4(W_{B^x})}^4.
 \end{aligned}
\end{equation}
Summing over a bounded-overlap cover of $\mathbb R^5\times[R,2R]$ yields
\begin{equation}
 \int_{\mathbb R^5\times[R,2R]}|\mathsf Tg|^pw
 \lesssim R^{-p+1+\varepsilon}\delta^{-\max\{p-23/4,0\}}
 \|g\|_\infty^{p-4}\|g\|_4^4.
\end{equation}
By a standard interpolation argument, we obtain
\begin{equation}
 \int_{\mathbb R^5\times[R,2R]}|\mathsf Tg|^p w
 \lesssim R^{-p+1+\varepsilon}
 \delta^{-\max\{p-23/4,0\}}
 \int_{\mathbb R^5}|g|^p.
\end{equation}
Since $\delta=R^{-1/5}$, taking $\varepsilon$ sufficiently small, we obtain
\begin{equation}\label{sketch:operator-five}
 \int_{\mathbb R^5\times[R,2R]}|\mathsf Tg|^p w
 \lesssim\delta^{-p}R^{-p(1+\eta_p)}
 \int_{\mathbb R^5}|g|^p
\end{equation}
for some $\eta_p>0$.
Taking $w=1_E$  proves
Theorem~\ref{thm:translated-plates} in this model case.

\section{Weighted Kakeya-type estimates}\label{sec:transfer}

In this section, we introduce the weighted Kakeya-type estimate
(Proposition~\ref{c13all:positive-moments}) and make some preparations for its proof.

\begin{ktdefinition}[$n$-dimensional weight]\label{def:dimensional-weight}
A nonnegative measurable function $w$ on $\mathbb R^{n+1}$ is called
an \emph{$n$-dimensional weight with constant $C_0$} if, for every
$x\in\mathbb R^{n+1}$ and $t\geq1$,
\begin{equation}\label{c13all:weight}
 0\leq w\leq1,\qquad
 \int_{B(x,t)}w(y)\,\mathrm{d}y\leq C_0t^n.
\end{equation}
When the value of $C_0$ is not so important, we simply call $w$ an
$n$-dimensional weight.
\end{ktdefinition}

Recall the scales in \eqref{red:scales}. Let
$\mathbb{I}_\delta$ be a set of  $\delta$-separated points in $[-1,1]$. We use the planks $q_s$ defined in Section~\ref{sec:frame-planks}.
\begin{equation}
 q_s=\left\{\sum_{j=0}^{n}u_j\mathbf e_{j+1}(s):
                   |u_j|\leq\rho_j,\ 0\leq j\leq n\right\},
 \qquad s\in\mathbb{I}_\delta.
\end{equation}
The following is a weighted Kakeya-type estimate, which will be proved in
Section~\ref{sec:positive}.

\begin{proposition}[Weighted Kakeya-type estimate]\label{c13all:positive-moments}
Let $n\geq4$, and let $r$ be an integer with $4\leq r\leq n$.
Let $g_s$, $s\in\mathbb I_\delta$, be nonnegative functions satisfying
\begin{equation}
 \sup_{s\in\mathbb{I}_\delta}\|g_s\|_\infty\leq1,\qquad
 \operatorname{supp}\widehat g_s\subset Cq_s.
\end{equation}
If $w$ is an $n$-dimensional weight, then
\begin{equation}\label{c13all:positive-moment-estimate}
 \int \Big|\sum_{s\in\mathbb I_\delta}g_s\Big|^r w\lesssim 
 R^\varepsilon\delta^{-(r-2)}\rho_{r-2}\int \Big|\sum_{s\in\mathbb I_\delta}g_s\Big|^2.
\end{equation}
\end{proposition}

Let us digest the estimate. Since the cardinality of $\mathbb I_{\delta}$ is bounded by $O(\delta^{-1})$, by taking the supremum, we obtain
\begin{equation}
 \int \Big|\sum_{s\in\mathbb I_\delta}g_s\Big|^r w\lesssim 
 R^\varepsilon\delta^{-(r-2)}\int \Big|\sum_{s\in\mathbb I_\delta}g_s\Big|^2.
\end{equation}
The weighted Kakeya-type estimate gives the extra factor $\rho_{r-2}$.
Note that
\begin{equation}
 \delta^{-(r-2)}\rho_{r-2}
 =\prod_{j=2}^{n-1}\delta_j^{-\min\{j-1,r-2\}}\geq1.
\end{equation}
For the worst-decay cone with $n=5$, we have
$(\delta_1,\delta_2,\delta_3,\delta_4)=(\delta,1,1,1)$ and
$\delta^{-2}\rho_2=1$; thus the case $r=4$ recovers
\eqref{sketch:kakeya-four}.
\medskip

The proof of Proposition \ref{c13all:positive-moments} relies on an inductive argument. 
The next proposition uses only one exponent of the preceding estimate,
so it can be applied within the induction in Section~\ref{sec:positive}. We also recall the signed frequency planks in \eqref{red:plates}:
\begin{equation}
 p_s=\left\{\sum_{j=0}^{n}u_j\mathbf e_{j+1}(s):
 \begin{array}{ll}
 |u_j|\leq C\rho_j,&0\leq j\leq n,\ \nu_{j+1}=0,\\
 c\rho_j\leq\nu_{j+1}u_j\leq C\rho_j,&0\leq j\leq n,\ \nu_{j+1}\ne0
 \end{array}\right\},\qquad s\in\mathbb I_\delta.
\end{equation}
Thus $p_s\subset Cq_s$.

\begin{proposition}\label{c13all:prefix-transfer}
Let $r$ be an integer with $4\leq r\leq n$ and let $r<p\leq r+2$.
Assume that the conclusion of
Proposition~\ref{c13all:positive-moments} holds at exponent $r$. Fix an admissible pair
$(\vec\delta,\vec\nu)$ at scale $R$, with $\delta_j=1$ for $j\geq\lfloor r/2\rfloor$.
Let $w$ be an $n$-dimensional weight.
Then 
\begin{equation}\label{c13all:prefix-estimate-local}
 \int_{B_R}\Big|\sum_{s\in\mathbb{I}_\delta}F_s\Big|^p w
 \lesssim  R^\varepsilon\prod_{i=1}^{\lfloor r/2\rfloor-1}\delta_i^{-(i-1)}
 \int\Big(\sum_{s\in\mathbb{I}_\delta}|F_s|^{p/2}\Big)^2
 W_{B_R}
\end{equation}
 for functions $F_s$ with
$\operatorname{supp}\widehat F_s\subset p_s$ for every
$s\in\mathbb I_\delta$, where $p_s$ corresponds to this pair.
\end{proposition}

Combining Propositions~\ref{c13all:positive-moments} and~\ref{c13all:prefix-transfer}, we obtain the following corollary. In Section~\ref{sec:completion}, we use only Corollary~\ref{cor:weighted-oscillatory-sums}.

\begin{corollary}\label{cor:weighted-oscillatory-sums}
Let $n\geq4$ and $n<p\leq n+2$. Fix an admissible pair
$(\vec\delta,\vec\nu)$ at scale $R$, with $\delta_j=1$ for
$\lfloor n/2\rfloor\leq j\leq n-1$.
Let $w$ be an $n$-dimensional weight.
Then
\begin{equation}
 \int_{\mathbb R^{n+1}}
 \Big|\sum_{s\in\mathbb I_\delta}F_s\Big|^p w
 \lesssim R^\varepsilon
 \prod_{j=1}^{\lfloor n/2\rfloor-1}\delta_j^{-(j-1)}
 \int_{\mathbb R^{n+1}}
 \Big(\sum_{s\in\mathbb I_\delta}|F_s|^{p/2}\Big)^2,
\end{equation}
for functions $F_s$ such that
$\operatorname{supp}\widehat F_s$ is contained in a fixed admissible
enlargement of $p_s$ for every $s\in\mathbb I_\delta$.
\end{corollary}

\subsection{Proof of Proposition~\ref{c13all:prefix-transfer}}\label{sec:proof-prefix-transfer}
For simplicity, we set $J:=\lfloor r/2\rfloor-1$ and $q:=\frac{(2r-2)p}{2r-p}$. Then
$4<q\leq(n+1-J)(n+2-J)$, and $p=\frac{2rq}{q+2r-2}$.
Also, \begin{equation}\delta^{-(r-2)}\rho_{r-2}=\prod_{i=1}^J\delta_i^{-(i-1)}\geq1,\end{equation}
since $\delta_j=1$ for $j>J$.

\medskip
\noindent\textbf{Step 1. Estimate without localization on the right-hand side.}
Fix a ball $B_R$. We first prove
\begin{equation}
 \int_{B_R}\Big|\sum_{s\in\mathbb{I}_\delta}F_s\Big|^p w
 \lesssim R^\varepsilon
 \prod_{i=1}^J\delta_i^{-(i-1)}
 \int_{\mathbb R^{n+1}}
 \Big(\sum_{s\in\mathbb{I}_\delta}|F_s|^{p/2}\Big)^2.
\end{equation}
Apply Proposition~\ref{exp:packets} to write
\begin{equation}
 F_s=\sum_{T\parallel q_s^*}f_T:=\sum_{T\parallel q_s^*}c_T\psi_T.
\end{equation}
By dyadic pigeonholing and homogeneity, we may work with a class
$\mathbb T$ satisfying $1\leq|c_T|<2$, retaining the same notation
for the normalized functions. 
For this  family, we claim that
\begin{align}
 \int_{\sum_{T\in\mathbb T}W_T\geq1}
 \Big|\sum_{s\in\mathbb{I}_\delta}F_s\Big|^p w
 &\lesssim R^\varepsilon\prod_{i=1}^J\delta_i^{-(i-1)}
 \int\Big|\sum_{T\in\mathbb T}W_T\Big|^2,
 \label{c13all:large-square-target}\\
 \int_{\sum_{T\in\mathbb T}W_T<1}
 \Big|\sum_{s\in\mathbb{I}_\delta}F_s\Big|^p w
 &\lesssim\int\Big|\sum_{T\in\mathbb T}W_T\Big|^2.
 \label{c13all:small-square-target}
\end{align}
Adding \eqref{c13all:large-square-target} and \eqref{c13all:small-square-target} gives
\begin{equation}\label{exp:equal-amplitudes}
 \int\Big|\sum_{s\in\mathbb{I}_\delta}F_s\Big|^pw
 \lesssim R^\varepsilon\prod_{i=1}^J\delta_i^{-(i-1)}
 \int\Big|\sum_{T\in\mathbb T}W_T\Big|^2.
\end{equation}
Since $|c_T| \sim 1$, one may see that
\begin{equation}
 \int_{B_R}\Big|\sum_{s\in\mathbb{I}_\delta}F_s\Big|^pw
 \lesssim R^\varepsilon\prod_{i=1}^J\delta_i^{-(i-1)}
 \int\Big|\sum_{T\in\mathbb T}|c_T|^{p/2}W_T\Big|^2.
\end{equation}
By \eqref{exp:packet-coefficients}, bounded overlap, and
Lemma~\ref{c13all:kernel-removal}, using $\delta\geq R^{-1/2}$ from
\eqref{red:scale-equation}--\eqref{red:scales},
\begin{equation}\label{c13all:sampling-energy}
 \begin{aligned}
 \int\Big|\sum_{T\in\mathbb T}|c_T|^{p/2}W_T\Big|^2
 &\lesssim\int\Big(\sum_{s\in\mathbb I_\delta}
           (|F_s|^{p/2}*w_{q_s^*})\Big)^2\lesssim R^\varepsilon
 \int\Big(\sum_{s\in\mathbb I_\delta}|F_s|^{p/2}\Big)^2.
 \end{aligned}
\end{equation}
Thus \eqref{c13all:large-square-target} and \eqref{c13all:small-square-target}
imply the desired estimate on $B_R$.
We now prove them.

\subsubsection{Proof of \eqref{c13all:large-square-target}}
We claim the following two estimates:
\begin{equation}\label{c13all:transfer-positive}
 \int\Big|\sum_{T\in\mathbb T}W_T\Big|^r w
 \lesssim R^\varepsilon\prod_{i=1}^J\delta_i^{-(i-1)}
 \int\Big|\sum_{T\in\mathbb T}W_T\Big|^2.
\end{equation}
\begin{equation}\label{c13all:transfer-refined}
 \int_{\sum_{T\in\mathbb T}W_T\geq1}
 \Big|\sum_{s\in\mathbb I_\delta}F_s\Big|^q
 \Big|\sum_{T\in\mathbb T}W_T\Big|^{-(q/2-1)}w
 \lesssim R^\varepsilon\int\Big|\sum_{T\in\mathbb T}W_T\Big|^2.
\end{equation}
Set $\theta:=\frac{2r}{q+2r-2}$, so that $q\theta=p$ and
$(q/2-1)\theta=r(1-\theta)$.
Assuming \eqref{c13all:transfer-positive} and
\eqref{c13all:transfer-refined}, H\"older's inequality yields
\begin{equation}
 \begin{aligned}
 \int_{\sum_{T\in\mathbb T}W_T\geq1}
 \Big|\sum_{s\in\mathbb I_\delta}F_s\Big|^p w
 &\lesssim R^\varepsilon
 \Big(\prod_{i=1}^J\delta_i^{-(i-1)}\Big)^{1-\theta}
 \int\Big|\sum_{T\in\mathbb T}W_T\Big|^2\\
 &\leq R^\varepsilon\prod_{i=1}^J\delta_i^{-(i-1)}
 \int\Big|\sum_{T\in\mathbb T}W_T\Big|^2.
 \end{aligned}
\end{equation}
So it suffices to prove the claim. First, 
Proposition~\ref{c13all:positive-moments} immediately implies \eqref{c13all:transfer-positive}. To prove the second estimate, we use the refined decoupling. Let $Y_\beta$ be the union of the unit cubes $Q$
meeting $\{\beta\leq\sum_{T\in\mathbb T}W_T<2\beta\}$, for dyadic $\beta\geq1$.
Local constancy  gives, for every such
cube $Q$ and $x\in Q$,
\begin{equation}
 \#\{T\in\mathbb T:CQ\cap10\delta^{-a_0}T\ne\varnothing\}
 \lesssim\delta^{-Ca_0}\sum_{T\in\mathbb T}W_T(x)
 \lesssim\delta^{-Ca_0}\beta.
\end{equation}
 Proposition~\ref{c13all:rp}  gives
\begin{equation}
 \begin{aligned}
 \int_{\beta\leq\sum_{T\in\mathbb T}W_T<2\beta}
 \Big|\sum_{s\in\mathbb I_\delta}F_s\Big|^q w
 &\lesssim R^\varepsilon\beta^{q/2-1}
          \sum_{T\in\mathbb T}|T|\\
 &\lesssim R^\varepsilon\beta^{q/2-1}
          \int\Big|\sum_{T\in\mathbb T}W_T\Big|^2.
 \end{aligned}
\end{equation}
Here the last inequality follows from $W_T\gtrsim1$ on $T$.
Since $\sum_{T\in\mathbb T}W_T\lesssim\delta^{-1}$, multiplying by
$\beta^{-(q/2-1)}$ and summing the dyadic ranges proves
\eqref{c13all:transfer-refined}.

\subsubsection{Proof of \eqref{c13all:small-square-target}}
On $\{\sum_{T\in\mathbb T}W_T<1\}$, since $p>4$ and
$0\leq w\leq1$, we have
\begin{equation}
 \Big|\sum_{s\in\mathbb I_\delta}F_s\Big|^p w
 \lesssim\Big|\sum_{T\in\mathbb T}W_T\Big|^p
 \leq\Big|\sum_{T\in\mathbb T}W_T\Big|^2.
\end{equation}
Integrating proves \eqref{c13all:small-square-target}, completing Step~1
for the original family.

\medskip
\noindent\textbf{Step 2. Localization.}
Fix a ball $B_R$ of radius $ R$. We now prove
\begin{equation}
 \int_{B_R}\Big|\sum_{s\in\mathbb{I}_\delta}F_s\Big|^p w
 \lesssim R^\varepsilon
 \prod_{i=1}^J\delta_i^{-(i-1)}
 \int_{\mathbb R^{n+1}}
 \Big(\sum_{s\in\mathbb{I}_\delta}|F_s|^{p/2}\Big)^2W_{B_R}.
\end{equation}
Choose a function $\chi_{B_R}$ such that 
\begin{equation}
 |\chi_{B_R}|\geq1\text{ on }B_R,\qquad
 \supp\widehat\chi_{B_R}\subset B(0,c/R),\qquad
 |\chi_{B_R}|\lesssim W_{B_R}.
\end{equation}
Moreover,
\begin{equation}
 \operatorname{supp}\widehat{\chi_{B_R}F_s}
 \subset \operatorname{supp}\widehat F_s+B(0,c/R).
\end{equation}
Since $\rho_j\geq\rho_n=R^{-1}$, we have
$c/R\leq c\rho_j$ for every $j$. Take $c$ sufficiently small. We then obtain
\begin{equation}
 \Big(\sum_{s\in\mathbb{I}_\delta}|\chi_{B_R}F_s|^{p/2}\Big)^2
 =|\chi_{B_R}|^p
 \Big(\sum_{s\in\mathbb{I}_\delta}|F_s|^{p/2}\Big)^2
 \lesssim W_{B_R}
 \Big(\sum_{s\in\mathbb{I}_\delta}|F_s|^{p/2}\Big)^2.
\end{equation}
Applying Step~1 to $\chi_{B_R}F_s$ gives
\begin{equation}
 \begin{aligned}
 \int_{B_R}\Big|\sum_{s\in\mathbb{I}_\delta}F_s\Big|^pw
 &\leq\int_{B_R}\Big|\sum_{s\in\mathbb{I}_\delta}\chi_{B_R}F_s\Big|^pw\\
 &\lesssim R^\varepsilon \prod_{i=1}^J\delta_i^{-(i-1)}
       \int\Big(\sum_{s\in\mathbb{I}_\delta}|\chi_{B_R}F_s|^{p/2}\Big)^2.
 \end{aligned}
\end{equation}
Together with the cutoff bound, this proves
\eqref{c13all:prefix-estimate-local}.

\section{Proof of the weighted Kakeya-type estimate}\label{sec:positive}

We prove Proposition~\ref{c13all:positive-moments}. Fix $n\geq4$, $R\geq1$,
and an integer $r$ with $4\leq r\leq n$. We use the scales and 
planks from Section~\ref{sec:transfer}. Let
$\mathbb I_\delta\subset[-1,1]$ be $\delta$-separated, and let $w$ be an
$n$-dimensional weight. For $s\in\mathbb I_\delta$,
suppose that the nonnegative functions $g_s$ satisfy
\begin{equation}
 \sup_{s\in\mathbb I_\delta}\|g_s\|_\infty\leq1,
 \qquad \operatorname{supp}\widehat g_s\subset Cq_s.
\end{equation}
Recall that
\begin{equation}
 q_s=\left\{\sum_{j=0}^n u_j\mathbf e_{j+1}(s):
                   |u_j|\leq\rho_j,\ 0\leq j\leq n\right\}.
\end{equation}
Our goal is to prove
\begin{equation}\label{1027.63}
 \int\Big|\sum_{s\in\mathbb I_\delta}g_s\Big|^r w
 \lesssim R^\varepsilon\delta^{-(r-2)}\rho_{r-2}
 \int\Big|\sum_{s\in\mathbb I_\delta}g_s\Big|^2.
\end{equation}

We argue by induction on $r$.
Section~\ref{sec:positive-radial} normalizes the annular pieces.
Section~\ref{sec:positive-decomposition} decomposes them into signed
frequency pieces and constructs the angular sums.
Lemma~\ref{lem:positive-angular-square} reduces the weighted integral
to an unweighted square function, and Lemma~\ref{lem:positive-square-charge}
controls that square function by the original nonnegative functions.
By \eqref{red:scale-equation} and \eqref{red:scales},
\begin{equation}
 \frac{\rho_j}{\rho_{j-1}}
 =\prod_{i=1}^{\min\{j,n-1\}}\delta_i\quad(1\leq j\leq n).
\end{equation}
Since $0<\delta_i\leq1$, these ratios are nonincreasing.
In particular, $\rho_j^2\geq\rho_{j-1}\rho_{j+1}$ and $ \delta\leq\frac{\rho_j}{\rho_{j-1}}\leq1$ for
$1\leq j<n$, and
\begin{equation}\label{c13all:ratio-properties}
 \begin{gathered}
 \frac{\rho_n}{\rho_{n-1}}=\frac{\rho_{n-1}}{\rho_{n-2}}=\delta,
 \qquad \delta^{-2}\leq R\leq\delta^{-n}.
 \end{gathered}
\end{equation}
We record these here because we use them frequently in this section.

\subsection{High--low decomposition and rescaling}\label{sec:positive-radial}

In this subsection, we apply the high-low decomposition to the function $\sum_s g_s$, and deal with the low frequency part. For the other frequencies, we apply rescaling so that we can do an inductive argument.

\subsubsection{High--low decomposition}

We decompose frequencies of $g_s$ radially. To be precise,
for every $s\in\mathbb I_\delta$, define
\begin{equation}
 g_{s,\mathrm{low}}:=[\varphi(\cdot/\rho_{r-2})\widehat g_s]^\vee,
 \qquad g_{s,\ell}:=[\eta_\ell\widehat g_s]^\vee.
\end{equation}
Here $\ell$ ranges over dyadic scales
$c^{-1}\rho_{r-2}<\ell\leq1$. The cutoffs satisfy
\begin{equation}
 \begin{gathered}
 \operatorname{supp}\widehat g_{s,\mathrm{low}}
   \subset B(0,C\rho_{r-2}),\qquad
 \operatorname{supp}\widehat g_{s,\ell}\subset\{|\xi|\sim\ell\},\\
 \big\|[\varphi(\cdot/\rho_{r-2})]^\vee\big\|_1
   +\|\eta_\ell^\vee\|_1\lesssim1,
 \end{gathered}
\end{equation}
Moreover, choose the cutoffs $\varphi$ and $\eta_\ell$ so that
 the high--low decomposition holds:
\begin{equation}\label{c13all:frequency-decomposition}
 g_s
 =g_{s,\mathrm{low}}
  +\sum_{\substack{c^{-1}\rho_{r-2}<\ell\leq1\\\ell\text{ dyadic}}}
     g_{s,\ell}.
\end{equation}
The reason why we use the scale $\rho_{r-2}$ for the low part is because in the low part we immediately have the desired result. Namely,
by local constancy at scale $\rho_{r-2}^{-1}$ and
\eqref{c13all:weight},
\begin{equation}
 \begin{aligned}
 \int\Big|\sum_{s\in\mathbb I_\delta}g_{s,\mathrm{low}}\Big|^rw
 &\leq\Big\|\sum_{s\in\mathbb I_\delta}g_{s,\mathrm{low}}\Big\|_\infty^{r-2}
       \int\Big|\sum_{s\in\mathbb I_\delta}g_{s,\mathrm{low}}\Big|^2w\\
 &\lesssim\delta^{-(r-2)}\rho_{r-2}
       \int\Big|\sum_{s\in\mathbb I_\delta}g_{s,\mathrm{low}}\Big|^2\\
 &\lesssim\delta^{-(r-2)}\rho_{r-2}
       \int\Big|\sum_{s\in\mathbb I_\delta}g_s\Big|^2.
 \end{aligned}
\end{equation}
This proves \eqref{1027.63} for the low frequency part.
\medskip

It remains to prove, for every dyadic $c^{-1}\rho_{r-2}<\ell\leq1$,
\begin{equation}
 \int\Big|\sum_{s\in\mathbb I_\delta}g_{s,\ell}\Big|^rw
 \lesssim R^\varepsilon\delta^{-(r-2)}\rho_{r-2}
 \int\Big|\sum_{s\in\mathbb I_\delta}g_s\Big|^2.
\end{equation}
We prove this estimate below.

\subsubsection{Rescaling an annulus}

Fix an even nonnegative smooth bump $\Psi_1$ with
$\int\Psi_1=1$ and $\Psi_1(x)\sim(1+|x|)^{-A}$, where
$A>n+1$ is a sufficiently large number. For $L>0$, set $\Psi_L(x):=L^{-(n+1)}\Psi_1(x/L)$.
Fix a dyadic scale $c^{-1}\rho_{r-2}<\ell\leq1$, and define
\begin{equation}\label{c13all:annulus-weight}
 W_\ell(y):=\ell^{-1}(w*\Psi_{\ell^{-1}})(y/\ell).
\end{equation}
Under the change of variables $x=y/\ell$,
\begin{equation}
 \widehat{g_{s,\ell}(\cdot/\ell)}(\zeta)
 =\ell^{n+1}\widehat g_{s,\ell}(\ell\zeta),\qquad
 \operatorname{supp}\widehat{g_{s,\ell}(\cdot/\ell)}
 \subset\{|\zeta|\sim1\}.
\end{equation}
By \eqref{c13all:weight}, $W_\ell$ is an $n$-dimensional weight. Local constancy and the change of
variables reduce the  estimate to
\begin{equation}
 \begin{aligned}
 \int\Big|\sum_{s\in\mathbb I_\delta}g_{s,\ell}(y/\ell)\Big|^r
       W_\ell(y)\,\mathrm{d}y
 &\lesssim R^\varepsilon\delta^{-(r-2)}\frac{\rho_{r-2}}{\ell}
 \int\Big|\sum_{s\in\mathbb I_\delta}g_s(y/\ell)\Big|^2\,\mathrm{d}y.
 \end{aligned}
\end{equation}

For each $s\in\mathbb I_\delta$, the rescaled function satisfies
\begin{equation}
 \operatorname{supp}\widehat{g_{s,\ell}(\cdot/\ell)}
 \subset\left\{\zeta\in\mathbb R^{n+1}: |\zeta|\sim1,\quad
 |\zeta\cdot\mathbf e_{j+1}(s)|\lesssim\rho_j/\ell
 \quad(0\leq j\leq n)\right\}.
\end{equation}
Thus the support is contained in a box whose side lengths, up to fixed
factors, are
\begin{equation}
 1\times\min\{1,\rho_1/\ell\}\times\cdots\times\min\{1,\rho_n/\ell\},
\end{equation}
in the directions $\mathbf e_1(s),\ldots,\mathbf e_{n+1}(s)$, respectively. In the next subsection, we  decompose these supports into
signed planks.

\subsection{Further frequency decomposition}\label{sec:positive-decomposition}

We construct auxiliary scales, decompose the annulus into signed
frequency pieces, and group them at the corresponding angular scale.

\subsubsection{Auxiliary scales}

To obtain the factor $\rho_{r-2}/\ell$ in this estimate, we construct
auxiliary widths with endpoint $\widetilde\rho_{r-2}=\rho_{r-2}/\ell$
and an angular scale $\widetilde\delta\geq\delta$.
The contact property \eqref{signedplankclaim} supplies the signed lower bounds
used in the decomposition below.
For $\xi\in\operatorname{supp}\eta_\ell
\cap\operatorname{supp}\widehat g_s$, put $\zeta:=\xi/\ell$.
Write
\begin{equation}\label{zeta-s}
 \begin{gathered}
 \zeta=\sum_{j=0}^n\lambda_j\mathbf e_{j+1}(s),
 \qquad |\lambda|\sim1, \qquad
 |\lambda_j|\leq\min(1,\rho_j/\ell)
 \end{gathered}
\end{equation}
for $0 \leq j \leq n$.
Since $|\lambda|\sim1$, we have
\begin{equation}
 \max_{r-2\leq j\leq n}|\lambda_j|
 \leq\frac{\rho_{r-2}}{\ell}\ll1
 \quad\Longrightarrow\quad
 \max_{0\leq j<r-2}|\lambda_j|\gtrsim1.
\end{equation}
\begin{lemma}\label{lem:positive-auxiliary-widths}
Let $\lambda$ satisfy \eqref{zeta-s}. There exists a positive sequence
$(\widetilde\rho_0,\ldots,\widetilde\rho_{r-2})$ satisfying
\begin{equation}\label{c13all:hull-data}
 \begin{gathered}
 \widetilde\rho_0=1,\qquad \widetilde\rho_{r-2}=\rho_{r-2}/\ell,\\
 \widetilde\rho_j\geq|\lambda_j|\quad(0\leq j\leq r-2),\\
 \widetilde\rho_j^2\geq\widetilde\rho_{j-1}\widetilde\rho_{j+1}
 \quad(1\leq j<r-2).
 \end{gathered}
\end{equation}
Moreover,
\begin{equation}\label{signedplankclaim}
 \widetilde\rho_j^2>\widetilde\rho_{j-1}\widetilde\rho_{j+1}
 \ \Longrightarrow\ |\lambda_j|=\widetilde\rho_j
 \qquad(1\leq j<r-2).
\end{equation}
\end{lemma}

\begin{figure}[H]
 \centering
 \resizebox{0.9\textwidth}{!}{
\begin{tikzpicture}[x=1bp,y=1bp,
  every node/.style={inner sep=0pt,outer sep=0pt},
  figone math/.style={font=\fontsize{10}{12}\selectfont,scale=1.40525},
  figone label/.style={font=\fontsize{10}{12}\selectfont,scale=1.304875},
  figone tick/.style={font=\fontfamily{phv}\fontsize{11bp}{13.2bp}\selectfont}]
  \definecolor{figoneblue}{rgb}{0.141176,0.4,0.619608}
  \definecolor{figonedark}{rgb}{0.109804,0.152941,0.192157}
  \definecolor{figonegray}{rgb}{0.529412,0.572549,0.607843}
  \path[use as bounding box] (0,0) rectangle (713,328);
  \begin{scope}[shift={(-42,-49)}]
    \filldraw[fill=figonedark,draw=figonedark,line width=1bp]
      (53,359) circle (3.4bp);
    \node[figone math,anchor=base west,text=figonedark]
      at (65,355) {$\log|\lambda_j|$};
    \draw[figoneblue,line width=2.2bp] (214,359)--(239,359);
    \filldraw[fill=white,draw=figoneblue,line width=1.5bp]
      (226.5,359) circle (4bp);
    \node[figone math,anchor=base west,text=figoneblue]
      at (249,355) {$u_j=\log\widetilde\rho_j$};
    \filldraw[fill=figoneblue,draw=white,line width=1bp]
      (485.5,354.5) rectangle (494.5,363.5);
    \node[anchor=base west,text=figonegray,
      font=\fontfamily{phv}\fontsize{12bp}{14.4bp}\selectfont]
      at (504,355) {fixed endpoints};

    \draw[figonegray,line width=.8bp] (69,326)--(69,77)--(680,77);
    \draw[figonegray,line width=.8bp] (65,309)--(73,309);
    \node[figone tick,anchor=base east,text=figonegray] at (57,305) {0};
    \foreach \j/\xx in {0/92,1/180,2/268,3/356,4/444,5/532,6/620}{
      \draw[figonegray,line width=.7bp] (\xx,74)--(\xx,80);
      \node[figone tick,anchor=base,text=figonegray] at (\xx,58) {\j};
    }
    \node[figone math,anchor=base west,text=figonegray] at (692,73) {$j$};

    \draw[figoneblue,line width=2.5bp,line join=round]
      (92,309)--(180,292.425)--(268,275.85)--(356,250.35)
      --(444,224.85)--(532,185.325)--(620,145.8);
    \foreach \xx/\yy in {180/292.425,356/250.35,532/185.325}{
      \filldraw[fill=white,draw=figoneblue,line width=1.7bp]
        (\xx,\yy) circle (4.2bp);
    }

    \foreach \xx/\yy in {92/286.05,180/278.4,356/232.5,532/158.55,620/125.4}{
      \filldraw[fill=figonedark,draw=figonedark,line width=1bp]
        (\xx,\yy) circle (3.8bp);
    }
    \foreach \xx/\yy in {268/275.85,444/224.85}{
      \filldraw[fill=white,draw=figoneblue,line width=1.7bp]
        (\xx,\yy) circle (7bp);
      \filldraw[fill=figonedark,draw=figonedark,line width=1bp]
        (\xx,\yy) circle (3.4bp);
    }

    \filldraw[fill=figoneblue,draw=white,line width=1bp]
      (87.5,304.5) rectangle (96.5,313.5);
    \filldraw[fill=figoneblue,draw=white,line width=1bp]
      (615.5,141.3) rectangle (624.5,150.3);
    \node[figone label,anchor=base west,text=figoneblue] at (103,319) {$u_0=0$};
    \node[figone label,anchor=base west,text=figoneblue] at (636,148) {$u_6=-6.4$};
    \node[figone label,anchor=base west,text=figonedark] at (636,119)
      {$\log|\lambda_6|=-7.2$};
  \end{scope}
\end{tikzpicture}}
 \caption{}
 \label{fig:positive-auxiliary-widths}
\end{figure}

\begin{proof}
Let $\mathcal C$ be the set of positive sequences satisfying \eqref{c13all:hull-data}.
Since $(\min\{1,\rho_j/\ell\})_{j=0}^{r-2}\in\mathcal C$, define
$\widetilde\rho_j:=\inf_{a\in\mathcal C}a_j$ for $0\leq j\leq r-2$.
Then
\begin{equation}
 \left(\frac{\rho_{r-2}}{\ell}\right)^{j/(r-2)}
 \leq\widetilde\rho_j<\infty
 \qquad(0\leq j\leq r-2),
\end{equation}
and
\begin{equation}
 \widetilde\rho_j^2=\inf_{a\in\mathcal C}a_j^2
 \geq\widetilde\rho_{j-1}\widetilde\rho_{j+1}
 \qquad(1\leq j<r-2).
\end{equation}
Thus \eqref{c13all:hull-data} holds.

If \eqref{signedplankclaim} fails at an interior index $j$, replacing
only the $j$-th entry by the maximum below preserves \eqref{c13all:hull-data}.
Consequently,
\begin{equation}
 \widetilde\rho_j=\inf_{a\in\mathcal C}a_j
 \leq\max\bigl\{|\lambda_j|,
       (\widetilde\rho_{j-1}\widetilde\rho_{j+1})^{1/2}\bigr\}
 <\widetilde\rho_j,
\end{equation}
a contradiction.
\end{proof}

\begin{lemma}\label{lem:positive-auxiliary-scales}
Let $(\widetilde\rho_0,\ldots,\widetilde\rho_{r-2})$ be as in
Lemma~\ref{lem:positive-auxiliary-widths}. There exist
$\widetilde\delta_1,\ldots,\widetilde\delta_{n-1}\in(0,1]$, with
$\widetilde\delta_j=1$ for $r-2<j\leq n-1$, whose associated widths
extend this sequence:
\begin{equation}\label{c13all:hull-tuple}
 \widetilde\rho_j=
 \prod_{i=1}^{\min(j,r-2)}\widetilde\delta_i^{j+1-i}
 \qquad(0\leq j\leq n).
\end{equation}
The associated angular scale satisfies
\begin{equation}\label{c13all:hull-slopes}
 \widetilde\delta:=\prod_{i=1}^{n-1}\widetilde\delta_i
 =\prod_{i=1}^{r-2}\widetilde\delta_i
 \geq\prod_{i=1}^{r-2}\delta_i\geq\delta.
\end{equation}
\end{lemma}

\begin{proof}
Let $a<b$ be consecutive indices where
$\widetilde\rho_j^2>\widetilde\rho_{j-1}\widetilde\rho_{j+1}$,
with $0$ and $r-2$ also included.
By \eqref{zeta-s}--\eqref{signedplankclaim},
$\widetilde\rho_k\leq\min\{1,\rho_k/\ell\}$ for $k\in\{a,b\}$.
Since $\widetilde\rho_j^2=\widetilde\rho_{j-1}\widetilde\rho_{j+1}$ for $a<j<b$,
log-concavity of $(\min\{1,\rho_j/\ell\})_{j=0}^{r-2}$ gives
\begin{equation}
 \widetilde\rho_j
 =\widetilde\rho_a^{(b-j)/(b-a)}\widetilde\rho_b^{(j-a)/(b-a)}
 \leq\min\{1,\rho_j/\ell\}
 \qquad(a\leq j\leq b).
\end{equation}
Set
\begin{equation}
 \widetilde\delta_1:=\widetilde\rho_1,\quad
 \widetilde\delta_j:=\frac{\widetilde\rho_j\widetilde\rho_{j-2}}{\widetilde\rho_{j-1}^2}
 \quad(2\leq j\leq r-2),\quad
 \widetilde\delta_j:=1\quad(r-2<j\leq n-1).
\end{equation}
The preceding bound and log-concavity give
$0<\widetilde\delta_j\leq1$. Telescoping recovers the given widths
in \eqref{c13all:hull-tuple} and gives
$\widetilde\delta=\widetilde\rho_{r-2}/\widetilde\rho_{r-3}$.
The endpoint condition and
$\widetilde\rho_{r-3}\leq\rho_{r-3}/\ell$ imply
\begin{equation}
 \widetilde\delta\geq\frac{\rho_{r-2}}{\rho_{r-3}}
 =\prod_{i=1}^{r-2}\delta_i\geq\delta,
\end{equation}
proving \eqref{c13all:hull-slopes}.
\end{proof}

For each $\lambda$, fix the scales and extended widths from
Lemma~\ref{lem:positive-auxiliary-scales}. They satisfy
\begin{equation}\label{c13all:hull-support}
 |\lambda_j|\leq\widetilde\rho_j\quad(0\leq j\leq n).
\end{equation}
For $0\leq j\leq r-2$, this follows from
Lemma~\ref{lem:positive-auxiliary-widths}.
For $r-2<j\leq n$, combine \eqref{zeta-s},
\eqref{c13all:hull-tuple}--\eqref{c13all:hull-slopes}
with the decreasing ratios $\rho_j/\rho_{j-1}$.
The tail widths satisfy
\begin{equation}
 \widetilde\rho_j=\frac{\rho_{r-2}}{\ell}
       \widetilde\delta^{j-r+2}
 \geq\frac{\rho_j}{\ell}\quad(r-2\leq j\leq n).
\end{equation}
The identity follows from \eqref{c13all:hull-tuple}
and the endpoint condition in \eqref{c13all:hull-data}.
The inequality follows from the decreasing original width ratios and
$\rho_{r-2}/\rho_{r-3}=\prod_{i=1}^{r-2}\delta_i\leq\widetilde\delta$.
In particular, all the extended widths lie in $[R^{-1},1]$.

\subsubsection{Signed frequency decomposition}\label{sec:positive-signed-decomposition}

We split the rescaled annulus into signed planks $p_{s,\alpha}$.
For $1\leq j<r-2$, Lemmas~\ref{lem:positive-auxiliary-widths}
and~\ref{lem:positive-auxiliary-scales} give
$|\lambda_j|=\widetilde\rho_j$ whenever $\widetilde\delta_{j+1}<1$;
set $\nu_{j+1}:=\operatorname{sgn}\lambda_j$ at these indices.
When $\widetilde\delta_1$ is sufficiently small, the bounds
$|\lambda_j|\leq\widetilde\rho_j\leq\widetilde\delta_1$ for $1\leq j\leq n$
and $|\lambda|\sim1$ give
\begin{equation}
 |\lambda_0|\geq|\lambda|-\sum_{j=1}^n|\lambda_j|
 \geq|\lambda|-n\widetilde\delta_1\gtrsim1.
\end{equation}
We then set $\nu_1:=\operatorname{sgn}\lambda_0$.
If $\widetilde\delta_1\sim1$, set $\nu_1:=0$.
Set all remaining signs to zero. No signed lower bound is required
for $r-2\leq j\leq n$.

For each $\lambda$, choose integers $k_j(\lambda)$ such that
\begin{equation}
 2^{k_j(\lambda)}\sim\widetilde\rho_j(\lambda)
 \quad(0\leq j\leq n).
\end{equation}
Let $\Lambda$ be the set of coefficient vectors in \eqref{zeta-s}.
Let $\mathcal A$ index the dyadic sequences and signs
$\alpha=(k,\nu)$ that occur. Then
\begin{equation}
 \Lambda=\bigsqcup_{\alpha\in\mathcal A}\Lambda_\alpha,\qquad
 \Lambda_\alpha:=\{\lambda\in\Lambda:k(\lambda)=k,\ \nu(\lambda)=\nu\}\subset B_\alpha.
\end{equation}
Since $R^{-1}\leq\widetilde\rho_j\leq1$ for $0\leq j\leq n$,
we have $\#\mathcal A\lesssim(1+\log R)^{n+1}$.
For each class, fix a representative coefficient vector and retain
its auxiliary scales; the dyadic values only index the classes.
The containing coordinate boxes are
\begin{equation}
 B_\alpha:=\left\{\lambda:
 \begin{array}{ll}
 |\lambda_j|\lesssim2^{k_j},&0\leq j\leq r-2,\\
 |\lambda_j|\leq\rho_j/\ell,&r-2<j\leq n,\\
 \nu_{j+1}\lambda_j\gtrsim2^{k_j},&0\leq j<r-2,\ \nu_{j+1}\ne0
 \end{array}\right\}.
\end{equation}
The corresponding frequency planks are
\begin{equation}
 p_{s,\alpha}:=\left\{\sum_{j=0}^n\lambda_j\mathbf e_{j+1}(s):
                              \lambda\in B_\alpha\right\}.
\end{equation}
For each fixed $s$, the planks $\{p_{s,\alpha}:\alpha\in\mathcal A\}$
form a bounded-overlap cover of the set of $\zeta$ in \eqref{zeta-s}.
In particular,
\begin{equation}
 \operatorname{supp}\widehat{g_{s,\ell}(\cdot/\ell)}
 \subset\bigcup_{\alpha\in\mathcal A}p_{s,\alpha}.
\end{equation}
Figure~\ref{fig:signed-frequency-decomposition} illustrates this covering for $r=4$.

\begin{figure}[htbp]
 \centering
 \resizebox{0.9\textwidth}{!}{
\begin{tikzpicture}[x=.01cm,y=-.01cm,line cap=round,line join=round,>=latex]
  \definecolor{sdInk}{RGB}{20,34,52}
  \definecolor{sdEdge}{RGB}{69,91,116}
  \definecolor{sdAxis}{RGB}{94,114,132}
  \definecolor{sdTop}{RGB}{207,220,232}
  \definecolor{sdFront}{RGB}{155,165,175}
  \definecolor{sdSide}{RGB}{173,184,195}
  \definecolor{sdBlue}{RGB}{73,136,191}
  \definecolor{sdBlueSide}{RGB}{53,111,159}
  \definecolor{sdBlueTop}{RGB}{83,149,207}
  \definecolor{sdOrange}{RGB}{225,146,48}
  \definecolor{sdOrangeSide}{RGB}{185,119,30}
  \definecolor{sdOrangeTop}{RGB}{245,163,57}
  \definecolor{sdGreen}{RGB}{68,148,132}
  \definecolor{sdGreenSide}{RGB}{56,128,115}
  \definecolor{sdGreenTop}{RGB}{79,162,146}
  \path[use as bounding box] (0,0) rectangle (1625,727);

  \node[anchor=west,inner sep=0pt,text=sdInk,
    font=\rmfamily\fontsize{10}{12}\selectfont]
    at (33,35) {Before decomposition};
  \node[anchor=west,inner sep=0pt,text=sdInk,
    font=\rmfamily\fontsize{10}{12}\selectfont]
    at (809,35) {After decomposition: one box per $\alpha$};

  \path[fill=sdSide]
    (245,337)
    .. controls (409,297) and (577,226) .. (581,121)
    -- (581,360)
    .. controls (555,459) and (415,541) .. (245,577) -- cycle;
  \path[fill=sdFront] (105,300)--(245,337)--(245,577)--(105,540)--cycle;
  \path[fill=sdTop]
    (105,300)
    .. controls (265,276) and (442,210) .. (441,117)
    -- (441,84)--(581,121)
    .. controls (577,226) and (409,297) .. (245,337)--cycle;

  \draw[sdAxis,line width=.52pt,->] (73,581)--(186,610);
  \draw[sdAxis,line width=.52pt,->] (73,581)--(156,527);
  \draw[sdAxis,line width=.52pt,->] (73,581)--(73,482);
  \node[text=sdAxis,inner sep=0pt,font=\fontsize{10}{12}\selectfont,scale=.84]
    at (204,609) {$\lambda_0$};
  \node[text=sdAxis,inner sep=0pt,font=\fontsize{10}{12}\selectfont,scale=.84]
    at (173,520) {$\lambda_1$};
  \node[text=sdAxis,inner sep=0pt,font=\fontsize{10}{12}\selectfont,scale=.84]
    at (57,463) {$\lambda_2$};
  \draw[sdAxis,line width=.36pt] (210,337)--(210,577);
  \draw[sdAxis,line width=.36pt] (200,337)--(220,337);
  \draw[sdAxis,line width=.36pt] (200,577)--(220,577);
  \node[text=sdAxis,inner sep=0pt,font=\fontsize{10}{12}\selectfont,scale=.8]
    at (184,451) {$2a$};
  \node[anchor=west,text=sdInk,inner sep=0pt,
    font=\fontsize{10}{12}\selectfont,scale=.75] at (207,647)
    {$|\lambda|\sim1,\; |\lambda_1|\leq1/2,\; |\lambda_2|\leq a$};

  \draw[sdAxis,line width=.84pt,->] (670,335)--(750,335);

  \path[fill=sdGreen]
    (865,484).. controls (884,486) and (909,495) .. (932,501)
    -- (932,615).. controls (905,610) and (882,605) .. (865,598)--cycle;
  \path[fill=sdGreenSide]
    (932,501).. controls (957,494) and (979,487) .. (994,480)
    --(994,596).. controls (974,604) and (949,611) .. (932,615)--cycle;
  \path[fill=sdGreenTop]
    (865,484).. controls (891,480) and (915,473) ..(937,466)
    --(994,480).. controls (979,487) and (957,494) ..(932,501)--cycle;

  \path[fill=sdOrange]
    (1050,394).. controls (1068,402) and (1090,409) ..(1108,414)
    --(1108,522).. controls (1086,518) and (1067,510) ..(1050,505)--cycle;
  \path[fill=sdOrangeSide]
    (1108,414).. controls (1121,410) and (1129,403) ..(1134,396)
    --(1134,506).. controls (1129,513) and (1120,519) ..(1108,522)--cycle;
  \path[fill=sdOrangeTop]
    (1050,394)--(1074,382)--(1134,396)
    .. controls (1129,403) and (1121,410) ..(1108,414)--cycle;

  \path[fill=sdBlue]
    (1158,330).. controls (1175,336) and (1196,343) ..(1216,348)
    --(1216,459).. controls (1194,455) and (1173,448) ..(1158,441)--cycle;
  \path[fill=sdBlueSide]
    (1216,348)--(1253,323)--(1253,435)--(1216,459)--cycle;
  \path[fill=sdBlueTop]
    (1158,330)--(1195,306)--(1253,323)--(1216,348)--cycle;

  \path[fill=sdOrange]
    (1278,251).. controls (1295,258) and (1315,265) ..(1333,270)
    --(1333,380).. controls (1310,376) and (1291,369) ..(1278,361)--cycle;
  \path[fill=sdOrangeSide]
    (1333,270).. controls (1341,265) and (1347,260) ..(1349,254)
    --(1349,363).. controls (1347,371) and (1341,376) ..(1333,380)--cycle;
  \path[fill=sdOrangeTop]
    (1278,251)--(1290,238)--(1349,254)
    .. controls (1347,260) and (1341,265) ..(1333,270)--cycle;

  \path[fill=sdGreen]
    (1405,166).. controls (1420,172) and (1439,177) ..(1457,182)
    --(1457,294).. controls (1438,289) and (1417,283) ..(1405,278)--cycle;
  \path[fill=sdGreenSide]
    (1457,182).. controls (1468,175) and (1476,163) ..(1479,149)
    --(1479,261).. controls (1477,277) and (1468,287) ..(1457,294)--cycle;
  \path[fill=sdGreenTop]
    (1405,166).. controls (1409,154) and (1412,142) ..(1413,131)
    --(1479,149).. controls (1476,163) and (1468,175) ..(1457,182)--cycle;

  \newcommand{\SDwirebox}[5]{%
    \coordinate (sdA) at (#1,#2);
    \coordinate (sdB) at ($(sdA)+(#3,#4)$);
    \coordinate (sdC) at ($(sdB)+(123,32)$);
    \coordinate (sdD) at ($(sdA)+(123,32)$);
    \coordinate (sdE) at ($(sdA)+(0,#5)$);
    \coordinate (sdF) at ($(sdB)+(0,#5)$);
    \coordinate (sdG) at ($(sdC)+(0,#5)$);
    \coordinate (sdH) at ($(sdD)+(0,#5)$);
    \begin{scope}
      \clip (sdA)--(sdB)--(sdC)--(sdD)--cycle;
      \foreach \i in {-20,...,20}{%
        \draw[sdEdge,line width=.14pt,opacity=.72]
          ($(sdB)+({12*\i},0)$)--($(sdD)+({12*\i},0)$);
      }
    \end{scope}
    \begin{scope}
      \clip (sdA)--(sdD)--(sdH)--(sdE)--cycle;
      \foreach \i in {-20,...,20}{%
        \draw[sdEdge,line width=.14pt,opacity=.72]
          ($(sdA)+({12*\i},0)$)--($(sdH)+({12*\i},0)$);
      }
    \end{scope}
    \begin{scope}
      \clip (sdD)--(sdC)--(sdG)--(sdH)--cycle;
      \foreach \i in {-20,...,20}{%
        \draw[sdEdge,line width=.14pt,opacity=.72]
          ($(sdD)+(0,{12*\i})$)--($(sdG)+(0,{12*\i})$);
      }
    \end{scope}
    \draw[sdEdge,line width=.36pt]
      (sdA)--(sdB)--(sdC)--(sdD)--cycle
      (sdA)--(sdE)--(sdH)--(sdG)--(sdC)
      (sdD)--(sdH);
    \draw[sdEdge,line width=.18pt,opacity=.7]
      (sdB)--(sdF)--(sdE) (sdF)--(sdG);
  }
  \SDwirebox{766}{525}{141}{-90}{125}
  \SDwirebox{960}{400}{70}{-44}{125}
  \SDwirebox{1101}{310}{40}{-26}{125}
  \SDwirebox{1212}{239}{70}{-45}{125}
  \SDwirebox{1334}{160}{141}{-90}{125}

  \node[anchor=east,inner sep=0pt,text=sdInk,
    font=\fontsize{10}{12}\selectfont,scale=.84] at (754,648) {$\nu_2=-1$};
  \node[anchor=west,inner sep=0pt,text=sdInk,
    font=\fontsize{10}{12}\selectfont,scale=.84] at (1176,506) {$\nu_2=0$};
  \node[anchor=west,inner sep=0pt,text=sdInk,
    font=\fontsize{10}{12}\selectfont,scale=.84] at (1458,349) {$\nu_2=+1$};

  \path[fill=sdBlue] (40,689) rectangle (66,715);
  \node[anchor=west,text=sdInk,inner sep=0pt,
    font=\fontsize{10}{12}\selectfont,scale=.85] at (79,701) {$|\lambda_1|\leq1/8$};
  \path[fill=sdOrange] (530,689) rectangle (556,715);
  \node[anchor=west,text=sdInk,inner sep=0pt,
    font=\fontsize{10}{12}\selectfont,scale=.85] at (568,701) {$1/8<|\lambda_1|\leq1/4$};
  \path[fill=sdGreen] (1075,689) rectangle (1101,715);
  \node[anchor=west,text=sdInk,inner sep=0pt,
    font=\fontsize{10}{12}\selectfont,scale=.85] at (1114,701) {$1/4<|\lambda_1|\leq1/2$};
\end{tikzpicture}}
 \caption{}
 \label{fig:signed-frequency-decomposition}
\end{figure}


We now use these planks to construct a smooth frequency partition.
For each $s$, choose smooth cutoffs $\chi_{\alpha,s}$ supported in
$p_{s,\alpha}$ and uniformly adapted to the containing coordinate boxes,
in the sense of Definition~\ref{exp:adapted-bump}, with
\begin{equation}
 \sum_{\alpha\in\mathcal A}\chi_{\alpha,s}(\zeta)=1
 \quad\text{on }\operatorname{supp}
       \widehat{g_{s,\ell}(\cdot/\ell)}.
\end{equation}
Define the pieces in the rescaled variable by
\begin{equation}
 \begin{aligned}
 \widehat g_{s,\ell,\alpha}(\zeta)
 &:=\chi_{\alpha,s}(\zeta)\,
       \ell^{n+1}\widehat g_{s,\ell}(\ell\zeta)=\chi_{\alpha,s}(\zeta)\eta_\ell(\ell\zeta)\,
       \ell^{n+1}\widehat g_s(\ell\zeta).
 \end{aligned}
\end{equation}
Since $\ell^{n+1}\widehat g_{s,\ell}(\ell\zeta)$ is the Fourier
transform of $g_{s,\ell}(\cdot/\ell)$, we obtain
\begin{equation}
 \begin{aligned}
 g_{s,\ell}(y/\ell)&=\sum_{\alpha\in\mathcal A}g_{s,\ell,\alpha}(y).
 \end{aligned}
\end{equation}
For each fixed dyadic scale $c^{-1}\rho_{r-2}<\ell\leq1$,
H\"older's inequality therefore gives
\begin{equation}\label{c13all:high-low-reduction}
 \begin{aligned}
 &\int\Big|\sum_{s\in\mathbb I_\delta}g_{s,\ell}(y/\ell)\Big|^r
       W_\ell(y)\,\mathrm{d}y\lesssim(1+\log R)^C\sum_{\alpha\in\mathcal A}
 \int\Big|\sum_{s\in\mathbb I_\delta}g_{s,\ell,\alpha}(y)\Big|^r
       W_\ell(y)\,\mathrm{d}y.
 \end{aligned}
\end{equation}

It therefore suffices to prove, uniformly in $\ell$ and $\alpha$,
\begin{equation}
 \begin{aligned}
 &\int\Big|\sum_{s\in\mathbb I_\delta}g_{s,\ell,\alpha}(y)\Big|^r
       W_\ell(y)\,\mathrm{d}y\lesssim R^\varepsilon\delta^{-(r-2)}\frac{\rho_{r-2}}{\ell}
 \int\Big|\sum_{s\in\mathbb I_\delta}g_s(y/\ell)\Big|^2\,\mathrm{d}y.
 \end{aligned}
\end{equation}
The logarithmic loss in \eqref{c13all:high-low-reduction} is absorbed
into $R^\varepsilon$.

\subsubsection{Grouping the angular pieces}\label{sec:positive-angular-grouping}

Fix a class $\alpha\in\mathcal A$ and use its representative auxiliary
scales $\widetilde\rho_j$, $\widetilde\delta_j$ and $\widetilde\delta$.
Let $\widetilde{\mathcal J}_{r-2}$ partition $[-1,1]$ into intervals
of length $c\widetilde\delta$, where $c>0$ is sufficiently small, and let $s_\tau$
be the center of $\tau$. Define
\begin{equation}
 g_\tau:=\sum_{s\in\mathbb I_\delta\cap\tau}g_{s,\ell,\alpha}.
\end{equation}
We prove the common signed-plank support and
\begin{equation}\label{c13all:coarse-amplitude}
 \|g_\tau\|_\infty\lesssim \frac{\widetilde\delta}{\delta}
 \qquad(\tau\in\widetilde{\mathcal J}_{r-2}).
\end{equation}
At every angular scale below, a fixed finite partition of the interval
centers gives the separation required by the angular estimates.

For $\zeta\in\operatorname{supp}\widehat g_{s,\ell,\alpha}$, write
$\zeta=\sum_{j=0}^n\lambda_j\mathbf e_{j+1}(s)$. The signed decomposition gives
\begin{equation}
 \begin{gathered}
 |\lambda_j|\lesssim\widetilde\rho_j\quad(0\leq j\leq n),\\
 \nu_{j+1}\lambda_j\gtrsim\widetilde\rho_j
 \quad(0\leq j<r-2,\ \nu_{j+1}\ne0).
 \end{gathered}
\end{equation}

\begin{lemma}[Change of frame]\label{prelim:change-of-frame}
Let $\mathbf b_j:=\mathbf e_{j+1}$, or $\mathbf b_j:=v^{(j)}$,
where $v$ has uniform $C^{n+1}$ bounds and
$[v,\ldots,v^{(n)}]^{-1}$ is uniformly bounded.
For $s,t\in[-1,1]$ with $d:=|s-t|\leq1$, write
\begin{equation}
 \xi=\sum_{i=0}^n\lambda_i\mathbf b_i(s)
    =\sum_{j=0}^n\mu_j\mathbf b_j(t).
\end{equation}
Then, for $0\leq j\leq n$,
\begin{equation}
 |\mu_j-\lambda_j|
 \lesssim\sum_{i=0}^{j-1}d^{j-i}|\lambda_i|
       +d\sum_{i=j}^n|\lambda_i|.
\end{equation}
\end{lemma}

\begin{proof}
Taylor's formula for the derivative frame, and the Frenet equations
for the orthonormal frame, show that the coefficient of
$\mathbf b_j(t)$ in $\mathbf b_i(s)$ is $O(d^{j-i})$ for $j>i$
and $\delta_{ij}+O(d)$ for $j\leq i$.
Summing against $\lambda_i$ gives the claim.
\end{proof}

For $s\in\tau$ and $\zeta\in\operatorname{supp}\widehat g_{s,\ell,\alpha}$,
 define the two coefficient vectors by
\begin{equation}
 \zeta=\sum_{i=0}^n\lambda_i\mathbf e_{i+1}(s)
           =\sum_{i=0}^n\lambda_i'\mathbf e_{i+1}(s_\tau).
\end{equation}
The widths are nonincreasing and satisfy
$\widetilde\rho_i \widetilde\delta^{j-i}\leq\widetilde\rho_j$ for
$0\leq i\leq j\leq n$.
Lemma~\ref{prelim:change-of-frame}, with $|s-s_\tau|\leq c\widetilde\delta$, gives
\begin{equation}
 |\lambda_j'-\lambda_j|\lesssim c\widetilde\rho_j
 \qquad(0\leq j\leq n).
\end{equation}
Taking $c>0$ sufficiently small preserves the signed lower bounds:
\begin{equation}
 \begin{gathered}
 |\lambda_j'|\lesssim\widetilde\rho_j\quad(0\leq j\leq n),\\
 \nu_{j+1}\lambda_j'\gtrsim\widetilde\rho_j
 \quad(0\leq j<r-2,\ \nu_{j+1}\ne0).
 \end{gathered}
\end{equation}
Consequently,
\begin{equation}
 \operatorname{supp}\widehat g_\tau
 \subset
 \left\{\sum_{j=0}^n\lambda_j'\mathbf e_{j+1}(s_\tau):
 \begin{array}{ll}
 |\lambda_j'|\lesssim\widetilde\rho_j,&0\leq j\leq n,\\
 \nu_{j+1}\lambda_j'\gtrsim\widetilde\rho_j,&0\leq j<r-2,\ \nu_{j+1}\ne0
 \end{array}\right\}.
\end{equation} In the directions
$\mathbf e_1(s_\tau),\ldots,\mathbf e_{n+1}(s_\tau)$, this frequency
plank has dimensions, up to fixed factors,
\begin{equation}
 1\times\widetilde\rho_1\times\cdots\times\widetilde\rho_{r-2}
 \times(\widetilde\rho_{r-2} \widetilde\delta)\times\cdots\times
 (\widetilde\rho_{r-2} \widetilde\delta^{n-r+2}).
\end{equation}
The adaptation of the cutoffs gives
\begin{equation}
 \big\|\mathcal F^{-1}_{\zeta}
 [\chi_{\alpha,s}(\zeta)\eta_\ell(\ell\zeta)]\big\|_1\lesssim1.
\end{equation}
Consequently, Young's inequality and $\|g_s\|_\infty\leq1$ imply
\begin{equation}
 \|g_{s,\ell,\alpha}\|_\infty\lesssim1,\qquad
 \|g_\tau\|_\infty
 \lesssim\#(\mathbb I_\delta\cap\tau)\lesssim \frac{\widetilde\delta}{\delta}.
\end{equation}
This proves \eqref{c13all:coarse-amplitude}.
\medskip

It remains to prove
\begin{equation}
 \int\Big|\sum_{\tau\in\widetilde{\mathcal J}_{r-2}}g_\tau(y)\Big|^r
       W_\ell(y)\,\mathrm{d}y
 \lesssim R^\varepsilon\bigl(\frac{\widetilde\delta}{\delta}\bigr)^{r-2}
 \prod_{j=1}^{r-2}\widetilde\delta_j^{-(j-1)}
 \int\Big|\sum_{s\in\mathbb I_\delta}g_s(y/\ell)\Big|^2\,\mathrm{d}y.
\end{equation}
We obtain this from two estimates. In Section~\ref{sec:positive-angular},
we prove the weighted estimate~\eqref{c13all:annular-square-estimate},
\begin{equation}
 \int\Big|\sum_{\tau\in\widetilde{\mathcal J}_{r-2}}g_\tau\Big|^rW_\ell
 \lesssim R^\varepsilon\bigl(\frac{\widetilde\delta}{\delta}\bigr)^{r-4}
 \prod_{j=1}^{r-2}\widetilde\delta_j^{-(j-1)}
 \int\Big(\sum_{\tau\in\widetilde{\mathcal J}_{r-2}}|g_\tau|^2\Big)^2.
\end{equation}
In Section~\ref{sec:positive-kernels}, we prove the square-function
bound~\eqref{c13all:coarse-amplitude-charge},
\begin{equation}
 \int\Big(\sum_{\tau\in\widetilde{\mathcal J}_{r-2}}|g_\tau|^2\Big)^2
 \lesssim R^\varepsilon\bigl(\frac{\widetilde\delta}{\delta}\bigr)^2
 \int\Big|\sum_{s\in\mathbb I_\delta}g_s(y/\ell)\Big|^2\,\mathrm{d}y.
\end{equation}
Combining these estimates with $\varepsilon/2$ in place of $\varepsilon$
gives the preceding bound. By \eqref{c13all:hull-data},
\eqref{c13all:hull-tuple}, and \eqref{c13all:hull-slopes},
\begin{equation}
 \bigl(\frac{\widetilde\delta}{\delta}\bigr)^{r-2}
 \prod_{j=1}^{r-2}\widetilde\delta_j^{-(j-1)}
 =\delta^{-(r-2)}\prod_{j=1}^{r-2}\widetilde\delta_j^{r-1-j}
 =\delta^{-(r-2)}\widetilde\rho_{r-2}
 =\delta^{-(r-2)}\frac{\rho_{r-2}}{\ell}.
\end{equation}
This gives the required estimate. Summing over $\ell$ and $\alpha$
completes the proof of \eqref{c13all:positive-moment-estimate}.

\subsection{The weighted estimate}\label{sec:positive-angular}

\begin{lemma}\label{lem:positive-angular-square}
Let $4\leq r\leq n$ be an integer. If $r\geq5$, assume that
Proposition~\ref{c13all:positive-moments}
holds at exponent $r-1$, uniformly over admissible scale tuples.
Then, for every $\varepsilon>0$,
\begin{equation}\label{c13all:annular-square-estimate}
 \begin{aligned}
 \int\Big|\sum_{\tau\in\widetilde{\mathcal J}_{r-2}}g_\tau\Big|^rW_\ell
 &\lesssim R^\varepsilon\bigl(\frac{\widetilde\delta}{\delta}\bigr)^{r-4}
 \prod_{j=1}^{r-2}\widetilde\delta_j^{-(j-1)}
 \int\Big(\sum_{\tau\in\widetilde{\mathcal J}_{r-2}}|g_\tau|^2\Big)^2.
 \end{aligned}
\end{equation}
\end{lemma}

Let us give a proof of Lemma~\ref{lem:positive-angular-square}.
When $\nu_1=0$, we apply the signed-plank estimates with
$(1,\widetilde\delta_2,\ldots,\widetilde\delta_{n-1})$ and its truncations.
The corresponding widths and angular scales are comparable to the original ones,
and the physical scales are no larger. All auxiliary scales in the formulas
below retain their original values.
For $r=4$, $W_\ell\lesssim1$ and \eqref{c13all:unweighted-four},
applied at a physical scale at most $\widetilde\rho_n^{-1}\leq\ell R\leq R$, give
\begin{equation}
 \int\Big|\sum_{\tau\in\widetilde{\mathcal J}_{r-2}}g_\tau\Big|^4W_\ell
 \lesssim R^\varepsilon
 \int\Big(\sum_{\tau\in\widetilde{\mathcal J}_{r-2}}|g_\tau|^2\Big)^2.
\end{equation}
This implies \eqref{c13all:annular-square-estimate}, since
$\prod_{j=1}^{r-2}\widetilde\delta_j^{-(j-1)}\geq1$.

For the remainder of the proof, assume $r\geq5$ and set
$J:=\lfloor(r-1)/2\rfloor-1$, so that $1\leq J\leq r-4$.
Extend $\widetilde{\mathcal J}_{r-2}$ to nested interval partitions
$\widetilde{\mathcal J}_j$ such that
\begin{equation}
 |I|\sim  \prod_{k=1}^j\widetilde\delta_k
 \quad(I\in\widetilde{\mathcal J}_j),\qquad 0\leq j\leq r-2.
\end{equation}
Put
\begin{equation}
 g_I:=\sum_{\substack{\tau\in\widetilde{\mathcal J}_{r-2}\\\tau\subset I}}g_\tau,
 \qquad I\in\widetilde{\mathcal J}_j,\quad 0\leq j\leq r-2.
\end{equation}

We record the Fourier support and physical scales used in both parts below.
The nestedness property gives  $g_I=\sum_{\substack{I'\subset I:
                        I'\in\widetilde{\mathcal J}_j}}g_{I'}$ for any $I\in\widetilde{\mathcal J}_{j-1},\ 1\leq j\leq r-2$.
For $0\leq j\leq r-2$, the tuple truncated at $j$ has side lengths 
\begin{equation}
 \begin{gathered}
 \widetilde\rho_k^{(j)}
 :=\prod_{b=1}^{\min(j,k)}\widetilde\delta_b^{k+1-b}, \qquad 
 \widetilde\rho_k^{(j)}
 \begin{cases}
 =\widetilde\rho_k,&0\leq k\leq j,\\
 \geq\widetilde\rho_k,&j<k\leq n.
 \end{cases}
 \end{gathered}
\end{equation}
for $0 \leq k \leq n$.
Consequently, for $1 \leq k \leq n$,
\begin{equation}
 \begin{gathered}
 \frac{\widetilde\rho_k^{(j)}}{\widetilde\rho_{k-1}^{(j)}}
 =\prod_{b=1}^{\min(j,k)}\widetilde\delta_b
 \geq\prod_{b=1}^j\widetilde\delta_b,
 \end{gathered}
 \end{equation}
 and for $0 \leq i \leq k \leq n$,
 \begin{equation}\begin{gathered}
 \widetilde\rho_i^{(j)}
 \left(\prod_{b=1}^j\widetilde\delta_b\right)^{k-i}
 \leq\widetilde\rho_k^{(j)}.
 \end{gathered}
\end{equation}
Fix $0\leq j\leq r-2$ and $I\in\widetilde{\mathcal J}_j$, and let
$s_I$ be its center. For $s\in I$ and
$\zeta\in\operatorname{supp}\widehat g_{s,\ell,\alpha}$, apply
Lemma~\ref{prelim:change-of-frame} with widths
$\widetilde\rho_k^{(j)}$ and
$|s-s_I|\leq c\prod_{b=1}^j\widetilde\delta_b$,
using the same small constant $c$. This gives
\begin{equation}\label{c13all:truncated-frame-change}
 |\lambda_k'-\lambda_k|\lesssim c\widetilde\rho_k^{(j)}
 \qquad(0\leq k\leq n),
\end{equation}
where $\lambda_k$ and $\lambda_k'$ are the coefficients of $\zeta$
in the frames at $s$ and $s_I$, respectively. At each signed coordinate of the truncated tuple,
\begin{equation}
 1\leq k\leq j,\quad\nu_k\ne0
 \quad\Longrightarrow\quad
 \widetilde\rho_{k-1}^{(j)}=\widetilde\rho_{k-1}.
\end{equation}
Taking $c$ sufficiently small in \eqref{c13all:truncated-frame-change}
therefore preserves the signed lower bounds and gives
\begin{equation}\label{c13all:grouped-truncated-support}
 \operatorname{supp}\widehat g_I\subset
 \left\{\zeta:
 \begin{array}{ll}
 |\zeta\cdot\mathbf e_{k+1}(s_I)|\leq C\widetilde\rho_k^{(j)},
     &0\leq k\leq n,\\
 \nu_{k+1}\zeta\cdot\mathbf e_{k+1}(s_I)\geq c\widetilde\rho_k^{(j)},
     &0\leq k<j,\ \nu_{k+1}\ne0
 \end{array}\right\}.
\end{equation}
For the signed-plank estimates, use the tuple
$(\widetilde\delta_1,\ldots,\widetilde\delta_j,1,\ldots,1)$,
replacing its first entry by $1$ when $\nu_1=0$, as above.
With signs $(\nu_1,\ldots,\nu_j,0,\ldots,0)$,
the support in \eqref{c13all:grouped-truncated-support} lies in a fixed
admissible enlargement of the corresponding plank \eqref{red:plates}.
The original physical scales satisfy
\begin{equation}
 \widetilde R_j:=(\widetilde\rho_n^{(j)})^{-1}
 \leq\widetilde\rho_n^{-1}\leq \ell R\leq R,
\end{equation}
so every factor $\widetilde R_j^\varepsilon$ is at most $R^\varepsilon$.
The estimates following \eqref{c13all:annulus-weight} show that, after
a fixed normalization, $W_\ell$ is an $n$-dimensional weight.

\subsubsection{Applying Proposition~\ref{c13all:prefix-transfer}}

We first prove
\begin{equation}\label{c13all:weighted-prefix-in-annulus}
 \int\Big|\sum_{\tau\in\widetilde{\mathcal J}_{r-2}} g_\tau\Big|^rW_\ell
 \lesssim  R^\varepsilon
 \prod_{i=1}^J\widetilde\delta_i^{-(i-1)}
             \int\Big(\sum_{I\in\widetilde{\mathcal J}_J}|g_I|^{r/2}\Big)^2.
\end{equation}
Using the induction hypothesis, apply the global consequence of
Proposition~\ref{c13all:prefix-transfer} with moment exponent $r-1$ and
$p=r$ to $\{g_I:I\in\widetilde{\mathcal J}_J\}$, indexed by $s_I$,
with weight $W_\ell$.
Since $\sum_{I\in\widetilde{\mathcal J}_J}g_I=\sum_{\tau\in\widetilde{\mathcal J}_{r-2}}g_\tau$,
this gives \eqref{c13all:weighted-prefix-in-annulus}.

\subsubsection{Estimating the mixed norm}

By \eqref{c13all:weighted-prefix-in-annulus}, it suffices to prove
\begin{equation}
 \begin{aligned}
 &\int\Big(\sum_{I\in\widetilde{\mathcal J}_J}|g_I|^{r/2}\Big)^2\lesssim R^\varepsilon
 \bigl(\widetilde\delta/\delta\bigr)^{r-4}
 \prod_{j=J+1}^{r-2}\widetilde\delta_j^{-(j-1)}
 \int\Big(\sum_{\tau\in\widetilde{\mathcal J}_{r-2}}|g_\tau|^2\Big)^2.
 \end{aligned}
\end{equation}

To stay within the range $4\leq q\leq(n+1-j)(n+2-j)-2$ in
\eqref{c13all:unweighted-stage}, set
\begin{equation}
 p_J:=r,\qquad
 p_j:=\min\{r,(n+1-j)(n+2-j)-2\}\quad(J<j\leq r-2).
\end{equation}
Thus $4\leq p_j\leq p_{j-1}$ for $J<j\leq r-2$; the exponent is
lowered only when required by this range.
By \eqref{c13all:coarse-amplitude}, for $0\leq j\leq r-2$ and
$I\in\widetilde{\mathcal J}_j$,
\begin{equation}
 \|g_I\|_\infty
 \lesssim \frac{\widetilde\delta}{\delta}
 \#\{\tau\in\widetilde{\mathcal J}_{r-2}:\tau\subset I\}
 \lesssim \frac{\widetilde\delta}{\delta}
 \prod_{k=j+1}^{r-2}\widetilde\delta_k^{-1}.
\end{equation}
Consequently, for $p\geq p'\geq4$,
\begin{equation}\label{c13all:mixed-drop}
 \int\Big(\sum_{I\in\widetilde{\mathcal J}_j}|g_I|^{p/2}\Big)^2
 \lesssim
 \bigl(\frac{\widetilde\delta}{\delta}\prod_{k=j+1}^{r-2}\widetilde\delta_k^{-1}\bigr)^{p-p'}
 \int\Big(\sum_{I\in\widetilde{\mathcal J}_j}|g_I|^{p'/2}\Big)^2.
\end{equation}
For $J<j\leq r-2$, apply \eqref{c13all:mixed-drop} at scale $j-1$
and then \eqref{c13all:unweighted-stage} with $q=p_j$, using the support
condition \eqref{c13all:grouped-truncated-support}, to obtain
\begin{equation}\label{c13all:capped-mixed-recursion}
 \begin{aligned}
 \int\Big(\sum_{I\in\widetilde{\mathcal J}_{j-1}}|g_I|^{p_{j-1}/2}\Big)^2
 &\lesssim  R^\varepsilon
 \bigl(\frac{\widetilde\delta}{\delta}\prod_{k=j}^{r-2}\widetilde\delta_k^{-1}\bigr)^{p_{j-1}-p_j}
 \widetilde\delta_j^{-(p_j/2-2)}\int\Big(\sum_{I\in\widetilde{\mathcal J}_j}|g_I|^{p_j/2}\Big)^2.
 \end{aligned}
\end{equation}
Iterating from $p_J=r$, then applying \eqref{c13all:mixed-drop} at
$j=r-2$ with $p=p_{r-2}$ and $p'=4$, gives
\begin{equation}
 \int\Big(\sum_{I\in\widetilde{\mathcal J}_J}|g_I|^{r/2}\Big)^2
 \lesssim R^\varepsilon\bigl(\widetilde\delta/\delta\bigr)^{r-4}
 \prod_{k=J+1}^{r-2}\widetilde\delta_k^{-(k-1)}
 \int\Big(\sum_{\tau\in\widetilde{\mathcal J}_{r-2}}|g_\tau|^2\Big)^2.
\end{equation}
Here the power of $\widetilde\delta/\delta$ is
$\sum_{j=J+1}^{r-2}(p_{j-1}-p_j)+(p_{r-2}-4)=r-4$,
and, for $J<k\leq r-2$,
\begin{equation}\label{c13all:capped-exponent}
 \frac{p_k}{2}-2+\sum_{j=J+1}^k(p_{j-1}-p_j)
 =r-2-\frac{p_k}{2}\leq k-1.
\end{equation}
If $p_k=r$, then $r-2-p_k/2=r/2-2\leq J\leq k-1$.
Otherwise, $p_k=(n+1-k)(n+2-k)-2$, so
$p_k/2\geq n-k-1\geq r-k-1$.
Applying the mixed-norm bound and
\eqref{c13all:weighted-prefix-in-annulus} with $\varepsilon/2$ proves
\eqref{c13all:annular-square-estimate}.
\qed

\subsection{The square function}\label{sec:positive-kernels}

\begin{lemma}\label{lem:positive-square-charge}
For the angular sums $g_\tau$ defined in
Section~\ref{sec:positive-angular-grouping} and every $\varepsilon>0$,
\begin{equation}\label{c13all:coarse-amplitude-charge}
 \int\Big(\sum_{\tau\in\widetilde{\mathcal J}_{r-2}}|g_\tau|^2\Big)^2
 \lesssim R^\varepsilon\bigl(\frac{\widetilde\delta}{\delta}\bigr)^2
 \int\Big|\sum_{s\in\mathbb I_\delta}g_s(y/\ell)\Big|^2\,\mathrm{d}y.
\end{equation}
\end{lemma}

\begin{proof}[Proof of Lemma~\ref{lem:positive-square-charge}]
Fix the dyadic scale $\ell$ and class $\alpha$ from
Section~\ref{sec:positive-decomposition}.
It suffices to construct kernels $K_s\geq0$, with $\int K_s=1$, such that
\begin{equation}
 \begin{gathered}
 |g_{s,\ell,\alpha}|\lesssim g_s(\cdot/\ell)*K_s
 \qquad(s\in\mathbb I_\delta),\\
 \int\Big|\sum_{s\in\mathbb I_\delta}g_s(\cdot/\ell)*K_s\Big|^2
 \lesssim R^\varepsilon
 \int\Big|\sum_{s\in\mathbb I_\delta}g_s(y/\ell)\Big|^2\,\mathrm{d}y.
 \end{gathered}
\end{equation}
Given such kernels, \eqref{c13all:coarse-amplitude} gives
\begin{equation}
 \begin{aligned}
 \sum_{\tau\in\widetilde{\mathcal J}_{r-2}}|g_\tau|^2
 \lesssim\frac{\widetilde\delta}{\delta}
       \sum_{\tau\in\widetilde{\mathcal J}_{r-2}}|g_\tau| \lesssim\frac{\widetilde\delta}{\delta}
       \sum_{s\in\mathbb I_\delta}g_s(\cdot/\ell)*K_s.
 \end{aligned}
\end{equation}
Squaring, integrating, and using the second kernel estimate gives
\begin{equation}
 \int\Big(\sum_{\tau\in\widetilde{\mathcal J}_{r-2}}|g_\tau|^2\Big)^2
 \lesssim\bigl(\frac{\widetilde\delta}{\delta}\bigr)^2
 \int\Big|\sum_{s\in\mathbb I_\delta}g_s(\cdot/\ell)*K_s\Big|^2
 \lesssim R^\varepsilon\bigl(\frac{\widetilde\delta}{\delta}\bigr)^2
 \int\Big|\sum_{s\in\mathbb I_\delta}g_s(y/\ell)\Big|^2\,\mathrm{d}y.
\end{equation}
This proves \eqref{c13all:coarse-amplitude-charge} once the kernels are available.
\medskip

It remains to construct these kernels and verify the two estimates. Define
\begin{equation}
 \varrho_j:=
 \begin{cases}
  \widetilde\rho_j,&0\leq j\leq r-2,\\
  \rho_j/\ell,&r-2<j\leq n.
 \end{cases}
\end{equation}
The corresponding  planks and their dual boxes are
\begin{equation}
 \begin{aligned}
 q'_s&:=\left\{\sum_{j=0}^n u_j\mathbf e_{j+1}(s):
                         |u_j|\leq\varrho_j,\ 0\leq j\leq n\right\},\\
 (q'_s)^*&:=\left\{x:|x\cdot\mathbf e_{j+1}(s)|
                   \leq\varrho_j^{-1},\ 0\leq j\leq n\right\}.
 \end{aligned}
\end{equation}
By \eqref{c13all:hull-tuple}--\eqref{c13all:hull-slopes}, the ratios
$\varrho_j/\varrho_{j-1}$ decrease, and
\begin{equation}
 \frac{\varrho_n}{\varrho_{n-1}}
 =\frac{\varrho_{n-1}}{\varrho_{n-2}}=\delta,
 \qquad \varrho_n^{-1}=\ell R\leq R.
\end{equation}
The corresponding $\de$-sequence is
\begin{equation}
 \delta'_j:=
 \begin{cases}
  \widetilde\delta_j,&1\leq j\leq r-2,\\
  \rho_{r-1}/(\widetilde\delta\rho_{r-2}),&j=r-1,\\
  \delta_j,&r\leq j\leq n-1.
 \end{cases}
\end{equation}
Since $\rho_{r-1}/\rho_{r-2}\leq \widetilde\delta$, these scales satisfy
\begin{equation}
 \begin{gathered}
 0<\delta'_j\leq1\quad(1\leq j\leq n-1),\qquad
 \prod_{j=1}^{n-1}\delta'_j
 =\frac{\varrho_{n-1}}{\varrho_{n-2}}=\delta,\\
 \varrho_j=\prod_{i=1}^{\min(j,n-1)}
                  (\delta'_i)^{j+1-i}\quad(0\leq j\leq n).
 \end{gathered}
\end{equation}
Thus Lemma~\ref{c13all:kernel-removal} applies with $q_s$ replaced by $q'_s$,
angular separation $\delta$ and physical scale $\ell R$.

By the definition of $g_{s,\ell,\alpha}$ in
Section~\ref{sec:positive-signed-decomposition}, we have
\begin{equation}
 g_{s,\ell,\alpha}=g_s(\cdot/\ell)*k_s,
\end{equation}
where $\wh k_s:=\chi_{\alpha,s}(\zeta)\eta_\ell(\ell\zeta)$ is a smooth bump adapted to $q'_s$.
Integration by parts and the weights from
Section~\ref{sec:notations} give nonnegative kernels $K_s$ adapted to
$(q'_s)^*$, with $\int K_s=1$,
$\operatorname{supp}\widehat K_s\subset Cq'_s$ and $|k_s|\lesssim K_s$.
Since $g_s\geq0$, we obtain
\begin{equation}
 |g_{s,\ell,\alpha}|
 \leq g_s(\cdot/\ell)*|k_s|
 \lesssim g_s(\cdot/\ell)*K_s.
\end{equation}
This establishes the required pointwise bound.
Finally, Lemma~\ref{c13all:kernel-removal}, applied with $b_s=g_s(\cdot/\ell)$,
gives the required second estimate. 
This completes the construction and the proof.
\end{proof}

\section{Proof of Theorem~\ref{thm:translated-plates}}\label{sec:completion}

We prove Theorem~\ref{thm:translated-plates} with
$\eta_p:=\min\{p-n,1\}/(32np)$.
To simplify the notation, write the summands in \eqref{red:translated-form} as
\begin{equation}\label{c13all:normalized}
 \mathsf{T}g(z)=\sum_s\mathsf{T}_sg(z), \qquad \mathsf{T}_sg(z)=a_s[\widehat g(\xi)\phi_{p_s}(\xi,\tau)]^\vee
                      (z-Ru_s\Gamma(s)).
\end{equation}

Fix
$B=B((x_B,t_B),R)$ with $|t_B|\lesssim R$, and put
$B^x:=B(x_B,R)\subset\mathbb R^n$. Let us prove
\begin{equation}\label{c13all:restricted-moment}
 \|\mathsf{T}g\|_{L^p(E\cap B)}^p
 \lesssim  R^{-p+\varepsilon}\delta^{-n}
           \|g\|_\infty^{p-4}\|g\|_{L^4(W_{B^x})}^4,
 \qquad n<p\leq n+1.
\end{equation}
Once \eqref{c13all:restricted-moment} is proved, for $n \geq 4$, a standard 
interpolation argument yields
\begin{equation}
 \|\mathsf{T}g\|_{L^p(E)}
 \lesssim\delta^{-1}R^{-1-\eta_p}\|g\|_p,
 \qquad  p>n.
\end{equation}
Thus \eqref{red:unit-estimate} holds for $n\geq4$;
see Remark~\ref{c13all:lowdim} for $n=2,3$.
\medskip

To prove \eqref{c13all:restricted-moment}, fix $n<p\leq n+1$
and set $J:=\lfloor n/2\rfloor-1$.
Recall  $\mathcal J_j$, $0\leq j\leq n-1$,
in Section~\ref{sec:scales-angular}. 
Choose a function $\chi_B$ such that 
\begin{equation}\label{completion:common-cutoff}
 |\chi_B|\geq1\text{ on }B,\qquad |\chi_B|\lesssim W_B,\qquad
 \operatorname{supp}\widehat{\chi_B}\subset B(0,c/R).
\end{equation}
The proof combines three estimates. We start with
\begin{equation}\label{completion:first-scales}
 \|\mathsf{T}g\|_{L^p(E\cap B)}^p
 \lesssim  R^\varepsilon
       \prod_{i=1}^J\delta_i^{-(i-1)}\int|\chi_B|^{p}
 \Big(\sum_{I\in\mathcal J_{J}}
       \Big|\sum_{s\in\mathbb I_\delta\cap I}\mathsf{T}_sg\Big|^{p/2}\Big)^2.
\end{equation}
For the remaining  scales, we prove
\begin{equation}\label{completion:remaining-scales}
 \begin{aligned}
 &\int|\chi_B|^p
 \Big(\sum_{I\in\mathcal J_J}
       \Big|\sum_{s\in\mathbb I_\delta\cap I}\mathsf{T}_sg\Big|^{p/2}\Big)^2\\
 &\quad\lesssim R^\varepsilon
       \bigl(R^{-1}\|g\|_\infty\bigr)^{p-4}
       \prod_{j=J+1}^{n-1}\delta_j^{-(j-1)}
       \int|\chi_B|^4
       \Big(\sum_{s\in\mathbb I_\delta}|\mathsf{T}_sg|^2\Big)^2.
 \end{aligned}
\end{equation}
The final estimate is
\begin{equation}\label{completion:last-scales}
 \int|\chi_B|^4\Big(\sum_{s\in\mathbb{I}_\delta}|\mathsf{T}_sg|^2\Big)^2
 \lesssim  R^{-3+\varepsilon}
                         \|g\|_{L^4(W_{B^x})}^4.
\end{equation}
Since $\prod_{j=1}^{n-1}\delta_j^{-(j-1)}=\delta^{-n}R^{-1}$,
these three estimates give
\begin{equation}
 \begin{aligned}
 \|\mathsf{T}g\|_{L^p(E\cap B)}^p
 &\lesssim R^\varepsilon(\delta^{-n}R^{-1})
       \bigl(R^{-1}\|g\|_\infty\bigr)^{p-4}
       R^{-3}\|g\|_{L^4(W_{B^x})}^4\\
 &=R^{-p+\varepsilon}\delta^{-n}
       \|g\|_\infty^{p-4}\|g\|_{L^4(W_{B^x})}^4,
 \end{aligned}
\end{equation}
which is \eqref{c13all:restricted-moment}. In the next subsections, we prove the three estimates.

\subsection{Applying Corollary~\ref{cor:weighted-oscillatory-sums}}\label{sec:completion-first-scales}
We prove \eqref{completion:first-scales}.

We use the admissible scale tuple
$(\delta_1,\ldots,\delta_J,1,\ldots,1)$ at scale $R_J\leq R$,
as in Section~\ref{sec:frame-planks}.
By \eqref{completion:common-cutoff} and
\cite[Lemma~3.10 and Section~9.1]{GMO2026}, for every
$I\in\mathcal J_j$ and $0\leq j\leq n-1$,
$\chi_B\sum_{s\in\mathbb I_\delta\cap I}\mathsf{T}_sg$ has Fourier support
in a fixed admissible enlargement of the truncated signed plank
\eqref{red:plates} with widths $\rho_l^{(j)}$ at the center of $I$.
By \eqref{conj:katz-tao}, $w:=\mathbf1_E$ satisfies
\eqref{c13all:weight}.
Apply Corollary~\ref{cor:weighted-oscillatory-sums}
at scale $R_J$ to $\{\chi_B\sum_{s\in\mathbb I_\delta\cap I}\mathsf{T}_sg:I\in\mathcal J_J\}$.
Since $R_J\leq R$, we obtain
\begin{equation}
 \|\mathsf{T}g\|_{L^p(E\cap B)}^p
 \lesssim\int\Big|\sum_{I\in\mathcal J_J}\chi_B\sum_{s\in\mathbb I_\delta\cap I}\mathsf{T}_sg\Big|^p w
 \lesssim R^\varepsilon\prod_{i=1}^J\delta_i^{-(i-1)}
 \int\Big(\sum_{I\in\mathcal J_J}|\chi_B\sum_{s\in\mathbb I_\delta\cap I}\mathsf{T}_sg|^{p/2}\Big)^2.
\end{equation}
This is \eqref{completion:first-scales}.

\subsection{Applying Proposition \ref{prop:gmo-unweighted}}
We prove \eqref{completion:remaining-scales} by iterating
\eqref{c13all:unweighted-stage} and lowering the exponent from $p$ to $4$.
Let $K_s:=(\phi_{p_s})^\vee$. Integration by parts using
\eqref{red:adapted} and integration in the frame coordinates give the first
bound below, since the Jacobian on each spatial slice is
$|\mathbf e_{n+1}(s)\cdot(0,\ldots,0,1)|=(1+|\gamma(s)|^2)^{-1/2}\gtrsim1$.
Young's inequality and $\delta$-separation give the other two:
\begin{equation}\label{c13all:kernel-bound}
 \begin{gathered}
 \sup_{t\in\mathbb R}\|K_s(\cdot,t)\|_1
 \lesssim\Big(\prod_{l=0}^n\rho_l\Big)
          \prod_{l=0}^{n-1}\rho_l^{-1}
 =R^{-1},\\
 \sup_{s\in\mathbb I_\delta}\|\mathsf{T}_sg\|_\infty
 \lesssim R^{-1}\|g\|_\infty,\qquad
 \|\mathsf{T}\|_{\infty\to\infty}\lesssim\delta^{-1}R^{-1}.
 \end{gathered}
\end{equation}

Set
\begin{equation}
 p_J:=p,\qquad
 p_j:=\min\{p,(n+1-j)(n+2-j)-2\},\quad J<j\leq n-1.
\end{equation}
Then $4\leq p_j\leq p_{j-1}$ and $p_{n-1}=4$.
Separation and \eqref{c13all:kernel-bound} give, for
$I\in\mathcal J_{j-1}$ and $J<j\leq n-1$,
\begin{equation}
 \Big\|\chi_B\sum_{s\in\mathbb I_\delta\cap I}\mathsf{T}_sg\Big\|_\infty
 \lesssim R^{-1}\|g\|_\infty
                 \prod_{k=j}^{n-1}\delta_k^{-1}.
\end{equation}
Using the support condition from Section~\ref{sec:completion-first-scales},
lower the exponent with this bound and apply \eqref{c13all:unweighted-stage}
at the  exponent $p_j$:
\begin{equation}\label{c13all:actual-drop}
 \begin{aligned}
 &\int|\chi_B|^{p_{j-1}}
 \Big(\sum_{I\in\mathcal J_{j-1}}
       \Big|\sum_{s\in\mathbb I_\delta\cap I}\mathsf{T}_sg\Big|^{p_{j-1}/2}\Big)^2\\
 &\quad\lesssim R^\varepsilon
 \bigl(R^{-1}\|g\|_\infty
             \prod_{k=j}^{n-1}\delta_k^{-1}\bigr)^{p_{j-1}-p_j}
 \delta_j^{-(p_j/2-2)}
 \int|\chi_B|^{p_j}
 \Big(\sum_{I\in\mathcal J_j}
       \Big|\sum_{s\in\mathbb I_\delta\cap I}\mathsf{T}_sg\Big|^{p_j/2}\Big)^2.
 \end{aligned}
\end{equation}
Iterating \eqref{c13all:actual-drop}, with
$\sum_{j=J+1}^{n-1}(p_{j-1}-p_j)=p-4$ and
$\beta_j:=p-2-p_j/2$ for $J<j\leq n-1$, gives
\begin{equation}
 \begin{aligned}
 &\int|\chi_B|^p
 \Big(\sum_{I\in\mathcal J_J}
       \Big|\sum_{s\in\mathbb I_\delta\cap I}\mathsf{T}_sg\Big|^{p/2}\Big)^2\\
 &\quad\lesssim R^\varepsilon
       \bigl(R^{-1}\|g\|_\infty\bigr)^{p-4}
       \prod_{j=J+1}^{n-1}\delta_j^{-\beta_j}
       \int|\chi_B|^4
       \Big(\sum_{s\in\mathbb I_\delta}|\mathsf{T}_sg|^2\Big)^2\\
 &\quad\leq R^\varepsilon
       \bigl(R^{-1}\|g\|_\infty\bigr)^{p-4}
       \prod_{j=J+1}^{n-1}\delta_j^{-(j-1)}
       \int|\chi_B|^4
       \Big(\sum_{s\in\mathbb I_\delta}|\mathsf{T}_sg|^2\Big)^2.
 \end{aligned}
\end{equation}
For the last inequality, if $p_j=p$, then
$\beta_j=p/2-2\leq(n-3)/2\leq J\leq j-1$.
Otherwise,
$\beta_j-(j-1)=p-n-1-(n+1-j)(n-j)/2\leq0$ since $p\leq n+1$.
This proves \eqref{completion:remaining-scales}.

\subsection{The last estimate}
We prove \eqref{completion:last-scales}. For $n\geq2$, the square function
estimate \cite[Proposition~9.4]{GMO2026} applies to the operator of
Section~\ref{sec:frame-planks} and balls with $|t_B|\lesssim R$:
\begin{equation}\label{c13all:operator-four}
 \int\Big(\sum_{s\in\mathbb{I}_\delta}|\mathsf{T}_sg|^2\Big)^2W_B
 \lesssim R^{-3+\varepsilon}\|g\|_{L^4(W_{B^x})}^4.
\end{equation}
Since $|\chi_B|^4 \lesssim W_B$ by \eqref{completion:common-cutoff},
\eqref{c13all:operator-four} immediately gives \eqref{completion:last-scales}.

\begin{remark}[The dimensions two and three]\label{c13all:lowdim}
Let $n=2,3$ and
$\operatorname{supp}\widehat g\subset\{\xi\in\mathbb R^n:1/2\leq|\xi|\leq2\}$.
For the fixed admissible enlargements, the projected support estimates in
\cite[Section~9.3, equations~(9.57)--(9.61)]{GMO2026} and Plancherel give
the $L^2$ bound below, while \cite[Theorem~5.3 and Proposition~9.5]{GMO2026},
together with \eqref{completion:common-cutoff} and \eqref{c13all:operator-four},
give the $L^4$ bound:
\begin{equation}
 \begin{aligned}
 \|\mathsf{T}g\|_{L^2(\mathbb R^{n+1})}
 &\lesssim R^{-1/2}\|g\|_2,\\
 \|\mathsf{T}g\|_{L^4(\mathbb R^n\times[R,2R])}
 &\lesssim R^{-3/4+\varepsilon}\|g\|_4.
 \end{aligned}
\end{equation}
Interpolation and $\delta\leq R^{-1/n}$ from \eqref{red:scale-equation} give
\begin{equation}
 \|\delta R\cdot\mathsf{T}g\|_{L^p(\mathbb R^n\times[R,2R])}
 \lesssim R^{1/p-1/n+\varepsilon}\|g\|_p,\qquad 2\leq p\leq4.
\end{equation}
Taking $\varepsilon$ small and interpolating the $p=4$ bound with
\eqref{c13all:kernel-bound} for $p>4$ yields \eqref{red:unit-estimate}:
\begin{equation}
 \|\mathsf{T}g\|_{L^p(E)}
 \lesssim\delta^{-1}R^{-1-\eta_p}\|g\|_p,\qquad p>n.
\end{equation}
\end{remark}

\newpage

\appendix
\renewcommand{\theequation}{\thesection.\arabic{equation}}
\section{Proof of the refined decoupling}\label{app:refined}

We prove Proposition~\ref{c13all:rp} in the appendix. Fix $n\geq2$,
$1\leq J\leq n-1$, $2\leq p\leq(n+1-J)(n+2-J)$, and $\varepsilon>0$. 
For the data in Proposition~\ref{c13all:rp}, our goal is
\begin{equation}
 \left\|\sum_{T\in\mathbb T}f_T\right\|_{L^p(Y)}
 \lesssim\delta^{-\varepsilon}M^{1/2-1/p}
       \left(\sum_{T\in\mathbb T}\|f_T\|_p^p\right)^{1/p}.
\end{equation}
Choose $a_0>0$ sufficiently small and put $a:=a_0/2$.
Use the admissible scales $0<\delta_j\leq1$, with $\delta_j=1$ for
$J<j\leq n-1$. By \eqref{red:scales},
\begin{equation}\label{app:refined-initial-scales}
 \rho_0=1,\qquad
 \rho_j=\prod_{\ell=1}^{\min(j,J)}\delta_\ell^{j+1-\ell}
 \quad(1\leq j\leq n),\qquad
 \delta=\prod_{j=1}^J\delta_j,\qquad \rho_n=R^{-1}.
\end{equation}

\subsection*{The change of coordinates}
Put $v:=\mathbf e_1$ and write
$\xi=\sum_{j=0}^n\lambda_jv^{(j)}(s)$ for $\xi$ in a fixed admissible
enlargement of $p_s$, $s\in[-1,1]$. We treat positive signs; the other cases
are analogous. The Frenet equations and $\rho_0\geq\cdots\geq\rho_n$ give,
for a fixed sufficiently small $0<\tau<1/2$,
\begin{equation}\label{app:refined-signed-change}
 \begin{gathered}
 |\lambda_j|\lesssim\rho_j\quad(0\leq j\leq n),\\
 \delta_\ell<\tau\quad\Longrightarrow\quad
 \lambda_{\ell-1}\sim\rho_{\ell-1}\quad(1\leq\ell\leq J),
 \end{gathered}
\end{equation}
Recall the wave packet $T$ centered at $x_T$ with parameter $s_T$, and define
the box $T^{(0)}$ in derivative coordinates by
\begin{equation}\label{app:refined-box-definitions}
 \begin{aligned}
 T&=\{x\in\mathbb R^{n+1}:
 |(x-x_T)\cdot\mathbf e_{j+1}(s_T)|\leq\rho_j^{-1},\ 0\leq j\leq n\},\\
 T^{(0)}&:=\{x\in\mathbb R^{n+1}:
 |(x-x_T)\cdot v^{(j)}(s_T)|\leq\rho_j^{-1},\ 0\leq j\leq n\}.
 \end{aligned}
\end{equation}
For $\delta$ sufficiently small in terms of $a$,
\begin{equation}\label{app:refined-box-comparison}
 T^{(0)}\subset CT,\qquad T\subset CT^{(0)},\qquad
 10\delta^{-a}T^{(0)}\subset10\delta^{-a_0}T.
\end{equation}
We henceforth switch to the derivative coordinates and unambiguously write $T$ for $T^{(0)}$.

After a finite partition and rescaling, take $v$ in a sufficiently small
$C^L$ neighborhood of $(1,s,\ldots,s^n/n!)$, with $L$ large.
For $1\leq k\leq J$ and $0<h<\tau$, partition the parameter interval into
$\mathcal I_h$, with intervals of length $h$ and centers $s_I$.

We will use the following decoupling inequality as a black box. 
Suppose that
\begin{equation}\label{app:refined-coarse-support}
 \operatorname{supp}\widehat F_I\subset
 \left\{\sum_{j=0}^n\lambda_jv^{(j)}(s_I):
 \begin{array}{ll}
 |\lambda_j|\lesssim1,&0\leq j<k,\\
 |\lambda_j|\lesssim h^{j+1-k},&k\leq j\leq n,\\
 \lambda_{k-1}\gtrsim1,&
 \end{array}\right\}.
\end{equation}
Then we have
\begin{equation}\label{c13all:ordinary-flat-decoupling}
 \left\|\sum_{I\in\mathcal I_h} F_I\right\|_p
 \lesssim  h^{-\eta}
       \left(\sum_{I\in\mathcal I_h}\|F_I\|_p^2\right)^{1/2},
 \qquad 2\leq p\leq(n+1-k)(n+2-k),
\end{equation}
for every $\eta>0$.
This is the $\ell^2$ estimate in
\cite[Section~7.4, proof of Theorem~4.4]{BGHS2021}, before H\"older's
inequality, in the reverse Frenet frame with dimension parameter
$n+1-k$, scale $h$, and ambient dimension $n+1$.

\medskip

We discuss a notation used in the proof. It is called \textit{parameter tuple} denoted by
$\mathcal P=(k;d_k,\ldots,d_J)$, where $1\leq k\leq J$ is the first remaining index and $0<d_j\leq1$ for $k\leq j\leq J$.
The associated widths $r_j$ and angular scale $d$ satisfy
\begin{equation}\label{c13all:rp-thin-class}
 \begin{split}
 &0<r_0\leq\cdots\leq r_{k-2}\leq1,
       \qquad r_{k-1}=1,\\
 &r_j:=\prod_{\ell=k}^{\min(j,J)}
             d_\ell^{\,j+1-\ell}\quad(k\leq j\leq n),
       \qquad d=\prod_{\ell=k}^J d_\ell.
 \end{split}
\end{equation}
The first string is absent when $k=1$. The initial parameter tuple is
$\mathcal P_0=(1;\delta_1,\ldots,\delta_J)$, for which $d_j=\delta_j$,
$d=\delta$, and $r_j=\rho_j$ by \eqref{app:refined-initial-scales}.
We also allow the terminal parameter tuple $(J+1;)$, with no remaining
scales and $d=1$.

\subsection*{Step 1. Rescaling}
Fix a parameter tuple $\mathcal P=(k;d_k,\ldots,d_J)$ with $1\leq k\leq J$,
and widths $r_j$ and angular scale $d$ satisfying \eqref{c13all:rp-thin-class}.
For each wave packet $f_T$, let $s_T$ be its angular parameter and
$x_T$ its center. Assume that every
$\xi=\sum_{j=0}^n\lambda_jv^{(j)}(s_T)\in\operatorname{supp}\widehat f_T$
satisfies
\begin{equation}\label{app:refined-induction-support}
 |\lambda_j|\lesssim r_j\quad(0\leq j\leq n),\qquad
 d_\ell<\tau\ \Longrightarrow\
 \lambda_{\ell-1}\sim r_{\ell-1}
 \quad(k\leq\ell\leq J).
\end{equation}
Here $0<\tau<1/2$ is the fixed sufficiently small constant in
\eqref{app:refined-signed-change}, independent of the parameter tuple.
The lower bound $\lambda_{\ell-1}\gtrsim r_{\ell-1}$ is imposed only
when $d_\ell<\tau$.
At the initial parameter tuple, \eqref{app:refined-signed-change} gives
\eqref{app:refined-induction-support}. The corresponding physical box is
\begin{equation}\label{app:refined-current-box}
 T=\{x: |(x-x_T)\cdot v^{(j)}(s_T)|\leq r_j^{-1},
                      \ 0\leq j\leq n\}.
\end{equation}

Let $d_k\leq h\leq1$.
The following lemma gives a change of variables for these wave
packets, with new angular scale $d/h$.
\begin{lemma}[Rescaling]\label{app:refined-rescaling-lemma}
Let $d_k\leq h\leq1$ and let $I$ be an interval of length $h$ centered
at $s_I$. For $s_T\in I$, put
$u_T:=(s_T-s_I)/h$.
Put
\begin{equation}
 \begin{aligned}
 A_I&:=h^{k-1}\operatorname{diag}(1,h^{-1},\ldots,h^{-n})
       [v(s_I),v'(s_I),\ldots,v^{(n)}(s_I)]^{-1},\\
 v_I(u)&:=h^{1-k}A_Iv(s_I+hu).
 \end{aligned}
\end{equation}
Under the change of variables $x=A_I^Ty$, the wave packets
$f_T(A_I^T\,\cdot)$, $s_T\in I$, have angular parameters $u_T$.
Define the new scales
\begin{equation}\label{c13all:rp-rescale}
 \begin{gathered}
 r_j':=h^{k-1-j}r_j\quad(0\leq j\leq n),\\
 d':=d/h,\qquad d_k':=d_k/h.
 \end{gathered}
\end{equation}
Then they satisfy \eqref{c13all:rp-thin-class} and
\eqref{app:refined-induction-support} with $v,d_k,r_j$
replaced by $v_I,d_k',r_j'$. 
\end{lemma}

\begin{proof}
The chain rule gives
\begin{equation}\label{app:refined-rescaled-frame}
 \begin{gathered}
 v_I^{(j)}(u)=h^{j+1-k}A_Iv^{(j)}(s_I+hu)\quad(0\leq j\leq L),\\
 [v_I(0),v_I'(0),\ldots,v_I^{(n)}(0)]=\operatorname{Id}_{n+1}.
 \end{gathered}
\end{equation}
Write
$\gamma(u)=(1,u,u^2/2!,\ldots,u^n/n!)$ and assume
$\|v-\gamma\|_{C^L}\leq\epsilon_0$, with $L$ large and $\epsilon_0$
fixed and sufficiently small. Then
\begin{equation}
 \|[v(s),\ldots,v^{(n)}(s)]^{-1}\|\lesssim1,
 \qquad \|v^{(r)}\|_\infty\lesssim\epsilon_0
                         \quad(n+1\leq r\leq L).
\end{equation}
For $0\leq r\leq n$, Taylor's formula gives
\begin{equation}
 v^{(r)}(s_I+hu)
 =\sum_{j=r}^n\frac{(hu)^{j-r}}{(j-r)!}v^{(j)}(s_I)
       +O(\epsilon_0h^{n+1-r}),\qquad |u|\leq1.
\end{equation}
By the definition of $A_I$,
\[
 \begin{gathered}
 \|A_I\|\lesssim h^{k-1-n},\\
 h^{r+1-k}A_I\sum_{j=r}^n
 \frac{(hu)^{j-r}}{(j-r)!}v^{(j)}(s_I)=\gamma^{(r)}(u)
 \quad(0\leq r\leq n).
 \end{gathered}
\]
Thus \eqref{app:refined-rescaled-frame} and the Taylor remainder give
\[
 \begin{aligned}
 \|v_I^{(r)}-\gamma^{(r)}\|_\infty
 &\lesssim h^{r+1-k}\|A_I\|\,\epsilon_0h^{n+1-r}\\
 &\lesssim\epsilon_0h^{r-n}h^{n+1-r}
 =\epsilon_0h\qquad(0\leq r\leq n).
 \end{aligned}
\]
For $n<r\leq L$, we have $\gamma^{(r)}=0$, so
\[
 \|v_I^{(r)}-\gamma^{(r)}\|_\infty
 \lesssim h^{r+1-k}\|A_I\|\,\epsilon_0
 \lesssim\epsilon_0h^{r-n}\leq\epsilon_0h.
\]
Here all supremum norms are over $|u|\leq1$. We obtain
\begin{equation}\label{app:refined-curve-bound}
 \|v_I-\gamma\|_{C^L}\lesssim\epsilon_0h
 \qquad(0<h\leq1).
\end{equation}
\medskip

For the Fourier support, put $\xi'=A_I\xi$. Then, using \eqref{app:refined-induction-support} and the definition of $r_j'$ in \eqref{c13all:rp-rescale}, we obtain
\begin{equation}\label{app:refined-transformed-support}
 \begin{gathered}
 \mathcal F_y[f_T(A_I^T\,\cdot)](\xi')
   =|\det A_I|^{-1}\widehat f_T(A_I^{-1}\xi'),\\
 \xi=\sum_{j=0}^n\lambda_jv^{(j)}(s_T)
 \quad\Longrightarrow\quad
 \xi'=\sum_{j=0}^n\lambda_j'v_I^{(j)}(u_T),
 \qquad\lambda_j'=h^{k-1-j}\lambda_j\quad(0\leq j\leq n),\\
 |\lambda_j'|\lesssim h^{k-1-j}r_j=r_j',\qquad
 \frac{\lambda_j'}{r_j'}=\frac{\lambda_j}{r_j}\quad(0\leq j\leq n).
 \end{gathered}
\end{equation}
Since $d_k'<\tau$ implies $d_k=h d_k'<\tau$, and
$d_\ell$ is unchanged for $\ell>k$, the invariant ratios preserve
all lower bounds in \eqref{app:refined-induction-support}.

For the widths with index below $k-1$,
\begin{equation}
 \frac{r_j'}{r_{j+1}'}=h\frac{r_j}{r_{j+1}}\leq1
 \quad(0\leq j<k-2),\qquad
 r_{k-2}'=hr_{k-2}\leq1=r_{k-1}'
 \quad(k\geq2).
\end{equation}
For the remaining widths and the angular parameters,
\begin{equation}\label{app:refined-new-scales}
 \begin{gathered}
 r_j'=\left(\frac{d_k}{h}\right)^{j+1-k}
       \prod_{\ell=k+1}^{\min(j,J)}d_\ell^{j+1-\ell}
                \quad(k\leq j\leq n),\\
 d'=\frac{d_k}{h}\prod_{\ell=k+1}^J d_\ell
          =\frac{d}{h}\leq1,\qquad
 |u_T-u_{T'}|=h^{-1}|s_T-s_{T'}|\geq d'
                 \quad(s_T\ne s_{T'}).
 \end{gathered}
\end{equation}
These give \eqref{c13all:rp-thin-class} and $d'$-separation.

\end{proof}

\subsection*{The induction constant}
For $r:=J-k+1$ remaining parameters, fix constants
\begin{equation}\label{app:refined-enlargements}
 L_0\geq10,\qquad L_r\gg L_{r-1}
                       \quad(1\leq r\leq J).
\end{equation}
The implicit separation is chosen sufficiently large in terms of $n$
and the fixed $\tau$.

For a parameter tuple $\mathcal P=(k;d_k,\ldots,d_J)$, define
\begin{equation}\label{app:refined-dec-definition}
 \begin{gathered}
 d(\mathcal P):=\prod_{\ell=k}^J d_\ell, \qquad
 \text{Dec}(\mathcal P):=
 \sup
 \frac{\|\sum_{T\in\mathbb T}f_T\|_{L^p(Y)}}
 {M^{1/2-1/p}(\sum_{T\in\mathbb T}\|f_T\|_p^p)^{1/p}}.
 \end{gathered}
\end{equation}
The supremum is over the class of curves $\{v: \|v-\gamma\|_{C^L}\le \epsilon_0\}$, all widths in \eqref{c13all:rp-thin-class}, and wave packets with support
\eqref{app:refined-induction-support}. The distinct angular parameters are
$d$-separated, and the boxes at each fixed parameter belong
to a tiling with bounded overlap. We also take the supremum over
$M\geq1$ and unions $Y$ of unit boxes satisfying
\begin{equation}\label{app:refined-induction-overlap}
 \begin{gathered}
 \#\{T\in\mathbb T:10d^{-a}T\cap L_rU\ne\varnothing\}\leq M
            \quad\text{for every constituent unit box }U\subset Y.
 \end{gathered}
\end{equation}
This includes all thin widths $r_0,\ldots,r_{k-2}$.
For the terminal parameter tuple $(J+1;)$, use
\begin{equation}
 d(J+1;)=1,\qquad r=0,\qquad
 0<r_0\leq\cdots\leq r_{J-1}\leq1=r_J=\cdots=r_n.
\end{equation}
Lemma~\ref{app:refined-rescaling-lemma} gives the following two changes
of parameter tuple:
\begin{equation}\label{app:refined-new-tuples}
 \mathcal P'=\begin{cases}
 (k;d_k/h,d_{k+1},\ldots,d_J),&d_k<h<1,\\
 (k+1;d_{k+1},\ldots,d_J),&h=d_k,
 \end{cases}
 \qquad d(\mathcal P')=d/h.
\end{equation}
Indeed, when $h=d_k$, then $d_k/h=1$, so we may take $k+1$ as the first remaining index.

Choose $\kappa>0$ and then $d_*>0$ sufficiently small so that
\begin{equation}\label{app:refined-small-parameters}
 \begin{gathered}
 0<\kappa\ll a\ll\varepsilon,\qquad (4n+10)\kappa<a,\\
 n(1-\kappa)>n-1,\qquad J\kappa<\varepsilon/3,\qquad
 0<d_*<1,\quad d_*^\kappa<\tau<1/2.
 \end{gathered}
\end{equation}
Decrease $d_*$ further as required in
Lemmas~\ref{app:refined-localization-lemma}--\ref{app:refined-transition-lemma}.

Bounded overlap of wave packets at each fixed parameter $s$ gives
\begin{equation}
 \begin{aligned}
\left\|\sum_{\substack{T\in\mathbb T\\s_T=s}}f_T\right\|_p^p
 \lesssim\sum_{\substack{T\in\mathbb T\\s_T=s}}\|f_T\|_p^p.
 \end{aligned}
\end{equation}
There are at most $3d^{-1}$ distinct parameters, so
H\"older's inequality gives
\begin{equation}\label{app:refined-trivial-bound}
 \begin{aligned}
 \left\|\sum_{T\in\mathbb T}f_T\right\|_{L^p(Y)}^p
 &\lesssim d^{-(p-1)}
       \sum_{s\in\{s_T:T\in\mathbb T\}}\left\|\sum_{\substack{T\in\mathbb T\\s_T=s}}f_T\right\|_p^p
 \lesssim d^{-(p-1)}\sum_{T\in\mathbb T}\|f_T\|_p^p.
 \end{aligned}
\end{equation}
Thus
\begin{equation}\label{app:refined-base-case}
 \text{Dec}(\mathcal P)\lesssim d_*^{-1+1/p}\lesssim1
 \qquad\text{if }d(\mathcal P)\geq d_*.
\end{equation}
This includes the terminal parameter tuple $(J+1;)$, for which $d=1$. In the remaining part, we consider $0<d<d_*$.

\subsection*{Step 2. Localization and decoupling}
Fix a parameter tuple $\mathcal P=(k;d_k,\ldots,d_J)$ with $1\leq k\leq J$
and $0<d=d(\mathcal P)<d_*$.
Let $v$, $\{f_T\}_{T\in\mathbb T}$, $Y$, and $M$ satisfy the assumptions
in \eqref{app:refined-dec-definition}, with widths
\eqref{c13all:rp-thin-class}, support
\eqref{app:refined-induction-support}, and overlap
\eqref{app:refined-induction-overlap}.
Use the parameters \eqref{app:refined-small-parameters}.
Suppose $d_k<d^\kappa$, and put
\begin{equation}\label{app:refined-localization-scales}
 h:=d^\kappa,\qquad d':=d^{1-\kappa},\qquad
 \mathcal P':=(k;d_k/h,d_{k+1},\ldots,d_J).
\end{equation}
Let $\mathcal I_h$ partition $[-1,1]$ into intervals of length
$h$, with centers $s_I$. For $I\in\mathcal I_h$, let $A_I$
be the map in Lemma~\ref{app:refined-rescaling-lemma}.
Fix a tiling $\mathcal B$ of the original physical space by cubes $B$ of side length
$C_1d^{-a}\max_{0\leq j\leq n}r_j^{-1}$, with $C_1$ fixed and
sufficiently large. The derivative frame and its inverse are uniformly
bounded, so
\begin{equation}
 \operatorname{diam}(T)\lesssim\max_{0\leq j\leq n}r_j^{-1},
 \qquad \operatorname{diam}(10d^{-a}T)\leq\operatorname{diam}(B).
\end{equation}
Thus we can assign each box $T$ to exactly one cube $B$ such that
$10d^{-a}T\subset3B$. Fix these assignments and set
\begin{equation}
 \begin{aligned}
 \mathbb T_{I,B}&:=\{T\in\mathbb T:s_T\in I,\ T\text{ is assigned to }B\},\\
 F_{I,B}&:=\sum_{T\in\mathbb T_{I,B}}f_T,
 \qquad I\in\mathcal I_h.
 \end{aligned}
\end{equation}

We set up some notation for Lemma~\ref{app:refined-localization-lemma} whose proof is analogous
to that of Guth--Iosevich--Ou--Wang~\cite[proof of Theorem~4.2]{GIOW2020}.
In particular, the dyadic choices of $M'$ and $M''$ below play similar
roles to the corresponding parameters there.

Put $L:=h^{-2n}$, $L':=(d')^{-2n\kappa}$ and
$b_a:=5+5\cdot2^{-a}$. Let $\mathcal Q_{L'}$ be a tiling of the
rescaled space by cubes of side $L'$. For a dyadic number $M'\geq1$,
define
\begin{equation}\label{app:refined-first-multiplicity}
 E_{I,B}:=\bigcup_{\substack{q\in\mathcal Q_{L'}:\ M'\leq
 \#\{T\in\mathbb T_{I,B}:\ b_a d^{-a}T\cap A_I^T(20q)\ne\varnothing\}<2M'}}q,
\end{equation}
Let $Y_{I,B}$ be the union of the unit grid cubes meeting $4q$
for some selected cube $q\subset E_{I,B}$.

For this fixed choice of $M'$, let $\mathcal Q_L$ be a tiling of the
original space by cubes of side $L$. For $Q\in\mathcal Q_L$ and a
fixed sufficiently large $C_3$, define
\begin{equation}\label{app:refined-second-multiplicity}
 \begin{gathered}
 \mathcal B_I(Q):=\{B\in\mathcal B:A_I^TE_{I,B}\cap C_3Q\ne\varnothing\},\qquad I\in\mathcal I_h,\\
 M''(Q):=\#\{I\in\mathcal I_h:\mathcal B_I(Q)\ne\varnothing\}.
 \end{gathered}
\end{equation}
Since $A_I^TE_{I,B}\subset CB$ and $L\ll\operatorname{diam}(B)$,
after a bounded partition of $\mathcal B$, which we suppress, we may assume
$\#\mathcal B_I(Q)\in\{0,1\}$.
Thus $M''(Q)$ counts the intervals $I$ for which $A_I^TE_{I,B}$ meets
$C_3Q$ for some $B$.
We select the cubes $Q$ with
$Q\cap Y\ne\varnothing$ and $M''\leq M''(Q)<2M''$ for a dyadic
number $M''\geq1$.

\begin{lemma}\label{app:refined-localization-lemma}
With the notation and assumptions above, for every $\eta>0$ one can
choose dyadic numbers $M'\geq1$ and then $M''\geq1$ so that
\begin{equation}\label{c13all:rp-recurrence}
 \begin{aligned}
 \left\|\sum_{T\in\mathbb T}f_T\right\|_{L^p(Y)}
 &\lesssim (\log(d^{-1}))^{C_n}h^{-\eta}(M'')^{\frac12-\frac1p}
 \left(\sum_{I\in\mathcal I_h}\sum_{B\in\mathcal B}
       \|F_{I,B}\|_{L^p(A_I^TY_{I,B})}^p\right)^{1/p}\\
 &\quad+d\left(\sum_{T\in\mathbb T}\|f_T\|_p^p\right)^{1/p}.
 \end{aligned}
\end{equation}
\end{lemma}

\begin{proof}
We remark that \eqref{c13all:rp-recurrence} is analogous to the estimate ``$\|f\|_{L^p(Y)}^p\lessapprox (M'')^{\frac{p}{2}-1}\sum_{\square \in \mathbb B} \|f_\square\|_{L^p(Y_\square)}^p$" in \cite[Proof of Theorem 4.2]{GIOW2020}.

We first check the support condition under new coordinates. Take $s_T\in I$ and write
\begin{equation}\label{app:refined-center-coordinates}
 \xi=\sum_{i=0}^n\lambda_i v^{(i)}(s_T)
     =\sum_{j=0}^n\mu_jv^{(j)}(s_I).
\end{equation}
Using $r_i\leq1$ for $0\leq i<k$ and
$r_i\leq d_k^{i+1-k}\leq h^{i+1-k}$ for $k\leq i\leq n$,
Lemma~\ref{prelim:change-of-frame} gives
\begin{equation}
 \begin{gathered}
 |\mu_j|\lesssim\sum_{i=0}^j r_i h^{j-i}
              +h\sum_{i=j+1}^n r_i
 \lesssim
 \begin{cases}
  1,&0\leq j<k,\\
  h^{j+1-k},&k\leq j\leq n,
 \end{cases}\\
 |\mu_{k-1}-\lambda_{k-1}|\lesssim h.
 \end{gathered}
\end{equation}
The assumptions $d_k<d^\kappa$ and $d<d_*$, together with
\eqref{app:refined-localization-scales} and \eqref{app:refined-small-parameters},
give $d_k<h=d^\kappa<d_*^\kappa<\tau$.
Thus \eqref{app:refined-induction-support} gives $\lambda_{k-1}\gtrsim1$. Hence
\begin{equation}
 \mu_{k-1}\geq\lambda_{k-1}-Ch\gtrsim1,
\end{equation}
after decreasing $d_*$. Thus \eqref{app:refined-coarse-support}
holds, and its exponent range contains $p$ since
\begin{equation}
 p\leq(n+1-J)(n+2-J)\leq(n+1-k)(n+2-k).
\end{equation}

The assignments give
\begin{equation}\label{app:refined-family-partition}
 \begin{gathered}
 \mathbb T=\bigsqcup_{I\in\mathcal I_h}\bigsqcup_{B\in\mathcal B}\mathbb T_{I,B},
 \qquad \sum_{T\in\mathbb T}f_T=\sum_{I\in\mathcal I_h}\sum_{B\in\mathcal B} F_{I,B},\\
 \sum_{I\in\mathcal I_h}\sum_{B\in\mathcal B}\sum_{T\in\mathbb T_{I,B}}\|f_T\|_p^p
 =\sum_{T\in\mathbb T}\|f_T\|_p^p.
 \end{gathered}
\end{equation}
Since the cubes $B$ have a common side length and
$L=d^{-2n\kappa}\ll d^{-a}\leq\operatorname{diam}(B)$,
each cube of side $C_3L$ meets only $O(1)$ fixed enlargements $CB$.

The cubes in $\mathcal Q_{L'}$ have side $L'$. We smooth their indicators
at the smaller scale $\sqrt{L'}=(d')^{-n\kappa}$.
For a nonnegative Schwartz function $\omega$ with integral one and Fourier
support in a sufficiently small ball, set
\begin{equation}\label{app:refined-smooth-cutoffs}
 \omega_{\sqrt{L'}}(y):=(L')^{-(n+1)/2}\omega(y/\sqrt{L'}),\qquad
 \eta_q:=\mathbf1_q*\omega_{\sqrt{L'}},\quad q\in\mathcal Q_{L'}.
\end{equation}
By \eqref{app:refined-smooth-cutoffs}, we have
\begin{equation}\label{app:refined-cutoff-support}
 0\leq\eta_q\leq1,\qquad
 \sum_{q\in\mathcal Q_{L'}}\eta_q=1,\qquad
 \operatorname{supp}\mathcal F
       \left(\eta_q\right)\subset B(0,C/\sqrt{L'}),
\end{equation}
for every $q\in\mathcal Q_{L'}$.
By the definition of $A_I$,
$\|A_I^T\|=\|A_I\|\lesssim h^{-(n+1-k)}$, so
\begin{equation}\label{c13all:rp-cube-sizes}
 \operatorname{diam}(A_I^Tq)
 \lesssim h^{-(n+1-k)}L'
 =d^{-\kappa(n+1-k)-2n\kappa(1-\kappa)}
 \lesssim d^{-3n\kappa}=o(d^{-a}).
\end{equation}
Recall from \eqref{app:refined-center-coordinates} that
$\xi=\sum_{j=0}^n\mu_jv^{(j)}(s_I)$: the coefficients $\mu_j$ are taken
in the derivative frame at the center $s_I$ of $I$.
For $\mathcal Q\subset\mathcal Q_{L'}$, multiplication by
$\sum_{q\in\mathcal Q}\eta_q(A_I^{-T}x)$ changes $\mu_j$ by at most
\begin{equation}
 C h^{j+1-k}(L')^{-1/2}
 =C h^{j+1-k}(d')^{n\kappa}\quad(0\leq j\leq n).
\end{equation}
For $k-1\leq j\leq n$ this is $o(h^{j+1-k})$, while for $0\leq j<k-1$ it is
\begin{equation}
 C d^{\kappa[n(1-\kappa)-(k-1-j)]}=o(1),
 \qquad n(1-\kappa)>n-1\geq k-1-j.
\end{equation}
These products satisfy \eqref{app:refined-coarse-support}.
Induction will be applied only to the original $F_{I,B}$ on
$A_I^TY_{I,B}$, not to these products.

The errors from the tails are controlled by the error term in
\eqref{c13all:rp-recurrence}, so we omit them below.

By \eqref{c13all:rp-cube-sizes} and bounded overlap at each angular
parameter,
\begin{equation}
 \#\{T\in\mathbb T_{I,B}:
        b_a d^{-a}T\cap A_I^T(20q)\ne\varnothing\}
 \lesssim d^{-1}(1+d^{-a})^{n+1}
 \lesssim d^{-1-(n+1)a}.
\end{equation}
Pigeonholing among $O(\log(d^{-1}))$ positive dyadic values of $M'$,
take $E_{I,B}$ from \eqref{app:refined-first-multiplicity} and set
\begin{equation}
 \eta_{I,B}(x):=(\mathbf1_{E_{I,B}}*\omega_{\sqrt{L'}})(A_I^{-T}x).
\end{equation}
Since $E_{I,B}$ is a union of cubes $q\in\mathcal Q_{L'}$, the Fourier-support
bound in \eqref{app:refined-cutoff-support} applies to its selected subcollection.
Since $\sqrt{L'}\ll L'$, we have, for every $A>0$ and $y\notin Y_{I,B}$,
\begin{equation}\label{app:refined-cutoff-tail}
 \begin{aligned}
 0\leq\eta_{I,B}(A_I^Ty)
 &\leq\int_{|z|\geq L'}\omega_{\sqrt{L'}}(z)\,\mathrm{d}z
 \lesssim_A(L')^{-A/2}
 =(d')^{n\kappa A}.
 \end{aligned}
\end{equation}
Indeed, $y\notin Y_{I,B}$ implies $y\notin4q$ for every selected cube $q$,
so $|y-z|\geq L'$ for all $z\in E_{I,B}$.

For this choice of $M'$, \eqref{c13all:rp-cube-sizes} and the assignment
$10d^{-a}T\subset3B$ imply
\begin{equation}
 A_I^TE_{I,B}\subset CB,\qquad
 \#\mathcal B_I(Q)\leq1,\qquad
 M''(Q)\leq\#\mathcal I_h\lesssim h^{-1}\leq d^{-1}.
\end{equation}
On each $Q$, retain for each $I$ the unique $B\in\mathcal B_I(Q)$ when it exists.
Pigeonholing the cubes $Q$ meeting $Y$ by $M''(Q)$ selects $M''\geq1$ such that
\begin{equation}\label{app:refined-pigeonhole}
 \begin{aligned}
 \left\|\sum_{T\in\mathbb T}f_T\right\|_{L^p(Y)}^p
 &\lesssim(\log(d^{-1}))^{C_n}
 \sum_{\substack{Q\in\mathcal Q_L:\ Q\cap Y\ne\varnothing\\
                         M''\leq M''(Q)<2M''}}
 \left\|\sum_{I\in\mathcal I_h}
          \sum_{B\in\mathcal B_I(Q)}\eta_{I,B}F_{I,B}\right\|_{L^p(Q)}^p.
 \end{aligned}
\end{equation}

For a selected cube $Q$, with $W_Q$ from Section~\ref{sec:notations},
choose $\chi_Q$ such that
\begin{equation}
 |\chi_Q|\geq1\text{ on }Q,\qquad |\chi_Q|^p\lesssim W_Q,\qquad
 \operatorname{supp}\widehat\chi_Q\subset B(0,c/L),
\end{equation}
where $c>0$ is sufficiently small.

Put
\begin{equation}
 S_I:=\sum_{B\in\mathcal B_I(Q)}\eta_{I,B}F_{I,B},\qquad
 \#\{I\in\mathcal I_h:S_I\ne0\}\leq M''(Q)<2M''.
\end{equation}
Multiplication by $\chi_Q$ changes every derivative coefficient by
$O(L^{-1})=O(h^{2n})$. Since
$h^{2n}=o(h^{n+1-k})$, the functions $\chi_QS_I$ still satisfy
\eqref{app:refined-coarse-support}. Ordinary decoupling and H\"older
give
\begin{equation}\label{c13all:rp-local}
 \begin{aligned}
 \left\|\sum_{I\in\mathcal I_h} S_I\right\|_{L^p(Q)}
 &\leq\left\|\sum_{I\in\mathcal I_h}\chi_QS_I\right\|_p
 \lesssim h^{-\eta}\left(\sum_{I\in\mathcal I_h}\|\chi_QS_I\|_p^2\right)^{1/2}\\
 &\leq h^{-\eta}(2M'')^{1/2-1/p}
       \left(\sum_{I\in\mathcal I_h}\|\chi_QS_I\|_p^p\right)^{1/p}\\
 &\lesssim h^{-\eta}(M'')^{1/2-1/p}
       \left(\sum_{I\in\mathcal I_h}\sum_{B\in\mathcal B_I(Q)}
                  \|\eta_{I,B}F_{I,B}\|_{L^p(W_Q)}^p\right)^{1/p}.
 \end{aligned}
\end{equation}
Each inner sum over $B\in\mathcal B_I(Q)$ has at most one term.

Using $0\leq\eta_{I,B}\leq1$, $\sum_{Q\in\mathcal Q_L}W_Q\lesssim1$, and
\eqref{app:refined-cutoff-tail} to localize to $A_I^TY_{I,B}$, we obtain
\begin{equation}
 \begin{aligned}
 &\sum_{Q\in\mathcal Q_L}\sum_{I\in\mathcal I_h}\sum_{B\in\mathcal B_I(Q)}
              \|\eta_{I,B}F_{I,B}\|_{L^p(W_Q)}^p\\
 &\quad\lesssim\sum_{I\in\mathcal I_h}\sum_{B\in\mathcal B}\int|\eta_{I,B}(x)F_{I,B}(x)|^p\,\mathrm{d}x\lesssim\sum_{I\in\mathcal I_h}\sum_{B\in\mathcal B}\|F_{I,B}\|_{L^p(A_I^TY_{I,B})}^p.
 \end{aligned}
\end{equation}
Insert this bound and \eqref{c13all:rp-local} into
\eqref{app:refined-pigeonhole}, and take the $p$th root to obtain
\eqref{c13all:rp-recurrence}.
\end{proof}

\subsection*{Step 3. The two overlap bounds}
\begin{lemma}[The overlap bounds]\label{app:refined-overlap-lemma}
Use the assignments from Step~2 and the dyadic numbers supplied by
Lemma~\ref{app:refined-localization-lemma}, with the sets defined in
\eqref{app:refined-first-multiplicity}--\eqref{app:refined-second-multiplicity}.
Recall that $d'=d/h$, $r=J-k+1$, and $L_r$ is defined in
\eqref{app:refined-enlargements}; $M$ is the original overlap bound
in \eqref{app:refined-induction-overlap}.
For a nonempty selected class in \eqref{app:refined-pigeonhole}, one has
\begin{equation}
 M'M''\leq M.
\end{equation}
Moreover, for every $I\in\mathcal I_h$ and cube $B$, each constituent
unit cube $U'$ of $Y_{I,B}$ satisfies
\begin{equation}
 \#\{T\in\mathbb T_{I,B}:
 10(d')^{-a}A_I^{-T}T\cap L_rU'\ne\varnothing\}<2M'.
\end{equation}
\end{lemma}

\begin{proof}
Choose a selected cube $Q$ in \eqref{app:refined-pigeonhole} and a
constituent unit cube $U_Q\subset Y$ meeting $Q$.
For each $I$ counted by $M''(Q)$, let $B$ be the unique element of
$\mathcal B_I(Q)$. Choose a selected cube $q_{I,B}\subset E_{I,B}$ satisfying
$A_I^Tq_{I,B}\cap C_3Q\ne\varnothing$.
The definition of $M'$ gives
\begin{equation}
 \#\{T\in\mathbb T_{I,B}:
 b_a d^{-a}T\cap A_I^T(20q_{I,B})\ne\varnothing\}\geq M'.
\end{equation}
For $z$ in the displayed intersection and $y\in U_Q$,
\begin{equation}
 |y-z|\lesssim L+\operatorname{diam}(A_I^Tq_{I,B})+1
 \lesssim d^{-3n\kappa}
\end{equation}
by \eqref{c13all:rp-cube-sizes}. Since
$10-b_a=5(1-2^{-a})>0$, $3n\kappa<a$, and $r_j\leq1$, we may choose
$d_*$ so that
\begin{equation}
 \begin{aligned}
 |(y-x_T)\cdot v^{(j)}(s_T)|
 &\leq |(z-x_T)\cdot v^{(j)}(s_T)|
                  +|y-z|\,|v^{(j)}(s_T)|\\
 &\leq b_a d^{-a}r_j^{-1}+C d^{-3n\kappa}
 \leq10d^{-a}r_j^{-1}
 \qquad(0\leq j\leq n).
 \end{aligned}
\end{equation}
Thus $U_Q\subset10d^{-a}T$. Disjointness in
\eqref{app:refined-family-partition} gives
\begin{equation}\label{c13all:rp-incidence}
 \begin{aligned}
 M'M''
 &\leq M'M''(Q)=M'\sum_{I\in\mathcal I_h}\#\mathcal B_I(Q)\\
 &\leq\sum_{I\in\mathcal I_h}\sum_{B\in\mathcal B_I(Q)}
 \#\{T\in\mathbb T_{I,B}:
 b_a d^{-a}T\cap A_I^T(20q_{I,B})\ne\varnothing\}\\
 &\leq\#\{T\in\mathbb T:10d^{-a}T\cap U_Q\ne\varnothing\}\\
 &\leq\#\{T\in\mathbb T:10d^{-a}T\cap L_rU_Q\ne\varnothing\}
 \leq M.
 \end{aligned}
\end{equation}

For a constituent unit cube $U'$ of $Y_{I,B}$, choose a selected cube
$q\subset E_{I,B}$ with $U'\cap4q\ne\varnothing$.
Choosing $d_*$ so that $L'\geq L_r+1$ gives
\begin{equation}
 L_rU'\subset20q.
\end{equation}
Moreover, $h<\tau<1/2$ and $d'=d/h$, so
\begin{equation}
 10(d')^{-a}=10h^a d^{-a}
 \leq10\cdot2^{-a}d^{-a}<b_a d^{-a}.
\end{equation}
Applying $A_I^T$ to the rescaled intersections therefore gives
\begin{equation}
 \begin{aligned}
 &\#\{T\in\mathbb T_{I,B}:
       10(d')^{-a}A_I^{-T}T\cap L_rU'\ne\varnothing\}\\
 &\quad\leq
 \#\{T\in\mathbb T_{I,B}:
       b_a d^{-a}T\cap A_I^T(20q)\ne\varnothing\}
 <2M'.
 \end{aligned}
\end{equation}
The last inequality is \eqref{app:refined-first-multiplicity}.
\end{proof}

\subsection*{Step 4. The smaller-scale induction estimate}
\begin{lemma}[The smaller-scale estimate]\label{app:refined-recursion-lemma}
Let the parameter tuple $\mathcal P=(k;d_k,\ldots,d_J)$ and its angular
scale $d=d(\mathcal P)$ be as in \eqref{app:refined-dec-definition},
with $1\leq k\leq J$, $0<d<d_*$, and the parameters
\eqref{app:refined-small-parameters}. If $d_k<d^\kappa$, put
\begin{equation}
 h:=d^\kappa,\qquad
 \mathcal P':=(k;d_k/h,d_{k+1},\ldots,d_J).
\end{equation}
For every $\eta>0$,
\begin{equation}\label{exp:refined-constant-recursion}
 \begin{gathered}
 \text{Dec}(\mathcal P)+1
 \lesssim (\log(d^{-1}))^{C_n}h^{-\eta}
                  (\text{Dec}(\mathcal P')+1),\\
 d(\mathcal P')=d^{1-\kappa}.
 \end{gathered}
\end{equation}
\end{lemma}

\begin{proof}
Apply Lemma~\ref{app:refined-localization-lemma} to a family in
\eqref{app:refined-dec-definition}. If no main term remains, the error
term gives \eqref{exp:refined-constant-recursion}; hence assume the
selected class is nonempty.

Fix $I\in\mathcal I_h$ and a cube $B$ from the fixed tiling. For
$T\in\mathbb T_{I,B}$, write
\begin{equation}\label{app:refined-rescaled-family}
 \begin{gathered}
 T':=A_I^{-T}T,\qquad f_{T'}'(y):=f_T(A_I^Ty),\qquad
 u_T:=\frac{s_T-s_I}{h},\\
 r_j'=h^{k-1-j}r_j\quad(0\leq j\leq n),\qquad
 d'=\frac{d}{h}=d^{1-\kappa}.
 \end{gathered}
\end{equation}
Lemma~\ref{app:refined-rescaling-lemma}, with
\eqref{c13all:rp-rescale} and \eqref{app:refined-new-scales}, gives
the uniform curve $v_I$, widths \eqref{c13all:rp-thin-class},
$d'$-separation, and support \eqref{app:refined-induction-support}
at $\mathcal P'$. Here $d_k'=d_k/h$ and
$d_\ell'=d_\ell$ for $k<\ell\leq J$.
A common invertible linear map preserves the tiling and
its overlap at each fixed parameter. Using
$W_T(A_I^Ty)\lesssim W_{T'}(y)$, we obtain
\begin{equation}\label{app:refined-wave-packet-rescaling}
 \begin{gathered}
 |T'|=|\det A_I|^{-1}|T|,\qquad
 \|f_{T'}'\|_p=|\det A_I|^{-1/p}\|f_T\|_p,\\
 \begin{aligned}
 |f_{T'}'(y)|&\lesssim |T|^{-1/p}\|f_T\|_pW_{T'}(y)\\
 &=|T'|^{-1/p}\|f_{T'}'\|_pW_{T'}(y).
 \end{aligned}
 \end{gathered}
\end{equation}
For every constituent unit cube $U'$ of $Y_{I,B}$,
Lemma~\ref{app:refined-overlap-lemma} gives
\begin{equation}
 \#\{T\in\mathbb T_{I,B}:
 10(d')^{-a}T'\cap L_rU'\ne\varnothing\}<2M',
 \qquad r=J-k+1.
\end{equation}
Since $k$ is unchanged, this is
\eqref{app:refined-induction-overlap} at $\mathcal P'$.

Applying \eqref{app:refined-dec-definition} and
$\mathrm{d}x=|\det A_I|\,\mathrm{d}y$ gives
\begin{equation}\label{app:refined-inductive-estimate}
 \begin{aligned}
 \|F_{I,B}\|_{L^p(A_I^TY_{I,B})}^p
 &=|\det A_I|
   \left\|\sum_{T\in\mathbb T_{I,B}}f_{T'}'\right\|_{L^p(Y_{I,B})}^p\\
 &\leq|\det A_I|\text{Dec}(\mathcal P')^p
       (2M')^{p/2-1}
       \sum_{T\in\mathbb T_{I,B}}\|f_{T'}'\|_p^p\\
 &=\text{Dec}(\mathcal P')^p(2M')^{p/2-1}
       \sum_{T\in\mathbb T_{I,B}}\|f_T\|_p^p.
 \end{aligned}
\end{equation}
We remark that the above estimate is analogous to \cite[(4.16)]{GIOW2020}.

Insert \eqref{app:refined-inductive-estimate} into
\eqref{c13all:rp-recurrence}, using \eqref{app:refined-family-partition}:
\begin{equation}
 \begin{aligned}
 \left\|\sum_{T\in\mathbb T}f_T\right\|_{L^p(Y)}^p
 &\lesssim (\log(d^{-1}))^{C_n}h^{-p\eta}
       \text{Dec}(\mathcal P')^p(M'M'')^{p/2-1}\\
 &\qquad\times\sum_{I\in\mathcal I_h}\sum_{B\in\mathcal B}
       \sum_{T\in\mathbb T_{I,B}}\|f_T\|_p^p
       +C d^p\sum_{T\in\mathbb T}\|f_T\|_p^p\\
 &\lesssim (\log(d^{-1}))^{C_n}h^{-p\eta}
       \text{Dec}(\mathcal P')^p M^{p/2-1}\\
 &\qquad\times\sum_{T\in\mathbb T}\|f_T\|_p^p
       +C d^p\sum_{T\in\mathbb T}\|f_T\|_p^p.
 \end{aligned}
\end{equation}
The second inequality uses \eqref{c13all:rp-incidence} and $p\geq2$.
Taking the $p$th root and the supremum in
\eqref{app:refined-dec-definition} gives
\begin{equation}
 \text{Dec}(\mathcal P)
 \lesssim(\log(d^{-1}))^{C_n}h^{-\eta}
           \text{Dec}(\mathcal P')+C d.
\end{equation}
Since $d<d_*<e^{-1}$ and $h^{-\eta}\geq1$, adding $1$
proves \eqref{exp:refined-constant-recursion}.
\end{proof}

\subsection*{Step 5. Passing to the next parameter}
\begin{lemma}[The next parameter]\label{app:refined-transition-lemma}
Let the parameter tuple $\mathcal P=(k;d_k,\ldots,d_J)$ and its angular
scale $d=d(\mathcal P)$ be as in \eqref{app:refined-dec-definition},
with $1\leq k\leq J$, $0<d<d_*$, and the parameters
\eqref{app:refined-small-parameters}. If $d_k\geq d^\kappa$, put
$\mathcal P':=(k+1;d_{k+1},\ldots,d_J)$. Then
\begin{equation}\label{exp:refined-constant-transition}
 \text{Dec}(\mathcal P)
 \lesssim d^{-\kappa}\text{Dec}(\mathcal P'),
 \qquad d(\mathcal P')=d/d_k.
\end{equation}
\end{lemma}

\begin{proof}
Fix an admissible family at $\mathcal P$, and put $r=J-k+1$.
If $d_k=1$, then $r_{k-1}=r_k=1$ and $L_{r-1}\leq L_r$;
the same family is admissible at $\mathcal P'$, so
$\text{Dec}(\mathcal P)\leq\text{Dec}(\mathcal P')$.

Suppose $d^\kappa\leq d_k<1$, and put $h:=d_k$.
Use the partition $\mathcal I_h$ of $[-1,1]$ at scale $h$
from Step~1, and write
\begin{equation}
 F_I:=\sum_{\substack{T\in\mathbb T\\s_T\in I}}f_T,
 \qquad \#\mathcal I_h\lesssim h^{-1}\leq d^{-\kappa}.
\end{equation}
The triangle inequality and H\"older's inequality give
\begin{equation}\label{app:refined-transition-sum}
 \begin{aligned}
 \left\|\sum_{T\in\mathbb T}f_T\right\|_{L^p(Y)}
 &\leq\sum_{I\in\mathcal I_h}\|F_I\|_{L^p(Y)}\\
 &\leq(\#\mathcal I_h)^{1-1/p}
       \left(\sum_{I\in\mathcal I_h}\|F_I\|_{L^p(Y)}^p\right)^{1/p}\\
 &\lesssim d^{-\kappa}
       \left(\sum_{I\in\mathcal I_h}\|F_I\|_{L^p(Y)}^p\right)^{1/p}.
 \end{aligned}
\end{equation}

For each $I$, Lemma~\ref{app:refined-rescaling-lemma} gives
\begin{equation}
 d_k'=1,\qquad r_{k-1}'=r_k'=1,\qquad
 d'=d/h=\prod_{\ell=k+1}^J d_\ell.
\end{equation}
We omit the unit parameter $d_k'=1$ and use the tuple $\mathcal P'$
for the rescaled family.
Lemma~\ref{app:refined-rescaling-lemma} and
\eqref{app:refined-wave-packet-rescaling} give the required Fourier
support, angular separation, and spatial wave packet bound at
$\mathcal P'$. The unit-cube overlap condition
\eqref{app:refined-induction-overlap} remains to be checked.

Let $Y'$ be the union of unit grid cubes meeting $A_I^{-T}Y$.
For each constituent unit cube $U'$ of $Y'$, choose a cube $U\subset Y$
with $A_I^TU'\cap U\ne\varnothing$. Then
\begin{equation}\label{app:refined-transition-cube-image}
 A_I^T(L_{r-1}U')
 \subset U+B(0,C L_{r-1}\|A_I^T\|).
\end{equation}
All dilations are about the centers. If $h<\tau<1/2$, then
$h\geq d^\kappa$ gives
\begin{equation}
 \|A_I^T\|\lesssim h^{-n}\leq d^{-n\kappa},\qquad
 10(d')^{-a}=10h^a d^{-a}
 \leq10\cdot2^{-a}d^{-a}.
\end{equation}
Since $a>n\kappa$, decreasing $d_*$ gives
\begin{equation}
 C L_{r-1}\|A_I^T\|
 \leq c\,10(1-2^{-a})d^{-a},
\end{equation}
where $B(0,c)\subset T-x_T$ uniformly, by $r_j\leq1$ and the
curve bounds. Convexity and symmetry give
\begin{equation}
 \begin{aligned}
 10(d')^{-a}T+B(0,C L_{r-1}\|A_I^T\|)
 &\subset x_T+
       \big(10h^a+10(1-2^{-a})\big)d^{-a}(T-x_T)\\
 &\subset10d^{-a}T.
 \end{aligned}
\end{equation}
For $\tau\leq h<1$, the map $A_I^T$ is uniformly bounded, so
$L_r\gg L_{r-1}$ in \eqref{app:refined-enlargements} gives
\begin{equation}
 A_I^T(L_{r-1}U')\subset L_rU,\qquad
 10(d')^{-a}T=10h^a d^{-a}T\subset10d^{-a}T.
\end{equation}
Together with \eqref{app:refined-transition-cube-image}, both cases give
\begin{equation}\label{app:refined-transition-overlap}
 \begin{aligned}
 &\#\{T\in\mathbb T:s_T\in I,\
       10(d')^{-a}A_I^{-T}T\cap L_{r-1}U'\ne\varnothing\}\\
 &\qquad\leq
 \#\{T\in\mathbb T:10d^{-a}T\cap L_rU\ne\varnothing\}
 \leq M.
 \end{aligned}
\end{equation}

By \eqref{app:refined-transition-overlap}, we may apply
\eqref{app:refined-dec-definition} at $\mathcal P'$.
Since $A_I^{-T}Y\subset Y'$, the determinant cancellation in
\eqref{app:refined-inductive-estimate} gives
\begin{equation}
 \|F_I\|_{L^p(Y)}^p
 \leq|\det A_I|\,\|F_I(A_I^T\,\cdot)\|_{L^p(Y')}^p
 \leq\text{Dec}(\mathcal P')^p M^{p/2-1}
 \sum_{\substack{T\in\mathbb T\\s_T\in I}}\|f_T\|_p^p.
\end{equation}
Substitute this in \eqref{app:refined-transition-sum}, sum over
the partition $\mathcal I_h$, and take the supremum in
\eqref{app:refined-dec-definition}. This gives
\begin{equation}
 \text{Dec}(\mathcal P)\lesssim
 d^{-\kappa}\text{Dec}(\mathcal P').
\end{equation}
\end{proof}

\subsection*{Step 6. Iteration}
\begin{lemma}[Iteration]\label{app:refined-iteration-lemma}
With the parameters in \eqref{app:refined-small-parameters}, every
parameter tuple $\mathcal P$ in \eqref{app:refined-dec-definition},
including the terminal parameter tuple $(J+1;)$, satisfies
\begin{equation}
 \text{Dec}(\mathcal P)\lesssim d(\mathcal P)^{-\varepsilon}.
\end{equation}
\end{lemma}

\begin{proof}
By \eqref{app:refined-small-parameters}, choose $\eta,\varepsilon_1>0$
with
\begin{equation}
 \eta+J\kappa+\varepsilon_1<\varepsilon.
\end{equation}
Start from a parameter tuple $\mathcal P_0$, with $d:=d(\mathcal P_0)<d_*$;
the other case is \eqref{app:refined-base-case}.
For the current parameter tuple $\mathcal P_i=(k;d_k,\ldots,d_J)$, apply
Lemma~\ref{app:refined-recursion-lemma} if
$d_k<d(\mathcal P_i)^\kappa$ (first kind), and
Lemma~\ref{app:refined-transition-lemma} otherwise (second kind).
Stop when $d(\mathcal P_i)\geq d_*$ or $\mathcal P_i=(J+1;)$.
The scale changes are
\begin{equation}\label{app:refined-iteration-scales}
 d(\mathcal P_{i+1})=
 \begin{cases}
 d(\mathcal P_i)^{1-\kappa},&\text{first kind},\\
 d(\mathcal P_i)/d_k,&\text{second kind},
 \end{cases}
 \qquad
 d\leq d(\mathcal P_i)\leq d(\mathcal P_{i+1})\leq1.
\end{equation}

There are at most $J$ steps of the second kind, since each removes
one parameter. By \eqref{app:refined-iteration-scales}, after $m$
steps of the first kind,
\begin{equation}
 \log\big(d(\mathcal P_i)^{-1}\big)
 \leq(1-\kappa)^m\log(d^{-1}).
\end{equation}
Thus the total number of steps is
$O(1+\log\log(d^{-1}))$.
At the last step, \eqref{app:refined-base-case} gives
$\text{Dec}(\mathcal P_i)+1\lesssim1$; the terminal parameter tuple
$(J+1;)$ has angular scale $1$.

At a step of the first kind, put
$h_i:=d(\mathcal P_i)^\kappa
=d(\mathcal P_i)/d(\mathcal P_{i+1})$.
The monotonicity in \eqref{app:refined-iteration-scales} gives
\begin{equation}\label{app:refined-iteration-powers}
 \begin{gathered}
 \prod_{\substack{i:\text{step }i\text{ of first kind}}}h_i^{-\eta}
 =\left(
 \prod_{\substack{i:\text{step }i\text{ of first kind}}}
 \frac{d(\mathcal P_{i+1})}{d(\mathcal P_i)}
 \right)^\eta
 \leq\left(
 \prod_{\substack{i:\text{all steps}}}
 \frac{d(\mathcal P_{i+1})}{d(\mathcal P_i)}
 \right)^\eta
 \leq d^{-\eta},\\
 \prod_{\substack{i:\text{step }i\text{ of second kind}}}
       d(\mathcal P_i)^{-\kappa}
 \leq\prod_{\substack{i:\text{step }i\text{ of second kind}}}
       d^{-\kappa}
 \leq d^{-J\kappa}.
 \end{gathered}
\end{equation}

Let $C_\eta$ bound the constants in both recursive estimates.
Since $d\leq d(\mathcal P_i)<d_*<e^{-1}$ at every
step, the remaining constants and logarithmic factors satisfy
\begin{equation}\label{app:refined-iteration-logs}
 \begin{aligned}
 \prod_{\substack{i:\text{all steps}}}
 C_\eta\big(\log(d(\mathcal P_i)^{-1})\big)^{C_n}
 &\leq\left[C_\eta\big(\log(d^{-1})\big)^{C_n}\right]^%
 {C_{\kappa,d_*}(1+\log\log(d^{-1}))}\\
 &\leq\exp\!\left(C_{n,\kappa,d_*,\eta}
 (1+\log\log(d^{-1}))^2\right)
 \lesssim d^{-\varepsilon_1}.
 \end{aligned}
\end{equation}
The last inequality follows from
$(1+\log\log(d^{-1}))^2=o(\log(d^{-1}))$ as $d\to0$.
Since $d(\mathcal P_i)^{-\kappa}\geq1$,
\eqref{exp:refined-constant-transition} also holds with $\text{Dec}$
replaced by $\text{Dec}+1$.
Iterating \eqref{exp:refined-constant-recursion} and
\eqref{exp:refined-constant-transition} in this form, then using the terminal
bound \eqref{app:refined-base-case},
\eqref{app:refined-iteration-powers}, and
\eqref{app:refined-iteration-logs}, gives
\begin{equation}
 \text{Dec}(\mathcal P_0)+1
 \lesssim d^{-\eta}d^{-J\kappa}d^{-\varepsilon_1}
 =d^{-(\eta+J\kappa+\varepsilon_1)}
 \leq d^{-\varepsilon}.
\end{equation}
\end{proof}

\begin{proof}[Proof of Proposition~\ref{c13all:rp}]
For the initial parameter tuple $\mathcal P_0=(1;\delta_1,\ldots,\delta_J)$,
take $d_j=\delta_j$. Then
\eqref{c13all:rp-thin-class} and
\eqref{app:refined-initial-scales} give
\begin{equation}
 r_j=\rho_j\quad(0\leq j\leq n),\qquad
 d(\mathcal P_0)=\prod_{j=1}^J\delta_j=\delta.
\end{equation}
The support condition is \eqref{app:refined-signed-change}. Also,
\eqref{app:refined-box-comparison} gives
\begin{equation}
 \begin{aligned}
 &\#\{T^{(0)}:10\delta^{-a}T^{(0)}\cap L_JU\ne\varnothing\}\\
 &\qquad\leq
 \#\{T:10\delta^{-a_0}T\cap L_JU\ne\varnothing\}\leq M
 \end{aligned}
\end{equation}
for each constituent unit cube $U$ of $Y$. The initial family
therefore enters \eqref{app:refined-dec-definition}, and
Lemma~\ref{app:refined-iteration-lemma} gives
\begin{equation}
 \begin{aligned}
 \left\|\sum_{T\in\mathbb T}f_T\right\|_{L^p(Y)}
 &\leq\text{Dec}(\mathcal P_0)M^{1/2-1/p}
       \left(\sum_{T\in\mathbb T}\|f_T\|_p^p\right)^{1/p}\\
 &\lesssim\delta^{-\varepsilon}M^{1/2-1/p}
       \left(\sum_{T\in\mathbb T}\|f_T\|_p^p\right)^{1/p}.
 \end{aligned}
\end{equation}
This proves \eqref{c13all:rp-estimate} for the fixed exponent $a_0=2a$.
\end{proof}

\bibliographystyle{alpha}
\bibliography{reference}

\end{document}